\documentclass{amsart}
\usepackage[utf8]{inputenc}
\usepackage[top=2.5cm,left=2cm,right=2cm,bottom=3cm]{geometry}
\usepackage[toc,page]{appendix}
\usepackage{amsmath, amssymb, amsthm, amstext}
\usepackage{enumitem}
\usepackage{stackengine} 
\usepackage{bm} 
\usepackage{stmaryrd} 
\usepackage{hyperref}
\usepackage{color}
\usepackage[british]{babel}
\usepackage{verbatim}
\usepackage{accents}
\usepackage{esint}
\usepackage{tabularx}
\usepackage[toc]{appendix}
\usepackage{amsfonts}\usepackage{amscd}\usepackage{url}\usepackage{color}\usepackage{amsopn}
\usepackage{bbm}
\usepackage{cite}\allowdisplaybreaks[1]
\usepackage{stmaryrd}

\makeatletter

\pdfpageheight\paperheight
\pdfpagewidth\paperwidth

\numberwithin{equation}{section}
\numberwithin{figure}{section}
\theoremstyle{plain}

\newcounter{cprop}[section]

\newtheorem{definition}[cprop]{Definition}
\newtheorem{remark}[cprop]{Remark}
\newtheorem{lemma}[cprop]{Lemma}
\newtheorem{proposition}[cprop]{Proposition}
\newtheorem{theorem}[cprop]{Theorem}

\newcommand{\tfy}{\tilde{y}} 

\newcommand{\mcA}{\mathcal{A}}
\newcommand{\mcB}{\mathcal{B}}

\newcommand{\mcH}{\mathcal H}

\newcommand{\mcL}{\mathcal{L}}

\newcommand{\mcP}{\mathcal{P}}

\newcommand{\mcV}{\mathcal{V}}

\newcommand{\mbD}{\mathbb{D}}

\newcommand{\mbH}{\mathbb{H}}
\newcommand{\mbL}{\mathbb{L}}

\newcommand{\R}{\mathbb{R}}

\newcommand{\N}{\mathbb{N}}

\newcommand{\D}{\Delta}

\newcommand{\restr}[1]{|_{#1}} 

\newcommand{\diff} {\mathrm{d}}

\newcommand{\esssup}{\text{ess sup}}
\DeclareMathOperator*{\supp}{supp}
\DeclareMathOperator*{\spann}{span}
\DeclareMathOperator*{\divv}{div}

\title{The Tamed MHD Equations}
\author{Andre Schenke}
\email{aschenke@math.uni-bielefeld.de}
\address{Fakult\"at f\"ur Mathematik, Universit\"at Bielefeld, 33615 Bielefeld, Germany}


\makeatother

\begin{document}

\begin{abstract}
	We study a regularised version of the magnetohydrodynamics (MHD) equations, the tamed MHD (TMHD) equations. They are a model for the flow of electrically conducting fluids through porous media. We prove existence and uniqueness of TMHD on the whole space $\mathbb{R}^{3}$, that smooth data give rise to smooth solutions, and show that solutions to TMHD converge to a suitable weak solution of the MHD equations as the taming parameter $N$ tends to infinity. Furthermore, we adapt a regularity result for the Navier-Stokes equations to the MHD case.
\end{abstract}
\maketitle
\section{Introduction}
\subsection{Magnetohydrodynamics}
The magnetohydrodynamics (MHD) equations describe the dynamic motion of electrically conducting fluids. They combine the equations of motion for fluids (Navier-Stokes equations) with the field equations of electromagnetic fields (Maxwell's equations), coupled via Ohm's law. In plasma physics, the equations are a macroscopic model for plasmas in that they deal with averaged quantities and assume the fluid to be a continuum with frequent collisions. Both approximations are not met in hot plasmas. Nonetheless, the MHD equations provide a good description of the low-frequency, long-wavelength dynamics of real plasmas. In this thesis, we consider the incompressible, viscous, resistive equations with homogeneous mass density, and regularised variants of it. In dimensionless formulation, the MHD equations are of the following form:
\begin{equation}\label{intro_eq_DMHD}
	\begin{split}
				\frac{\partial \bm{v} }{\partial t} &= \frac{1}{Re} \Delta \bm{v} - \left( \bm{v} \cdot \nabla \right) \bm{v}  + S \left( \bm{B} \cdot \nabla \right) \bm{B} + \nabla \left( p + \frac{S | \bm{B} |^2}{2} \right), \\
		\frac{\partial \bm{B}}{\partial t} &= \frac{1}{Rm} \Delta \bm{B} - \left( \bm{v} \cdot \nabla \right) \bm{B} + (\bm{B} \cdot \nabla) \bm{v} \\
		\divv \bm{v} &= 0, \quad \divv \bm{B} = 0.
	\end{split}
\end{equation}
Here, $\bm{v} = \bm{v}(x,t)$, $\bm{B} = \bm{B}(x,t)$ denote the velocity and magnetic fields, $p = p(x,t)$ is the pressure, $Re > 0$, $Rm > 0$ are the Reynolds number and the magnetic Reynolds number and $S >0$ denotes the Lundquist number (all of which are dimensionless constants). The two last equations concerning the divergence-freeness of the velocity and magnetic field are the incompressibility of the flow and Maxwell's second equation. For simplicity, in the remainder of the paper, we set $S= Rm = Re = 1$.

Mathematical treatment of the deterministic MHD equations reaches back to the works of G. Duvaut and J.-L. Lions \cite{DL72} and M. Sermange and R. Temam \cite{ST83}. Since then, a large amount of papers have been devoted to the subject. We only mention several interesting regularity criteria \cite{CW10,HX05a,HX05b,KL09} and the more recent work on non-resistive MHD equations ($Rm = \infty$) by C.L. Fefferman, D.S. McCormick J.C. Robinson and J.L. Rodrigo on local existence via higher-order commutator estimates \cite{FMRR14,FMRR17}.

In this paper, we want to study a regularised version of the MHD equations on the whole space $\R^{3}$, which we call the \emph{tamed MHD equations} (TMHD), following M. R\"ockner and X.C. Zhang \cite{RZ09a}. They arise from \eqref{intro_eq_DMHD} by adding two extra terms (the \emph{taming terms}) that act as restoring forces:
\begin{equation*}
	\begin{split}
		\frac{\partial \bm{v} }{\partial t} &=  \Delta \bm{v} - \left( \bm{v} \cdot \nabla \right) \bm{v}  +  S\left( \bm{B} \cdot \nabla \right) \bm{B} - \nabla \left( p + \frac{S | \bm{B} |^2}{2} \right) - g_{N}(| (\bm{v}, \bm{B}) |^{2}) \bm{v}, \\
		\frac{\partial \bm{B}}{\partial t}
		&= \Delta \bm{B} - \left( \bm{v} \cdot \nabla \right) \bm{B} + (\bm{B} \cdot \nabla) \bm{v} + \nabla \pi - g_{N}(| (\bm{v}, \bm{B}) |^{2}) \bm{B}.
	\end{split}
\end{equation*}
The taming terms are discussed in more detail in Section \ref{intro_ssec_tamed}, and we discuss the results of this work in Section \ref{intro_ssec_results}. The extra term $\nabla \pi$, which we call the magnetic pressure, will be explained in Section \ref{DTMHD_sssec_MPP}. However, before we study the tamed equations, we want to give an overview of regularisation schemes for the Navier-Stokes and the MHD equations to put our model into the broader context of the mathematical literature.

\subsubsection{Damped Navier-Stokes Equations (or BFeD Models)}\label{DTMHD_sssec_BFD}
A related model to the tamed Navier--Stokes equations are the so-called \emph{(nonlinearly) damped Navier-Stokes equations}:
\begin{align*}
	\frac{\partial \bm{v}}{\partial t} = \Delta \bm{v} - (\bm{v} \cdot \nabla)\bm{v} - \nabla p - \alpha |\bm{v}|^{\beta-1} \bm{v},
\end{align*}
with $\alpha > 0$ and $\beta \geq 1$. The damping term $- \alpha |\bm{v}|^{\beta-1} \bm{v}$ models the resistence to the motion of the flow resulting from physical effects like porous media flow, drag or friction or other dissipative mechanisms (cf. \cite{CJ08} and Section \ref{DTMHD_sssec_physical_motivation}). It represents a restoring force, which for $\beta = 1$ assumes the form of classical, linear damping, whereas $\beta > 1$ means a restoring force that grows superlinearly with the velocity (or magnetic field). X.J. Cai and Q.S. Jiu \cite{CJ08} first proved existence and uniqueness of a global strong solution for $\frac{7}{2} \leq \beta \leq 5$. This range was lowered down to $\beta \in (3,5]$ by Z.J. Zhang, X.L. Wu and M. Lu in \cite{ZWL11}. Furthermore, they considered the case $\beta = 3$ to be critical \cite[Remark 3.1]{ZWL11}. Y. Zhou in \cite{Zhou12} proved the existence of a global solution for all $\beta \in [3,5]$. For the case $\beta \in [1,3)$, he established regularity criteria that ensure smoothness. Uniqueness holds for any $\beta \geq 1$ in the class of weak solutions. Existence, decay rates and qualitative properties of weak solutions were also investigated by S.N. Antontsev and H.B. de Oliveira \cite{AO10}.

The Brinkman-Forchheimer-extended Darcy model (BFeD model, cf. Section \ref{DTMHD_sssec_physical_motivation}) is a related model for flow of fluids through porous media and uses the damping terms for $\beta \in \{ 1,2,3 \}$
\begin{align*}
	\frac{\partial \bm{v}}{\partial t} = \Delta \bm{v} - (\bm{v} \cdot \nabla)\bm{v} - \nabla p - \alpha_{0} \bm{v} - \alpha_{1} P |\bm{v}| \bm{v} -\alpha_{2} P |\bm{v}|^{2} \bm{v}.
\end{align*}
The first problems studied were continuous dependence of the solutions on their parameters, e.g. in F. Franchi, B. Straughan \cite{FS03}. V.K. Kalantarov and S. Zelik \cite{KZ12} and P.A. Markowich, E.S. Titi and S. Trabelsi \cite{MTT16} proved existence and uniqueness of a weak solution for Dirichlet and periodic boundary conditions, respectively. Long-time behaviour and existence of global attractors have been studied by several authors \cite{OY09, Ugurlu08, WL08, YZZ12}. An anisotropic version of the equations was studied by H. Bessaih, S. Trabelsi and H. Zorgati \cite{BTZ16}.

The flow of electrically conducting fluids through porous media, modelled by MHD equations with damping, was studied first by Z. Ye in \cite{Ye16}. He considered the system with nonlinear damping in the equations for both the velocity field (with nonlinear damping parameter $\alpha$) and the magnetic field (with paramter $\beta$) and he proved existence and uniqueness of global strong solutions in the full space case for several ranges of parameters, most interestingly for our purposes for $\alpha, \beta \geq 4$. Z.J. Zhang and X. Yang \cite{ZY16} tried to improve this to $\alpha, \beta > 3$, but apparently made a mistake in their proof (\cite[Remark after Equation (9), p. 2]{ZWY18}). Z.J. Zhang, C.P. Wu, Z.A. Yao \cite{ZWY18} then improved the range to $\alpha \in [3, \frac{27}{8}], \beta \geq 4$. The present paper, in a way, deals with the ``critical" case $\alpha = \beta = 3$, see the discussion of the results below. Furthermore, E.S. Titi and S. Trabelsi \cite{TT18} proved global well-posedness for an MHD model with nonlinear damping only in the velocity field. They thus avoid the magnetic pressure problem outlined in Section \ref{DTMHD_sssec_MPP}, as opposed to the above papers which seem to have overlooked this issue.

\subsection{The Tamed Equations}\label{intro_ssec_tamed}
We first motivate the tamed equations from a physical point of view by pointing out situations where similar models arise naturally in applications. The tamed Navier-Stokes equations are in a sense a variant of the Navier-Stokes equations with damping in the critical case $\beta = 3$, combined with a cutoff.
\subsubsection{Physical Motivation - Flows Through Porous Media}\label{DTMHD_sssec_physical_motivation}
Since the tamed equations are closely related to the damped equations of Section \ref{DTMHD_sssec_BFD}, which are much more well-studied, we focus on the occurence of these in the physics literature.
 
A system with possibly nonlinear damping is considered as a model for the flow of a fluid through porous media. To make this clear, consider the following 1-D compressible Euler equations with damping:
\begin{equation}
	\begin{split}
		\rho_{t} + \partial_{x} (\rho v) &= 0, \\
		(\rho v)_{t} + \partial_{x} \left( \rho v^{2} + p \right) &= - \alpha \rho v.
	\end{split}
\end{equation}
The interpretation that this equation models the flow through porous media is in line with the result that as $t \rightarrow \infty$, the density $\rho$ converges to the solution of the porous medium equation (cf. e.g. F.M. Huang, R.H. Pan \cite{HP03}). The momentum, on the other hand, is described in the limit by \emph{Darcy's law}:
	\begin{align*}
		\nabla p = - \frac{\mu}{k} \bm{v},
	\end{align*}
which represents a simple linear relationship between the flow rate and the pressure drop in a porous medium. Here, $k$ is the permeability of the porous medium and $\mu$ is the dynamic viscosity. The velocity $\bm{v}$ is called Darcy's seepage velocity.

In the interface region between a porous medium and a fluid layer, C.T. Hsu and P. Cheng \cite[Equation (31), p. 1591]{HC90} proposed the following equation\footnote{For ease of presentation, we have omitted various physical constants in the formulation of the equations.}:
\begin{align*}
	\divv \bm{v} &= 0, \\
	\partial_{t} \bm{v} + \divv (\bm{v} \otimes \bm{v}) &= - \nabla p + \nu \Delta \bm{v} - \alpha_{0} \bm{v} - \alpha_{1} |\bm{v}| \bm{v}, 
\end{align*}
where $\bm{v}$ is the so-called volume-averaged Darcy seepage velocity and $p$ is the volume-averaged pressure. This equation is motivated by a quadratic correction of P. Forchheimer to Darcy's law, called \emph{Forchheimer's law} or \emph{Darcy-Forchheimer law} (cf. for example P.A. Markowich, E.S. Titi, S. Trabelsi \cite{MTT16}):
\begin{align*}
	\nabla p = - \frac{\mu}{k} \bm{v}_{F} - \gamma \rho_{F} |\bm{v}_{F}| \bm{v}_{F},
\end{align*}
with the Forchheimer coefficient $\gamma > 0$, the Forchheimer velocity $\bm{v}_{F}$ as well as the density $\rho_{F}$. Furthermore, this correction becomes necessary at higher flow rates through porous media, see below for a more detailed discussion.

The question arises whether there are cases where a nonlinear correction of yet higher degree is necessary, i.e., where the flow obeys a \emph{cubic Forchheimer's law}:
\begin{equation}\label{DTMD_eq_cubicForchheimer}
	\nabla p = - \frac{\mu}{k} \bm{v} - \gamma \rho |\bm{v}| \bm{v} - \kappa \rho^{2} |\bm{v}|^{2} \bm{v}.
\end{equation}
Indeed, this seems to be the case. P. Forchheimer \cite{Forchheimer01} himself suggested several corrections to Darcy's law at higher flow velocities, one of them being the cubic law \eqref{DTMD_eq_cubicForchheimer}. M. Firdaouss, J.-L. Guermond and P. Le Qu\'{e}r\'{e} \cite{FGL97} revisited several historic data sets, amongst them the ones used by Darcy and Forchheimer (who did not correct for Reynolds numbers) and found that the data are actually better described by a linear and cubic Darcy-Forchheimer law (i.e., where $\gamma = 0$), in the regime of low to moderate Reynolds numbers, which, as they note \cite[p. 333]{FGL97}, includes most practical cases: 
\begin{equation}\label{DTMD_eq_cubiconlyForchheimer}
	\nabla p = - \frac{\mu}{k} \bm{v} - \kappa \rho^{2} |\bm{v}|^{2} \bm{v}.
\end{equation}
At higher Reynolds numbers, the correct behaviour seems to be quadratic, i.e., Forchheimer's law, in accordance with numerical simulations, e.g. in the work of M. Fourar \emph{et al.} \cite{FRLM04}. The point at which this behaviour changes seems to be dimension-dependent: it occurs much earlier in the numerical simulations of \cite[Figure 7]{FRLM04} in the 3D case than in the 2D case. Another instance where a cubic Forchheimer law is observed is the high-rate flow in a radial fracture with corrugated walls, cf. M. Bu\`{e}s, M. Panfilov, S. Crosnier and C. Oltean \cite[Equation (7.2), p. 54]{BPCO04}.

Taking into account all nonlinear corrections of Darcy's law, we arrive at the Brinkman-Forchheimer-extended Darcy (BFeD) model 
\begin{align*}
	\divv \bm{v} &= 0, \\
	\partial_{t} \bm{v} + \divv (\bm{v} \otimes \bm{v}) &= - \nabla p + \nu \Delta \bm{v} - \alpha_{0} \bm{v} - \alpha_{1} |\bm{v}| \bm{v} - \alpha_{2} |\bm{v}|^{2} \bm{v}.
\end{align*}
The tamed Navier-Stokes equations model the behaviour of the flow through porous media in the regime of relatively low to moderate Reynolds numbers, assuming that the higher-order behaviour is much more significant than the linear Darcy behaviour. For a more physically accurate model, one should also include the linear damping term, but we want to focus on the nonlinear effects here and thus for simplicity have omitted this term.
The fact that the onset of nonlinear behaviour occurs at higher flow rates is modelled by the cutoff function $g_{N}$ which is nonzero only for sufficiently high velocity.

\subsubsection{Review of Results for Tamed Navier-Stokes Equations}
The tamed Navier-Stokes equations were introduced in \cite{RZ09a} by M. R\"ockner and X.C. Zhang and have the following form:
\begin{equation}\label{intro_eq_TNSE}
	\begin{split}
		&\frac{\partial \bm{v} }{\partial t} = \nu \Delta \bm{v} - (\bm{v} \cdot \nabla) \bm{v} - g_{N}(|\bm{v}|^2)\bm{v} - \nabla p + \bm{f} \\
		&\nabla \cdot {\bm{v}} = 0 \\
		&\bm{v}(0,x) = \bm{v}_{0}(x).
	\end{split}
\end{equation}
The ``taming function" $g_{N}$ (defined below) allowed them to obtain stronger estimates than for the untamed Navier-Stokes equations, and hence regularity results that are out of reach for the Navier-Stokes equations. Furthermore, they could show that bounded solutions to the Navier-Stokes equations, if they exist, coincide with the solutions to the tamed Navier-Stokes equations, as shown in \cite{RZ09a}. This is a feature that most regularisations of the Navier-Stokes equations do not share. 

The deterministic case was further studied by X.C. Zhang on uniform $C^{2}$-domains in \cite{ZhangX09}. 
In a series of subsequent papers, various properties of the stochastic version of the equations were studied: existence and uniqueness to the stochastic equation as well as ergodicity in \cite{RZ09a}, Freidlin-Wentzell type large deviations in \cite{RZZ10} as well as the case of existence, uniqueness and small time large deviation principles for the Dirichlet problem in bounded domains \cite{RZ12} (both with T.S. Zhang). More recently, there has been resparked interest in the subject, with contributions by Z. Dong and R.R. Zhang \cite{DZ18} (existence and uniqueness for multiplicative L\'{e}vy noise) as well as Z. Brze\'{z}niak and G. Dhariwal \cite{BD19} (existence, uniqueness and existence of invariant measures in the full space $\R^{3}$ for a slightly simplified system and by different methods).

The taming function was subsequently simplified by changing the expression of $g_{N}$ as well as replacing the argument of the function $g_{N}$ by the square of the spatial $L^{\infty}$ norm of the velocity, i.e., $g_{N} \left( \| \bm{v} \|_{\mbL^{\infty}}^{2} \right)$, see W. Liu and M. R\"ockner \cite[pp. 170 ff]{LR15}. This leads to simpler assumptions on $g_{N}$ as well as easier proofs, especially when spatial derivatives are concerned (which then act only on the remaining factor $\bm{v}$). However, this only seems to work within the framework of locally monotone operators, which cannot be applied here due to the crucial assumption of compact embeddings. Thus we do not use this simplification in this work.

\subsubsection{The Magnetic Pressure Problem}\label{DTMHD_sssec_MPP}
From the form of the MHD equations, it would seem like there should also be a ``pressure" term $\nabla \pi$ in the equation for the magnetic field. That this is not the case is due to the structure of the nonlinear term in the equation, as was noted already in the work of M. Sermange and R. Temam \cite[p. 644]{ST83}. To make this precise, note that 
\begin{equation}\label{DTMHD_eq_curlcurl}
	- (\bm{v} \cdot \nabla) \bm{B} + (\bm{B} \cdot \nabla) \bm{v} = \nabla \times (\bm{v} \times \bm{B}),
\end{equation}
i.e., the nonlinear terms in the magnetic field equation combine to an expression that is manifestly divergence-free. If there existed a magnetic pressure $\pi$ such that
\begin{align*}
	\partial_{t} \bm{B} = \Delta \bm{B} + \nabla \times (\bm{v} \times \bm{B}) - \nabla \pi,
\end{align*}
taking the divergence of this equation, observing that $\divv \bm{B} = 0$, would give
\begin{align*}
	\Delta \pi = 0,
\end{align*}
where $\nabla \pi (t,x) \in L_{\text{loc}}^{2}(\R_{+}; L^{2}(\R^{3}))$, which implies $\nabla \pi = 0$. 
Thus, a careful balancing in the two nonlinear terms leads to the ``magnetic pressure" being zero. Now, if we introduce further nonlinearities into the equation for the magnetic field, we might offset this cancellation and thus we will get an artificial ``magnetic pressure" in our tamed equations. We can show that this pressure converges to zero as $N \rightarrow \infty$, but for the tamed equations, it is undeniably present. We will informally name this phenomenon the \emph{magnetic pressure problem}:
\begin{definition}[Magnetic Pressure Problem]
	Introducing extra terms $\mathcal{N}(\bm{v},\bm{B})$ that are not divergence-free into the equation for the magnetic field $\bm{B}$ in the MHD equations will lead to the appearance of an artificial, possibly unphysical ``magnetic pressure" $\pi$, i.e., it will be of the form
	\begin{align*}
		\partial_{t} \bm{B} = \Delta \bm{B} - (\bm{v} \cdot \nabla)\bm{B} - (\bm{B} \cdot \nabla) \bm{v} - \nabla \pi + \mathcal{N}(\bm{v},\bm{B}).
	\end{align*}
\end{definition}
This term, being of gradient type, does not manifest itself in the weak formulation of the problem, which is most often studied. Our system is no exception here, so when talking about the pointwise form of the equation, we have to include the magnetic pressure term $\pi$, as above. This fact is easily overlooked when introducing regularising terms into the equation for the magnetic field. To give an example, in the work of Z.J. Zhang, C.P. Wu and Z.A. Yao \cite{ZWY18}, the authors introduce a damping term $|\bm{B}|^{\beta-1} \bm{B}$ into the magnetic field equation, but forgot to include a ``magnetic pressure" in the strong form of this equation. Note that in other regularisations of the MHD equations, such as the Leray-$\alpha$ model, this problem is avoided by only introducing terms that preserve the structure of the nonlinearities \eqref{DTMHD_eq_curlcurl}.

Ideally, one should thus introduce taming terms for the velocity field only. For mathematical reasons, however, at this point we have to content ourselves with taming terms in both components, for otherwise, in the crucial $H^{1}$-estimate \eqref{ar_eq_main_estimate} we could not cancel all four nonlinearities.

\subsubsection{The Magnetic Field: To Regularise or Not to Regularise?}
There seems to be no clear answer, even for schemes which do not introduce magnetic pressure, to the question of whether in the MHD equations the magnetic field should be regularised as well, or whether one should restrict oneself to only regularising the velocity field. A mathematical criticism formulated in J.S. Linshiz and E.S. Titi \cite[p. 3]{LT07} is that regularising the magnetic part as well might add an unnecessary amount of dissipativity to the system. However, for the mathematical reasons discussed in the previous section, we add a taming term to the magnetic field equation as well.

\subsubsection{The Tamed MHD Equations}\label{DTMHD_sssec_tamed}
We investigate the case of the deterministic version of tamed magnetohydrodynamics (TMHD) equations in this paper. They can be understood as a model of an electrically conducting fluid in a porous medium at low to moderate Reynolds numbers (cf. P.A. Markowich, E.S. Titi and S. Trabelsi \cite{MTT16}). Following the approach of M. R\"ockner and X.C. Zhang, we study the following equations:
\begin{equation}\label{intro_eq_TMHD}
	\begin{split}
		\frac{\partial \bm{v} }{\partial t} &=  \Delta \bm{v} - \left( \bm{v} \cdot \nabla \right) \bm{v}  +  \left( \bm{B} \cdot \nabla \right) \bm{B} - \nabla \left( p + \frac{| \bm{B} |^2}{2} \right) - g_{N}(| (\bm{v}, \bm{B}) |^{2}) \bm{v} + \bm{f}_{v} \\
		\frac{\partial \bm{B}}{\partial t}
		&= \Delta \bm{B} - \left( \bm{v} \cdot \nabla \right) \bm{B} + (\bm{B} \cdot \nabla) \bm{v} - \nabla \pi - g_{N}(| (\bm{v}, \bm{B}) |^{2}) \bm{B} + \bm{f}_{B}.
	\end{split}
\end{equation}
If we write $y := (\bm{v}, \bm{B})$, the equations differ from the ``untamed" MHD equations by the taming term 
$$
 - g_{N}( | y(t,x) |^{2})y(t,x),
$$
which is a direct generalisation of the term in \eqref{intro_eq_TNSE}. The norm is defined in equation \eqref{ar_eq_def_Snorm} below. 

The taming function $g_{N} \colon \mathbb{R}_{+} \rightarrow \mathbb{R}_{+}$ is defined by
\begin{equation}\label{DTMHD_eq_def_g_N}
	\begin{cases}
		g_{N}(r) := 0, \quad & r \in [0,N], \\
		g_{N}(r) := C_{\text{taming}} \left( r - N - \frac{1}{2} \right)	, \quad & r \geq N+1, \\
		0 \leq g_{N}'(r) \leq C_{1}, \quad & r \geq 0, \\
		| g_{N}^{(k)}(r)| \leq C_{k}, \quad & r \geq 0, k \in \mathbb{N}.
	\end{cases}
\end{equation}
Here, the constant $C_{\text{taming}}$ is defined by
$$
	C_{\text{taming}} := 2 \max \{ Re, Rm \} = 2.
$$
For the Navier-Stokes case, M. R\"ockner and X.C. Zhang in \cite{RZ09a} set $C = \frac{1}{\nu} \propto Re$, so the fact that $C_{\text{taming}} \propto Re$ is not surprising. The factor 2 arises from the fact that we need to tame more terms here. The dependency on $Rm$ seems natural as well.

The idea of the taming procedure remains very clear: try to counteract the nonlinear terms of which there are four in the case of the MHD equations. To pinpoint the exact place where the power of the taming function unfolds, see the discussion after Lemma \ref{ar_thm_estimates}.
\newpage
\subsection{Results and Structure of This Paper}\label{intro_ssec_results}
We follow the ideas of \cite{RZ09a}. However, the proof of the regularity of the solution requires an MHD adaptation of a result from E.B. Fabes, B.F. Jones and N.M. Rivi\`{e}re \cite{FJR72}, which the author could not find in the literature. See Appendix \ref{chap:FJR} for a discussion and a proof of this result.

Our main results can be summarised as follows:
\begin{theorem}[Global well-posedness]\label{DTMHD_thm_intro_GWP}
	Let $y_{0} = (\bm{v}_{0}, \bm{B}_{0}) \in \mathcal{H}^{1}$ and $f = (\bm{f}_{v}, \bm{f}_{B}) \in L_{\text{loc}}^{2}(\mathbb{R}_{+}; \mathcal{H}^{0})$. For any $N > 0$, there exists a unique weak solution $y$ to the TMHD equation in the sense of Definition \ref{exuniq_def_weak}, depending continuously on the initial data, such that
	\begin{enumerate}[label=(\roman*), ref=(\roman*)]
		\item\label{exuniq_itm_H0est} For all $t \geq 0$,
		\begin{equation}\label{exuniq_eq_soln_H0est}
			\| y(t) \|_{\mathcal{H}^{0}} \leq \| y_{0} \|_{\mathcal{H}^{0}} + \int_{0}^{t} \| f(s) \|_{\mathcal{H}^{0}} \diff s,
		\end{equation}
		and
		\begin{equation}\label{exuniq_eq_soln_H1dot_int_est}
            \begin{split}
			&\int_{0}^{t} \| \nabla y(s) \|_{\mathcal{H}^{0}} + \| \sqrt{g_{N}(|y(s)|^2)} |y(s)| \|_{L^{2}}^{2} \diff s \\
			&\leq \| y_{0} \|_{\mathcal{H}^{0}}^{2} + 2 \left[ \int_{0}^{t} \| f(s) \|_{\mathcal{H}^{0}} \diff s \right]^{2}.
			\end{split}
		\end{equation}
		\item\label{exuniq_itm_H1est} The solution satisfies $y \in C(\mathbb{R}_{+}; \mathcal{H}^{1}) \cap L_{\text{loc}}^{2}(\mathbb{R}_{+}; \mathcal{H}^{2})$, $\partial_{t} y \in L_{\text{loc}}^{2}(\mathbb{R}_{+}; \mathcal{H}^{0})$ and for all $t \geq 0$,
		\begin{equation}\label{exuniq_eq_soln_H1est}
			\begin{split}
			&\| y(t) \|_{\mathcal{H}^{1}}^{2} + \int_{0}^{t} \left( \| y(s) \|_{\mathcal{H}^{2}}^{2} + \| | y(s) | | \nabla y(s) | \|_{L^{2}}^{2} \right) \diff s \\
			& \leq C \left( \| y_{0} \|_{\mathcal{H}^{1}}^{2} + \int_{0}^{t} \| f(s) \|_{\mathcal{H}^{0}}^{2} \diff s \right)  \\
			&\quad + C(1 + N + t) \left( \| y_{0} \|_{\mathcal{H}^{0}}^{2} + \left[ \int_{0}^{t} \| f(s) \|_{\mathcal{H}^{0}} \diff s \right]^{2} \right).
			\end{split}
		\end{equation}
		\item\label{exuniq_itm_pointwiseeqn} There exist real-valued functions $p(t,x)$ and $\pi(t,x)$, satisfying the conditions $\nabla p \in L_{\text{loc}}^{2}(\mathbb{R}_{+}; L^{2}(\mathbb{R}^{3}; \mathbb{R}^{3}))$, $\nabla \pi \in L_{\text{loc}}^{2}(\mathbb{R}_{+}; L^{2}(\mathbb{R}^{3}; \mathbb{R}^{3}))$, such that for almost all $t \geq 0$, in $L^{2}(\mathbb{R}^{3}; \mathbb{R}^{6})$ we have
		\begin{align*}
		\frac{\partial \bm{v} }{\partial t} &=  \Delta \bm{v} - \left( \bm{v} \cdot \nabla \right) \bm{v}  +  \left( \bm{B} \cdot \nabla \right) \bm{B} - \nabla \left( p + \frac{| \bm{B} |^2}{2} \right) - g_{N}(| (\bm{v}, \bm{B}) |^{2}) \bm{v} + \bm{f}_{v}, \\
		\frac{\partial \bm{B}}{\partial t}
		&= \Delta \bm{B} - \left( \bm{v} \cdot \nabla \right) \bm{B} + (\bm{B} \cdot \nabla) \bm{v} - \nabla \pi - g_{N}(| (\bm{v}, \bm{B}) |^{2}) \bm{B} + \bm{f}_{B}.
		\end{align*}
	\end{enumerate}
\end{theorem}
In the case of smooth data, we can prove smoothness of the solutions to the TMHD equations:
\begin{theorem}[Regularity and Strong Solutions]\label{DTMHD_thm_reg}
Let $y_{0} \in \mathcal{H}^{\infty} := \bigcap_{m \in \mathbb{N}_{0}} \mathcal{H}^{m}$ and $\mathbb{R}_{+} \ni t \mapsto f(t) \in \mathcal{H}^{m}$ be smooth for any $m \in \mathbb{N}_{0}$. Then there exists a unique smooth velocity field 
\begin{align*}
	\bm{v}_{N} \in C^{\infty}(\mathbb{R}_{+} \times \mathbb{R}^{3} ; \mathbb{R}^{3}) \cap C(\mathbb{R}_{+}; \mathcal{H}^{2}),
\end{align*}
a unique smooth magnetic field
\begin{align*}
	\bm{B}_{N} \in C^{\infty}(\mathbb{R}_{+} \times \mathbb{R}^{3} ; \mathbb{R}^{3}) \cap C(\mathbb{R}_{+}; \mathcal{H}^{2}),
\end{align*}
and smooth pressure functions
\begin{align*}
	p_{N}, \pi_{N} \in C^{\infty}(\mathbb{R}_{+} \times \mathbb{R}^{3}; \mathbb{R}),
\end{align*}
which are defined up to a time-dependent constant. Furthermore, the quadruplet $(\bm{v}_{N}, \bm{B}_{N}, p_{N}, \pi_{N})$ solves the tamed MHD equations \eqref{intro_eq_TMHD}.

Moreover, we have the following estimates: for any $T, N > 0$
\begin{equation}\label{exuniq_eq_ssoln_H0est}
	\sup_{t \in [0,T]} \| y_{N}(t) \|_{\mathcal{H}^{0}}^{2} + \int_{0}^{T} \| \nabla y_{N} \|_{\mathcal{H}^{0}}^{2} \diff s \leq C \left( \| y_{0} \|_{\mathcal{H}^{0}}^{2} + \left[ \int_{0}^{T} \| f(s) \|_{\mathcal{H}^{0}} \diff s \right]^{2} \right),
\end{equation}
\begin{equation}\label{exuniq_eq_ssoln_H1est}
	\sup_{t \in [0,T]} \| y_{N}(t) \|_{\mathcal{H}^{1}}^{2} + \int_{0}^{T} \| y_{N}(s) \|_{\mathcal{H}^{2}}^{2} \diff s \leq C_{T,y_{0},f} \cdot (1+N),
\end{equation}
\begin{equation}\label{exuniq_eq_ssoln_H2est}
	\sup_{t \in [0,T]} \| y_{N}(t) \|_{\mathcal{H}^{2}}^{2} \leq C_{T,y_{0},f}' + C_{T,y_{0},f}\cdot(1+N^{2}).
\end{equation}
To be precise, the constant $C_{T,y_{0},f}$ depends on $\| y_{0} \|_{\mathcal{H}^{1}}$ and $\int_{0}^{T} \| f \|_{\mathcal{H}^{0}} \diff s$ and goes to zero as both these quantities tend towards zero. The constant $C_{T,y_{0},f}'$ depends on $T$, $\| y_{0} \|_{\mathcal{H}^{2}}$ and $\sup_{t \in [0,T]} \| f(t) \|_{\mathcal{H}^{0}}$ as well as $\int_{0}^{T} \| \partial_{s} f \|_{\mathcal{H}^{0}}^{2} \diff s$.
\end{theorem}
Finally, we have the following convergence result for vanishing taming terms, i.e., in the limit $N \rightarrow \infty$.

\begin{theorem}[Convergence to the untamed equations]\label{DTMHD_thm_conv}
	Let $y_{0} \in \mathcal{H}^{0}$, $f \in L^{2}([0,T];\mathcal{H}^{0})$, $y_{0}^{N} \in \mathcal{H}^{1}$ such that $\mathcal{H}^{0}-\lim_{N \rightarrow \infty} y_{0}^{N} = y_{0}$. Denote by $(y_N, p_{N}, \pi_{N})$ the unique solutions to the tamed equations \eqref{intro_eq_TMHD} with initial value $y_{0}^{N}$ given by Theorem \ref{DTMHD_thm_intro_GWP}. 
	
	Then there is a subsequence $(N_{k})_{k \in \mathbb{N}}$ such that $y_{N_{k}}$ converges to a $y$ in $L^{2}([0,T];\mathcal{L}_{\text{loc}}^{2})$ and $p_{N_{k}}$ converges weakly to some $p$ in $L^{9/8}([0,T];L^{9/5}(\mathbb{R}^{3}))$. The magnetic pressure $\pi_{N_{k}}$ converges to zero, weakly in $L^{9/8}([0,T];L^{9/5}(\mathbb{R}^{3}))$. Furthermore, $(y,p)$ is a weak solution to \eqref{intro_eq_DMHD} such that the following generalised energy inequality holds:
	\begin{equation*}
		\begin{split}
		2 \int_{0}^{T} \int_{\mathbb{R}^{3}} | \nabla y |^{2} \phi \diff x \diff s &\leq \int_{0}^{T} \int_{\mathbb{R}^{3}} \Big[ |y|^{2} \left( \partial_{t} \phi + \Delta \phi \right) + 2 \langle y, f \rangle \phi \\
		&\quad + (|y|^{2} - 2p)\langle \bm{v}, \nabla \phi \rangle - 2 \langle \bm{B}, \bm{v} \rangle \langle \bm{B}, \nabla \phi \rangle \Big] \diff x \diff s.
		\end{split}
	\end{equation*}
\end{theorem}
\begin{remark}
	Compared to the Navier-Stokes case, we get a different type of term in the inequality, namely the last one on the right-hand side of the inequality.	
	Note also that the "magnetic pressure" disappears as $N_{k} \rightarrow \infty$.
\end{remark}
We have been able to extend all the results of \cite{RZ09a} to the case of tamed MHD equations. This posed several technical obstacles: we had to extend the regularity result of \cite{FJR72} to the MHD case, which the author could not find in the literature. Moreover, we describe the magnetic pressure problem in regularised MHD equations. Furthermore, our work basically provides the critical case $\alpha = \beta = 3$ of the model considered in \cite{Ye16, ZY16, ZWY18}.

The paper is organised as follows: we start in Section \ref{sec:DTMHD_ar} by introducing the functional framework of the problem. Then we state and prove a number of elementary lemmas regarding estimates as well as (local) convergence results for the operators appearing in the tamed MHD equations. Existence and uniqueness of a weak solution is shown in Section \ref{sec:DTMHD_exuniq} via a Faedo-Galerkin approximation procedure. Employing the results of Appendix \ref{chap:FJR}, we then show in Section  \ref{sec:DTMHD_regul} that for smooth data the solution to the tamed MHD equations remains smooth. Finally, in Section \ref{sec:DTMHD_conv} we show that as $N \rightarrow \infty$, the solution to the tamed MHD equations converges to a suitable weak solution of the (untamed) MHD equations.

The results of this paper are part of the dissertation \cite{Schenke20} of the author.

\subsection{Notation}\label{intro_ssec_notation}
Let $G \subset \mathbb{R}^{3}$ be a domain and denote the divergence operator by $\divv$. We use the following notational hierarchy for $L^{p}$ and Sobolev spaces:  
\begin{enumerate}[label=(\arabic*), ref=(\arabic*)]
	\item For the spaces $L^{p}(G,\mathbb{R})$ of real-valued integrable (equivalence classes of) functions we use the notation $L^{p}(G)$ or $L^{p}$ if no confusion can arise. These are the spaces of the components $v^{i}$, $B^{i}$ of the velocity and magnetic field vector fields.
	\item We sometimes use the notation $\bm{L}^{p}(G) := L^{p}(G;\mathbb{R}^{3})$ to denote 3-D vector-valued integrable quantities, especially the velocity vector field and magnetic vector field $\bm{v}$ and $\bm{B}$.
	\item The divergence-free and $p$-integrable vector fields will be denoted by \verb=\mathbb= symbols, so $\mathbb{L}^{p}(G) := \bm{L}^{p}(G) \cap \divv^{-1} \{ 0 \}$. Its elements are still denoted by bold-faced symbols $\bm{v}$, $\bm{B}$ and they satisfy by definition $\divv \bm{v} = \nabla \cdot \bm{v} = 0$, $\divv \bm{B} = 0$.
	\item Finally, we denote the space of the combined velocity and magnetic vector fields by \verb=\mathcal= symbols, i.e., $\mathcal{L}^{p}(G) := \mathbb{L}^{p}(G) \times \mathbb{L}^{p}(G)$. It contains elements of the form $y = (	\bm{v} , \bm{B})$, with both $\bm{v}$ and $\bm{B}$ divergence-free.
\end{enumerate}
For Sobolev spaces, we use the same notational conventions, so for example $\mathbb{H}^{k}(G) := \bm{H}^{k}(G) \cap \divv^{-1} \{ 0 \} := W^{k,2}(G ; \mathbb{R}^{3}) \cap \divv^{-1} \{ 0 \} $ etc. Finally, if the domain of the functions is not in $\mathbb{R}^{3}$, in particular if it is a real interval (for the time variable), then we use the unchanged $L^{p}$ notation.

For brevity, we use the following terminology when discussing the terms on the right-hand side of the tamed MHD equations: the terms involving the Laplace operator are called the \emph{linear terms}, the terms involving the taming function $g_{N}$ are called \emph{taming terms} and the other terms are called the \emph{nonlinear terms}. Furthermore, we refer to the initial data $y_{0} = (\bm{v}_{0} , \bm{B}_{0} )$ and the force $f = ( \bm{f}_{v} , \bm{f}_{B})$ collectively as the \emph{data} of the problem.

\section{The Tamed MHD Equations on the Whole Space}\label{DTMHD_sec_Cauchy}

\subsection{Auxiliary Results}\label{sec:DTMHD_ar}
We define the following spaces:
\begin{align*}
	\bm{W}^{m,p} := \overline{C_{0}^{\infty}(\R^{3};\R^{3})}^{\| \cdot \|_{m,p}},
\end{align*}
the closure being with respect to the norm
\begin{align*}
	\| \bm{u} \|_{m,p} := \left( \int_{\R^{3}} | (I - \Delta)^{m/2} \bm{u} |^{p} \diff x \right)^{1/p}.
\end{align*}
This norm is equivalent to the Sobolev norm given by
\begin{align*}
	\| \bm{u} \|_{W^{m,p}} := \sum_{j=0}^{m} \| \nabla^{j} \bm{u} \|_{L^{p}}
\end{align*}
where $\nabla^{j} \bm{u}$ denotes the j-th total weak derivative of $\bm{u}$ of order j. We define the solenoidal spaces by
\begin{equation}
	\mathbb{H}^{m} := \{ \bm{u} \in \bm{W}^{m,2} ~|~ \nabla \cdot \bm{u} = 0 \},
\end{equation}
where the divergence is taken in the sense of Schwartz distributions.

To handle the velocity and the magnetic field of the MHD equations at the same time, we will need to define a norm on the space $\mathcal{H}^{m} := \mathbb{H}^{m} \times \mathbb{H}^{m}$. We will define the scalar products in the usual way (see \cite{ST83}, p. 7): for the vector field $y = (\bm{v}, \bm{B})$ define 
\begin{equation} \label{ar_eq_def_Snorm}
	\langle y_{1} (x), y_{2}(x) \rangle := \left\langle 
	\begin{pmatrix}
		\bm{v}_{1} \\ \bm{B}_{1}
	\end{pmatrix} (x),
	\begin{pmatrix}
		\bm{v}_{2} \\ \bm{B}_{2}
	\end{pmatrix} (x)
	\right\rangle
	:= \langle \bm{v}_{1}(x) , \bm{v}_{2}(x) \rangle + \langle \bm{B}_{1}(x), \bm{B}_{2}(x) \rangle
\end{equation}
and similarly, for $y \in \mbH^{m} \times \mbH^{m}$, we set
\begin{equation} 
	( y_1, y_{2} )_{\mathcal{H}^{m}} := (\bm{v}_{1}, \bm{v}_{2})_{\mbH^{m}} + (\bm{B}_{1}, \bm{B}_{2})_{\mbH^{m}},
\end{equation}
and accordingly for the norms. They behave just like an $\ell^{2}$-type product norm.

In a similar fashion we define Lebesgue norms by
\begin{align*}
	\| y \|_{\mathcal{L}^{p}} := \left( \int_{\R^{d}} \left( | \bm{v} |^{2} +  | \bm{B} |^{2} \right)^{p/2} \diff x \right)^{1/p} = \| ~|y|~ \|_{L^{p}(\mathbb{R}^{3})}
\end{align*}
and
\begin{equation}\label{ar_eq_def_L_infty}
	\| y \|_{\mathcal{L}^{\infty}} := \esssup_{x \in \R^{3}} \left( | \bm{v}(x)|^{2} + | \bm{B}(x) |^{2} \right)^{1/2} = \esssup_{x \in \mathbb{R}^{3}} |y(x)|.
\end{equation}

We will often employ the following Gagliardo-Nirenberg-Sobolev-type interpolation inequality: Let $p,q,r \geq 1$ and $0 \leq j < m$. Assume the following three conditions:
$$
	m - j - \frac{3}{p} \notin \mathbb{N}_{0}, \quad \frac{1}{r} = \frac{j}{3} + \alpha \left( \frac{1}{p} - \frac{m}{3} \right) + \frac{1 - \alpha}{q}, \quad \frac{j}{m} \leq \alpha \leq 1.
$$
Then for any $\bm{u} \in W^{m,p} \cap L^{q}(\mathbb{R}^{3};\mathbb{R}^{3})$, we have the following estimate:
\begin{equation}\label{ar_eq_interpolation_3D}
	\| \nabla^{j} \bm{u} \|_{L^{r}} \leq C_{m,j,p,q,r} \| \bm{u} \|_{m,p}^{\alpha} \| \bm{u} \|_{L^{q}}^{1-\alpha}.
\end{equation}
Applying it to each component of the norm for $y = (\bm{v},\bm{B})$, the same estimate carries over to yield
\begin{equation}\label{ar_eq_interpolation_6D}
	\| \nabla^{j} y \|_{L^{r}} \leq C_{m,j,p,q,r} \| y \|_{m,p}^{\alpha} \| y \|_{L^{q}}^{1-\alpha}.
\end{equation}
Define the space of (solenoidal) test functions by
$$
	\mathcal{V} := \{ y = (\bm{v},\bm{B}) ~ \colon ~ \bm{v}, \bm{B} \in C_{0}^{\infty}(\R^{3}; \R^{3}), \nabla \cdot \bm{v} = 0, \nabla \cdot \bm{B} = 0 \} \subset C_{0}^{\infty}(\R^{3}; \R^{6}).
$$
We then have the following lemma.
\begin{lemma}\label{ar_thm_density}
	The space $\mathcal{V}$ is dense in $\mathcal{H}^{m}$ for any $m \in \mathbb{N}$.
\end{lemma}
\begin{proof}
	See \cite{RZ09a}, Lemma 2.1.
\end{proof}

Let $P \colon L^{2}(\R^{3};\R^{3}) \rightarrow \mbH^{0}$ be the Leray-Helmholtz projection. Then $P$ commutes with derivative operators (\cite[Lemma 2.9, p. 52]{RRS16}) and can be restricted to a bounded linear operator
\begin{align*}
	P \restr{H^{m}} \colon H^{m} \rightarrow \mbH^{m}.
\end{align*} 
Furthermore, consider the tensorised projection
\begin{align*}
	\mathcal{P} := P \otimes P, \quad \mcP y := (P \otimes P) \begin{pmatrix} \bm{v} \\ \bm{B}	\end{pmatrix} = \begin{pmatrix} P \bm{v} \\ P \bm{B} \end{pmatrix}.
\end{align*}
Then $\mcP \colon \mcL^{2} \rightarrow \mcH^{0}$ is a bounded linear operator.
We define the following operator for the terms on the right-hand side of the TMHD equations, projected on the space of divergence free functions:
\begin{align*}
	\mcA(y) := \mcP \D y - \mcP \begin{pmatrix}
		(\bm{v} \cdot \nabla) \bm{v}  - (\bm{B} \cdot \nabla) \bm{B} \\ (\bm{v} \cdot \nabla) \bm{B} - (\bm{B} \cdot \nabla) \bm{v} 
	\end{pmatrix} - \mcP \left( g_{N}(|y|^{2})y \right).
\end{align*}
For $y := (\bm{v}, \bm{B})$ and a test function $\tilde{y} := (\tilde{\bm{v}}, \tilde{\bm{B}}) \in \mcH^{1}$, consider (using the self-adjointness of the projection $\mcP$)
\begin{equation}
	\begin{split}
		\langle \mcA(y), \tilde{y} \rangle_{\mcH^{0}} &= \langle \bm{v}, \Delta \tilde{\bm{v}} \rangle_{\bm{L}^{2}} + \langle \bm{B}, \Delta \tilde{\bm{B}} \rangle_{\bm{L}^{2}} - \langle (\bm{v} \cdot \nabla) \bm{v}, \tilde{\bm{v}} \rangle_{\bm{L}^{2}} + \langle (\bm{B} \cdot \nabla )\bm{B}, \tilde{\bm{v}} \rangle_{\bm{L}^{2}} \\
		& - \langle (\bm{v} \cdot \nabla) \bm{B}, \tilde{\bm{B}} \rangle_{\bm{L}^{2}} + \langle (\bm{B} \cdot \nabla) \bm{v}, \tilde{\bm{B}} \rangle_{\bm{L}^{2}} - g_{N}(|y|^{2}) \langle y , \tilde{y} \rangle.
	\end{split}
	\end{equation}
and for $\tilde{y} \in \mathcal{H}^{3}$
\begin{equation}\label{ar_eq_testing_H1}
	\begin{split}
		\langle \mcA(y), \tilde{y} \rangle_{\mcH^{1}} &= \langle \mcA(y), (I - \Delta )\tilde{y} \rangle_{0} \\
		&= - \langle \nabla \bm{v}, (I - \Delta)\nabla \tilde{\bm{v}} \rangle_{L^{2}} - \langle \nabla \bm{B}, \nabla (I - \Delta) \tilde{\bm{B}} \rangle_{L^{2}} \\
		&- \langle (\bm{v} \cdot \nabla) \bm{v}, (I - \Delta)\tilde{\bm{v}} \rangle_{L^{2}} + \langle (\bm{B} \cdot \nabla )\bm{B}, (I - \Delta)\tilde{\bm{v}} \rangle_{L^{2}} \\
		& - \langle (\bm{v} \cdot \nabla) \bm{B}, (I - \Delta)\tilde{\bm{B}} \rangle_{L^{2}} + \langle (\bm{B} \cdot \nabla) \bm{v}, (I - \Delta)\tilde{\bm{B}} \rangle_{L^{2}} \\
		& - \langle g_{N}(|y|^{2})y, (I-\Delta)\tilde{y} \rangle_{\mathcal{L}^{2}}.
	\end{split}
\end{equation}
Let us give names to the linear, nonlinear and taming terms of \eqref{ar_eq_testing_H1}:
\begin{align*}
	\mcA_{1}(y, \tilde{y}) &:= - \langle \nabla \bm{v}, (I - \Delta)\nabla \tilde{\bm{v}} \rangle_{L^{2}} - \langle \nabla \bm{B}, \nabla (I - \Delta) \tilde{\bm{B}} \rangle_{L^{2}}, \\
	\mcA_{2}(y,\tilde{y}) &:= -\langle (\bm{v} \cdot \nabla) \bm{v}, (I - \Delta)\tilde{\bm{v}} \rangle_{L^{2}} + \langle (\bm{B} \cdot \nabla )\bm{B}, (I - \Delta)\tilde{\bm{v}} \rangle_{L^{2}} \\
		& \quad - \langle (\bm{v} \cdot \nabla) \bm{B}, (I - \Delta)\tilde{\bm{B}} \rangle_{L^{2}} + \langle (\bm{B} \cdot \nabla) \bm{v}, (I - \Delta)\tilde{\bm{B}} \rangle_{L^{2}}, \\
	\mcA_{3}(y, \tilde{y}) &:= - \langle g_{N}(|y|^{2})y, (I-\Delta)\tilde{y} \rangle_{\mathcal{L}^{2}}.
\end{align*}
The following lemma provides elementary estimates on the terms defined above.

\begin{lemma}~
\label{ar_thm_estimates}
\begin{enumerate}[label=(\roman*), ref=(\roman*)]
	\item\label{ar_itm_H3_bound} For any $y \in \mathcal{H}^{1}$ and $\tilde{y} \in \mathcal{V}$,
$$
	| \langle \mcA(y), \tilde{y} \rangle_{\mcH^{1}} | \leq C (1 + \| y \|_{\mathcal{H}^{1}}^{3}) \| \tilde{y} \|_{\mathcal{H}^{3}},
$$
i.e., $\langle \mcA(y), \cdot \rangle_{\mcH^{1}}$ can be considered as an element in the dual space $(\mathcal{H}^{3})'$ with its norm bounded by $C(1 + \| y \|_{\mathcal{H}^{1}}^{3})$.
	\item If $y \in \mathcal{H}^{1}$, then 
	\begin{equation}\label{ar_eq_0eq}
		\langle \mcA(y), y \rangle_{\mcH^{0}} = - \| \nabla y \|_{\mathcal{H}^{0}}^{2} - \langle g_{N}( | y |^{2}) y , y \rangle_{\mathcal{L}^{2}}.
	\end{equation}
	\item If $y \in \mathcal{H}^{2}$, then
\begin{equation} \label{ar_eq_main_estimate}
	\begin{split}
		\langle \mcA(y), y \rangle_{\mcH^{1}} &\leq - \frac{1}{2} \| y \|_{\mathcal{H}^{2}}^{2} + \| y \|_{\mathcal{H}^{0}}^{2} + 2(N + 1) \| \nabla y \|_{\mathcal{H}^{0}}^{2}  \\
		&\quad - \| | \bm{v} | | \nabla \bm{v} | \|_{\bm{L}^{2}}^{2}  - \| | \bm{B} | | \nabla \bm{B} | \|_{\bm{L}^{2}}^{2} \\
		&\quad - \| | \bm{v}| | \nabla \bm{B} | \|_{\bm{L}^{2}}^{2} - \| | \bm{B} | | \nabla \bm{v} | \|_{\bm{L}^{2}}^{2}.
	\end{split}
\end{equation}
\end{enumerate}
\end{lemma}
\begin{proof}
Throughout this proof, let $\bm{\varphi}, \bm{\psi}, \bm{\theta} \in \{ \bm{v}, \bm{B} \}$. 

To prove \ref{ar_itm_H3_bound}, using also the Sobolev embedding theorem we find
\begin{align*}
	\langle \bm{\varphi} \otimes \bm{\psi}, \nabla (I-\Delta) \bm{\theta} \rangle_{\bm{L}^{2}} \leq C \| \bm{\varphi} \|_{\bm{L}^{4}} \| \bm{\psi} \|_{\bm{L}^{4}} \| \bm{\theta} \|_{\mathbb{H}^{3}} \leq C \| \bm{\varphi} \|_{\mathbb{H}^{1}} \| \bm{\psi} \|_{\mathbb{H}^{1}} \| \bm{\theta} \|_{\mathbb{H}^{3}},
\end{align*}
which yields
$$
	\mcA_{2}(y, \tilde{y}) \leq C \| y \|_{\mathcal{H}^{1}}^{2} \| \tilde{y} \|_{\mathcal{H}^{3}}.
$$
Similarly one can show $\mcA_{1}(y,\tilde{y}) \leq C \| y \|_{\mathcal{H}^{1}} \| \tilde{y} \|_{\mathcal{H}^{3}},$ $\mcA_{3}(y,\tilde{y}) \leq  C \| y \|_{\mathcal{H}^{1}}^{3} \| \tilde{y} \|_{\mathcal{H}^{3}},$ which implies the assertion.

Equality \eqref{ar_eq_0eq} follows from the divergence-freeness.

To prove \eqref{ar_eq_main_estimate}, we start with the inequality
\begin{align*}
	\mcA_{1}(y,y) = -  \| \bm{v} \|_{\mathbb{H}^{2}}^{2} -  \| \bm{B} \|_{\mathbb{H}^{2}}^{2} +  \| \nabla \bm{v} \|_{\mathbb{H}^{0}}^{2} +  \| \nabla \bm{B} \|_{\mathbb{H}^{0}}^{2} + \| \bm{v} \|_{\mathbb{H}^{0}}^{2} +  \| \bm{B} \|_{\mathbb{H}^{0}}^{2}.
\end{align*}
The nonlinear terms can be estimated by Young's inequality:
\begin{align*}
	\mcA_{2}(y,y) &\leq \| (\bm{v} \cdot \nabla) \bm{v} \|_{\bm{L}^{2}}^{2} + \frac{1}{4} \| \bm{v} \|_{\mathbb{H}^{2}}^{2} + \| (\bm{B} \cdot \nabla) \bm{B} \|_{\bm{L}^{2}}^{2} + \frac{1}{4} \| \bm{v} \|_{\mathbb{H}^{2}}^{2} \\
		& \quad + \| (\bm{v} \cdot \nabla) \bm{B} \|_{\bm{L}^{2}}^{2} + \frac{1}{4} \| \bm{B} \|_{\mathbb{H}^{2}}^{2} + \| (\bm{B} \cdot \nabla) \bm{v} \|_{\bm{L}^{2}}^{2} + \frac{1}{4} \| \bm{B} \|_{\mathbb{H}^{2}}^{2}.
\end{align*}
The taming terms are estimated using $ g_{N}(r) \geq C(r - N)$:
\begin{align*}
	&\mcA_{3}(y,y) = - \langle g_{N}(|y|^{2})y,y \rangle_{\mathcal{L}^{2}} - \langle \nabla \left( g_{N}(|y|^{2})y \right), \nabla y \rangle_{\mathcal{L}^{2}} \\
	&\leq - \int_{\mathbb{R}^{3}} \sum_{i,k=1}^{3} \partial_{i} v^{k} \partial_{i} \left( g_{N}(|y|^{2}) v^{k} \right) -  \partial_{i} B^{k} \partial_{i} \left( g_{N}(|y|^{2}) B^{k} \right) \diff x \\
	&= - \int_{\mathbb{R}^{3}} g_{N}(|y|^{2}) | \nabla y |^{2} \diff x - \frac{1}{2} \int_{\mathbb{R}^{3}} g_{N}'(|y|^{2}) \left| \nabla | y |^{2} \right| \diff x \\
	&\leq - \int_{\mathbb{R}^{3}} g_{N}(|y|^{2}) | \nabla y |^{2} \diff x,
\end{align*}
Finally, since $g_{N}(|y|^{2}) \geq C_{\text{taming}} \left( |y|^{2} - (N + \frac{1}{2}) \right)$ by definition, we get that
\begin{align*}
	\mcA_{3}(y,y) &\leq - C_{\text{taming}} \int_{\mathbb{R}^{3}} | y |^{2} | \nabla y |^{2} \diff x + C_{\text{taming}} (N + \frac{1}{2}) \| \nabla y \|_{\mathcal{H}^{0}}^{2} \\
	&= -C_{\text{taming}} \big( \| | \bm{v} | | \nabla \bm{v} | \|_{L^{2}}^{2} + \| | \bm{v} | | \nabla \bm{B} | \|_{L^{2}}^{2} \\
	&\quad + \| | \bm{B} | | \nabla \bm{v} | \|_{L^{2}}^{2} + \| | \bm{B} | | \nabla \bm{B} | \|_{L^{2}}^{2} \big) + C_{\text{taming}} (N + \frac{1}{2}) \| \nabla y \|_{\mathcal{H}^{0}}^{2}.
\end{align*}
Since $C_{\text{taming}} = 2$, we get \eqref{ar_eq_main_estimate} by combining the above three estimates.
\end{proof}
\begin{lemma}\label{ar_thm_convergence} Let $y_{n}, \tilde{y} \in \mathcal{V}$ and $y \in \mathcal{H}^{1}$. Let $\Omega := \supp (\tilde{y})$ and assume that
\begin{align*}
	\sup_{n} \| y_{n} \|_{\mathcal{H}^{1}} < \infty, \quad \lim_{n \rightarrow \infty} \| (y_{n} - y) 1_{\Omega} \|_{\mathcal{L}^{2}} = 0.
\end{align*}
Then 
\begin{align*}
	\lim_{n \rightarrow \infty} \langle \mcA(y_{n}), \tfy \rangle_{\mcH^{1}} = \langle \mcA(y), \tfy \rangle_{\mcH^{1}}.
\end{align*}
\end{lemma}
\begin{proof}
The proof follows in the same way as in \cite[Lemma 2.3, p. 533 f.]{RZ09a}.
\end{proof}
\subsection{Existence and Uniqueness of Weak Solutions}\label{sec:DTMHD_exuniq}
In this section we will study the well-posedness of the weak formulation of the TMHD equations. We start by stating our notion of weak solution. We proceed to show uniqueness first and then existence of weak solutions via a Faedo-Galerkin approximation scheme.
\begin{definition}[Weak solution]\label{exuniq_def_weak}
	Let $y_{0} \in \mathcal{H}^{0}$, $f \in L^{2}_{\text{loc}}(\R_{+}; \mathcal{H}^{0})$. Let $y = \begin{pmatrix}
	\bm{v} \\ \bm{B}	
\end{pmatrix}$ where $\bm{v}$ and $\bm{B}$ are measurable vector fields, $\bm{v}, \bm{B} \colon \mathbb{R}_{+} \times \mathbb{R}^{3} \rightarrow \mathbb{R}^{3}$. We call $y$ a \emph{weak solution} of the tamed MHD equations if
\begin{enumerate}[label=(\roman*), ref=(\roman*)]
	\item $\bm{v}, \bm{B} \in L^{\infty}_{\text{loc}}(\mathbb{R}_{+} ; L^{4}(\mathbb{R}^{3}; \mathbb{R}^{3})) \cap L^{2}_{\text{loc}}(\mathbb{R}_{+} ; \mathbb{H}^{1})$.
	\item For all $\tilde{y} \in \mathcal{V}$ and $t \geq 0$,
	\begin{equation} \label{exuniq_eq_defweak}
		\begin{split}
		\langle y(t), \tilde{y} \rangle_{\mathcal{H}^{0}} 
		&= \langle y_{0}, \tilde{y} \rangle_{\mathcal{H}^{0}} - \int_{0}^{t}  \langle \nabla y, \nabla \tilde{y} \rangle_{\mathcal{H}^{0}} \diff s  - \int_{0}^{t} \langle (\bm{v} \cdot \nabla) \bm{v}, \tilde{\bm{v}} \rangle_{\bm{L}^{2}} \diff s + \int_{0}^{t} \langle (\bm{B} \cdot \nabla) \bm{B}, \tilde{\bm{v}} \rangle_{\bm{L}^{2}} \diff s \\
		&\quad - \int_{0}^{t} \langle (\bm{v} \cdot \nabla) \bm{B}, \tilde{\bm{B}} \rangle_{\bm{L}^{2}} \diff s + \int_{0}^{t} \langle (\bm{B} \cdot \nabla) \bm{v}, \tilde{\bm{B}} \rangle_{\bm{L}^{2}} \diff s - \int \langle g_{N}(|y|^{2})y, \tilde{y} \rangle_{L^{2}} \diff s + \int_{0}^{t} \langle f, \tilde{y} \rangle_{\mathcal{H}^{0}} \diff s.
		\end{split}
	\end{equation}
	\item $\lim_{t \downarrow 0} \| y(t) - y_{0} \|_{\mathcal{L}^{2}} = 0$.
\end{enumerate}
\end{definition}
This definition deals with purely spatial test functions, but it can be extended to the case of test functions that depend also on time, as the next proposition demonstrates.
\begin{proposition}\label{exuniq_thm_spacetimetest}
	Let $y = \begin{pmatrix} \bm{v} \\ \bm{B} \end{pmatrix}$ be a weak solution and let $T > 0$. Then for all $\tilde{y} \in C^{1}([0,T]; \mathcal{H}^{1})$ such that $\tilde{y}(T) = 0$, we have
	\begin{equation} \label{exuniq_eq_tdepweak}
		\begin{split}
		\int_{0}^{T} \langle y , \partial_{t} \tilde{y} \rangle_{\mathcal{H}^{0}} \diff t &= - \langle y_{0}, \tilde{y}(0) \rangle_{\mathcal{H}^{0}} + \int_{0}^{t}  \langle \nabla y, \nabla \tilde{y} \rangle_{\mathcal{H}^{0}} \diff s \\
		&+ \int_{0}^{t} \langle (\bm{v} \cdot \nabla) \bm{v}, \tilde{\bm{v}} \rangle_{\bm{L}^{2}} \diff s - \int_{0}^{t} \langle (\bm{B} \cdot \nabla) \bm{B}, \tilde{\bm{v}} \rangle_{\bm{L}^{2}} \diff s \\
		&+ \int_{0}^{t} \langle (\bm{v} \cdot \nabla) \bm{B}, \tilde{\bm{B}} \rangle_{\bm{L}^{2}} \diff s - \int_{0}^{t} \langle (\bm{B} \cdot \nabla) \bm{v}, \tilde{\bm{B}} \rangle_{\bm{L}^{2}} \diff s \\
		&+ \int_{0}^{t} \langle g_{N}(|y|^{2})y, \tilde{y} \rangle_{L^{2}} \diff s - \int_{0}^{t} \langle f, \tilde{y} \rangle_{\mathcal{H}^{0}} \diff s.
		\end{split}
	\end{equation}
Moreover, the following energy equality holds:
\begin{equation}\label{exuniq_eq_energyeq}
	\begin{split}
	\| y(t) \|_{\mathcal{H}^{0}}^{2} &+ 2 \int_{0}^{t} \| \nabla y \|_{\mathcal{H}^{0}}^{2} \diff s + 2 \int_{0}^{t} \| \sqrt{g_{N}(|y|^{2} )} |y| \|_{L^{2}}^{2} \diff s \\
	&= \| y_{0} \|_{\mathcal{H}^{0}}^{2} + 2 \int_{0}^{t} \langle f, y \rangle_{\mathcal{H}^{0}} \diff s, \quad \forall t \geq 0.
	\end{split}
\end{equation}
\end{proposition}
\begin{proof}
To prove that the right-hand side of \eqref{exuniq_eq_tdepweak} is well-defined, we proceed as in \cite[Proposition 3.3, pp. 534 f.]{RZ09a}. The energy equality then follows from approximating the solution accordingly, cf. \cite[Lemma 2.7, p. 635]{BDPRS07}.
\end{proof}
\begin{theorem}[Uniqueness of weak solutions]\label{exuniq_thm_uniq} Let $y_{1}, y_{2}$ be two weak solutions in the sense of Definition \ref{exuniq_def_weak}. Then we have $y_{1} = y_{2}$.
\end{theorem}
\begin{proof}
	We fix a $T > 0$ and set $z(t) := y_{1}(t) - y_{2}(t) =: \begin{pmatrix} \bm{v} \\ \bm{B} \end{pmatrix}$. Then we find
	\begin{align*}
		\| z(t) \|_{\mathcal{H}^{0}}^{2} &= -2 \int_{0}^{t} \| \nabla z \|_{\mathcal{H}^{0}}^{2} \diff s - 2 \int_{0}^{t} \langle \bm{v}, (\bm{v}_{1} \cdot \nabla) \bm{v}_{1} - (\bm{v}_{2} \cdot \nabla) \bm{v}_{2} \rangle_{\bm{L}^{2}} \diff s \\
		&+ 2 \int_{0}^{t} \langle \bm{v}, (\bm{B}_{1} \cdot \nabla) \bm{B}_{1} - (\bm{B}_{2} \cdot \nabla) \bm{B}_{2} \rangle_{\bm{L}^{2}} \diff s \\
		&-2 \int_{0}^{t} \langle \bm{B}, (\bm{v}_{1} \cdot \nabla) \bm{B}_{1} - (\bm{v}_{2} \cdot \nabla) \bm{B}_{2} \rangle_{\bm{L}^{2}} \diff s \\
		&+ 2 \int_{0}^{t} \langle \bm{B}, (\bm{B}_{1} \cdot \nabla) \bm{v}_{1} - (\bm{B}_{2} \cdot \nabla) \bm{v}_{2} \rangle_{\bm{L}^{2}} \diff s \\
		&-2 \int_{0}^{t} \langle z, g_{N}(|y_{1}|^{2})y_{1} - g_{N}(|y_{2}|^{2})y_{2} \rangle_{\mathcal{L}^{2}} \diff s \\
		&=: I_{L}(t) + I_{NL}(t) + I_{T}(t).
	\end{align*}
We first investigate the linear term and find, using integration by parts and the definition of the norms $\| \cdot \|_{\mathcal{H}^{1}}$ that
\begin{align*}
	I_{L}(t) = -2 \int_{0}^{t} \| z \|_{\mathcal{H}^{1}}^{2} \diff s - 2 \int_{0}^{t} \| z \|_{\mathcal{H}^{0}}^{2} \diff s.
\end{align*}
The nonlinear term $I_{NL}(t)$ consists of four terms of the same structure which are estimated by Young's inequality:
\begin{align*}
	\langle \bm{\varphi}, (\bm{\psi}_{1} \cdot \nabla) \bm{\theta}_{1} - (\bm{\psi}_{2} \cdot \nabla) \bm{\theta}_{2} \rangle_{\bm{L}^{2}} = - \langle \nabla \bm{\varphi}, \bm{\psi}_{1} \otimes \bm{\theta}_{1} - \bm{\psi}_{2} \otimes \bm{\theta}_{2} \rangle_{\bm{L}^{2}} \leq \frac{1}{4} \| \nabla \bm{\varphi} \|_{\mathbb{H}^{0}}^{2} + \|  \bm{\psi}_{1} \otimes \bm{\theta}_{1} - \bm{\psi}_{2} \otimes \bm{\theta}_{2} \|_{L^{2}}^{2}.
\end{align*}
Thus we need to estimate the latter term which we do by applying  the Sobolev embedding theorem as well as Young's inequality:
\begin{align*}
	 &\|  \bm{\psi}_{1} \otimes \bm{\theta}_{1} - \bm{\psi}_{2} \otimes \bm{\theta}_{2} \|_{L^{2}}^{2} \\
	 &\leq 2 \left( \| \bm{\psi}_{1} - \bm{\psi}_{2} \|_{\bm{L}^{4}}^{2} \| \bm{\theta}_{1} \|_{\bm{L}^{4}}^{2} + \| \bm{\psi}_{2} \|_{\bm{L}^{4}}^{2} \| \bm{\theta}_{1} - \bm{\theta}_{2} \|_{\bm{L}^{4}}^{2}  \right) \\
	 &\leq 2 C_{1,0,2,2,4} \big( \| \bm{\psi}_{1} - \bm{\psi}_{2} \|_{\mathbb{H}^{1}}^{3/2} \| \bm{\psi}_{1} - \bm{\psi}_{2} \|_{\mathbb{H}^{0}}^{1/2} \| \bm{\theta}_{1} \|_{\bm{L}^{4}}^{2} + \| \bm{\psi}_{2} \|_{\bm{L}^{4}}^{2} \| \bm{\theta}_{1} - \bm{\theta}_{2} \|_{\mathbb{H}^{1}}^{3/2} \| \bm{\theta}_{1} - \bm{\theta}_{2} \|_{\mathbb{H}^{0}}^{1/2}  \big) \\
	 &\leq C_{\varepsilon} \left( \| \bm{\psi}_{1} - \bm{\psi}_{2} \|_{\mathbb{H}^{0}}^{2}  \| \bm{\theta}_{1} \|_{\bm{L}^{4}}^{8} + \| \bm{\theta}_{1} - \bm{\theta}_{2} \|_{\mathbb{H}^{0}}^{2}  \| \bm{\psi}_{2} \|_{\bm{L}^{4}}^{8}\right) + \varepsilon \| \bm{\psi}_{1} - \bm{\psi}_{2} \|_{\mathbb{H}^{1}}^{2} + \varepsilon \| \bm{\theta}_{1} - \bm{\theta}_{2} \|_{\mathbb{H}^{1}}^{2}.
\end{align*}
We collect the four terms and use the previous estimates to find (again using the definition of the Sobolev norms)
\begin{align*}
	I_{NL}(t) \leq (1 + 8 \varepsilon)\int_{0}^{t} \| z \|_{\mathcal{H}^{1}}^{2} \diff s + C_{\varepsilon} M_{y_{1},y_{2},t} \int_{0}^{t} \| z \|_{\mathcal{H}^{0}}^{2} \diff s, 
\end{align*}
where, by our definition of weak solutions
$$
	M_{y_{1},y_{2},t} := \esssup_{s \in [0,t]} \left( 1 + \| y_{1} \|_{\mathcal{L}^{4}}^{8} + \| y_{2} \|_{\mathcal{L}^{4}}^{8} \right) < \infty.
$$

Concerning the taming term, $I_{T}(t)$, we have, using the mean-value theorem of calculus (and the fact that $|g'| \leq C$), the inequality $g_{N}(r) \leq C r$, Sobolev's embedding theorem and Young's inequality (for $(p,q) = (4,4/3)$)
\begin{align*}
	&| \langle z, g_{N}(|y_{1}|^{2})y_{1} - g_{N}(|y_{2}|^{2})y_{2} \rangle_{L^{2}} | \leq \int_{\mathbb{R}^{3}} | z |^{2} g_{N}(|y_{1}|^{2}) \diff x + \int_{\mathbb{R}^{3}} | z | | g_{N}(|y_{1}|^{2}) - g_{N}(|y_{2}|^{2}) | | y_{2} | \diff x \\
	&\leq C \int_{\mathbb{R}^{3}} | z |^{2} |y_{1}|^{2} + | z | | |y_{1}|^{2} - |y_{2}|^{2}| |y_{2} | \diff x \\
	&\leq C \int_{\mathbb{R}^{3}} | z |^{2} |y_{1}|^{2} + | z |^{2}  \left( |y_{1}| + |y_{2}| \right) |y_{2} | \diff x \\
	&\leq C \|  z \|_{\mathcal{L}^{4}}^{2} \|  |y_{1}| + |y_{2}| \|_{L^{4}}^{2} \\
	&\leq C \|  z \|_{\mathcal{H}^{1}}^{3/2} \|  z \|_{\mathcal{H}^{0}}^{1/2} \|  |y_{1}| + |y_{2}| \|_{L^{4}}^{2} \\
	&\leq C_{\varepsilon} \| z \|_{\mathcal{H}^{0}}^{2} \left( \| y_{1} \|_{\mathcal{L}^{4}}^{8} + \| y_{2} \|_{\mathcal{L}^{4}}^{8} \right) + \varepsilon \| z \|_{\mathcal{H}^{1}}^{2}.
\end{align*}
This implies that
\begin{align*}
	I_{T}(t) &= -2 \int_{0}^{t} \langle z, g_{N}(|y_{1}|^{2})y_{1} - g_{N}(|y_{2}|^{2})y_{2} \rangle_{\mathcal{L}^{2}} \diff s \\
	&\leq 2\varepsilon \int_{0}^{t} \| z \|_{\mathcal{H}^{1}}^{2} \diff s + C_{\varepsilon} \int_{0}^{t} \| y \|_{\mathcal{H}^{0}}^{2} \left( \| y_{1} \|_{\mathcal{L}^{4}}^{8} + \| y_{2} \|_{\mathcal{L}^{4}}^{8} \right) \diff s \\
	&\leq 2\varepsilon \int_{0}^{t} \| z \|_{\mathcal{H}^{1}}^{2} \diff s + C_{\varepsilon} M_{y_{1},y_{2},t} \int_{0}^{t} \| z \|_{\mathcal{H}^{0}}^{2} \diff s.
\end{align*}
Hence, altogether we have the inequality
\begin{align*}
	\| z(t) \|_{\mathcal{H}^{0}}^{2} &\leq -2 \int_{0}^{t} \| z(s) \|_{\mathcal{H}^{1}}^{2} \diff s + (1 + 10 \varepsilon)\int_{0}^{t} \| z(s) \|_{\mathcal{H}^{1}}^{2} \diff s \\
	&+ C_{\varepsilon} M_{y_{1},y_{2},t} \int_{0}^{t} \| z(s) \|_{\mathcal{H}^{0}}^{2} \diff s.
\end{align*}
Choosing $\varepsilon = \frac{1}{10}$, we find that
\begin{align*}
	\| z(t) \|_{\mathcal{H}^{0}}^{2} \leq C M_{y_{1},y_{2},t} \int_{0}^{t} \| z(s) \|_{\mathcal{H}^{0}}^{2} \diff s
\end{align*}
and Gronwall's lemma implies that $z(s) = 0$ for all $s \in [0,t]$.
\end{proof}

Our next step is to establish existence of weak solutions, i.e., the existence part of Theorem \ref{DTMHD_thm_intro_GWP}.

\begin{proof}[Proof of existence of a weak solution]

We use a Faedo-Galerkin approximation on $[0,T]$. Take an orthonormal basis $\{ e_{k} ~|~ k \in \mathbb{N} \} \subset \mathcal{V}$ of $\mathcal{H}^{1}$ such that $\spann \{ e_{k} \}$ is dense in $\mathcal{H}^{3}$. Fix an $n \in \mathbb{N}$. For $z = \begin{pmatrix} z^{1}, \ldots, z^{n} \end{pmatrix} \in \mathbb{R}^{n}$ and $e = \begin{pmatrix} e_{1}, \ldots, e_{n} \end{pmatrix} \in \mcV^{n}$ set
\begin{align*}
	z \cdot e &:= \sum_{i=1}^{n} z^{i} e_{i} \in \mathcal{V} \\
	b_{n}(z) &:= \left(\langle \mcA(z \cdot e), e_{i} \rangle_{\mcH^{1}} \right)_{i=1}^{n} \\
	f_{n}(t) &:= \left( \langle \rho_{n} * f(t), e_{1} \rangle_{\mathcal{H}^{1}}, \ldots, \langle \rho_{n} * f_{n}(t), e_{n} \rangle_{\mathcal{H}^{1}} \right),
\end{align*}
where the $\rho_{n}$ are a family of mollifiers such that 
$$
	\| \rho_{n} * f(t) \|_{\mathcal{H}^{0}} \leq \| f(t) \|_{\mathcal{H}^{0}}, \quad \lim_{n \rightarrow \infty} \| \rho_{n} * f(t) - f(t) \|_{\mathcal{H}^{0}} = 0.
$$
Now we consider the ordinary differential equation
	\begin{equation}\label{exuniq_eq_Galerkin_ODE}
		\begin{cases} \frac{\diff z_{n}}{\diff t}(t) &= b_{n}(z_{n}(t)) + f_{n}(t), \\
		z_{n}(0) &= \left( \langle y_{0}, e_{i} \rangle_{\mathcal{H}^{1}} \right)_{i=1}^{n}.
		\end{cases}
	\end{equation}
Then we have
\begin{align*}
	\langle z, b_{n}(z) \rangle_{\mathbb{R}^{n}} = \sum_{i=1}^{n} z^{i} \langle \mcA(z \cdot e), e_{i} \rangle_{\mcH^{1}} = \langle \mcA(z \cdot e), z \cdot e \rangle_{\mcH^{1}}.
\end{align*}
Noting that $z \cdot e \in \mathcal{H}^{3} \subset \mathcal{H}^{2}$, the estimate \eqref{ar_eq_main_estimate} on $\langle \mcA(y),y \rangle_{\mcH^{1}}$ from Lemma \ref{ar_thm_estimates} then yields
$$
	\langle z, b_{n}(z) \rangle_{\mathbb{R}^{n}} = \langle \mcA(z \cdot e), z \cdot e \rangle_{\mcH^{1}} \leq C_{n,N} |z|^{2},
$$
where the constant $C_{n,N}$ contains the norms of all the $e_{i}$ for $i=1,\ldots,n$ (as all the terms on the right-hand side of \eqref{ar_eq_main_estimate} are quadratic in $y$). Moreover, since
$$
	z \mapsto b_{n}(z) = \left( \langle \mcA(z \cdot e), e_{i} \rangle_{\mcH^{1}} \right)_{i=1}^{n} \in \mathbb{R}^{n}
$$
is a polynomial in the components of $z$ in each component, it is a smooth map. Hence, the differential equation \eqref{exuniq_eq_Galerkin_ODE} has a unique solution $z_{n}(t)$ such that
$$
	z_{n}(t) = z_{n}(0) + \int_{0}^{t} b_{n}(z_{n}(s)) \diff s + \int_{0}^{t} f_{n}(s) \diff s, \quad t \geq 0.
$$
Now we set
\begin{align*}
	y_{n}(t) &:= z_{n}(t) \cdot e = \sum_{i=1}^{n} z_{n}^{i}(t) e_{i}, \\
	\prod_{n} \mcA(y_{n}(t)) &:= \sum_{i=1}^{n} \langle \mcA(y_{n}(t)), e_{i} \rangle_{\mcH^{1}} e_{i}, \\
	\prod_{n} f(t) &:= \sum_{i=1}^{n} \langle \rho_{n} * f(t), e_{i} \rangle_{\mathcal{H}^{1}} e_{i}.
\end{align*}
Then the function $y_{n}$ satisfies the differential equation
\begin{align*}
	\partial_{t} y_{n}(t) &= \left( \partial_{t} z_{n}(t) \right) \cdot e = \left( b_{n}(z_{n}(t)) \cdot e \right) + \left( f_{n}(t) \cdot e \right) \\
	&= \prod_{n} \mcA(y_{n}(t)) + \prod_{n} f(t)
\end{align*}
and for all $n \geq k$
\begin{equation} \label{exuniq_eq_eqnfinitedim}
	\begin{split}
	\langle y_{n}(t), e_{k} \rangle_{\mathcal{H}^{1}} &= \langle y_{n}(0), e_{k} \rangle_{\mathcal{H}^{1}} \\
	&\quad + \int_{0}^{t} \Big\langle \prod_{n} \mcA(y_{n}(s)), e_{k} \Big\rangle_{\mathcal{H}^{1}} \diff s + \int_{0}^{t} \Big\langle \prod_{n} f(s), e_{k} \Big\rangle_{\mathcal{H}^{1}} \diff s \\
	&= \langle y_{0}, e_{k} \rangle_{\mathcal{H}^{1}} + \int_{0}^{t} \langle \mcA(y_{n}(s)), e_{k} \rangle_{\mcH^{1}} \diff s + \int_{0}^{t} \langle \rho_{n} * f(s), e_{k} \rangle_{\mathcal{H}^{1}} \diff s.
	\end{split}
\end{equation}
This implies that
$$
	\| y_{n}(t) \|_{\mathcal{H}^{1}}^{2} = \| y_{0} \|_{\mathcal{H}^{1}}^{2} + 2 \int_{0}^{t} \langle \mcA(y_{n}(s)), y_{n}(s) \rangle_{\mcH^{1}} \diff s + 2 \int_{0}^{t} \langle \rho_{n} * f(s), y_{n}(s) \rangle_{\mathcal{H}^{1}} \diff s.
$$
Using the definition of $\mathcal{H}^{1}$, the self-adjointness of $(I - \Delta)$ and Young's inequality for the last term as well as \eqref{ar_eq_main_estimate} for the second term (dropping the nonlinear terms, all of which have negative signs), we find that
\begin{align*}
	\| y_{n}(t) \|_{\mathcal{H}^{1}}^{2}
	&\leq \| y_{0} \|_{\mathcal{H}^{1}}^{2} - \int_{0}^{t} \frac{1}{2} \| y_{n} \|_{\mathcal{H}^{2}}^{2} + 2\| y_{n} \|_{\mathcal{H}^{0}}^{2} + 4(N + 1) \| \nabla y_{n} \|_{\mathcal{H}^{0}}^{2} \diff s  + 2 \int_{0}^{t} \| f(s) \|_{\mathcal{H}^{0}}^{2} \diff s.
\end{align*}
This implies
\begin{equation}\label{exuniq_eq_H1H2_est}
	\| y_{n}(t) \|_{\mathcal{H}^{1}}^{2} + \int_{0}^{t} \| y_{n} \|_{\mathcal{H}^{2}}^{2} \diff s \leq C_{N} \left(\| y_{0} \|_{\mathcal{H}^{1}}^{2} + \int_{0}^{t} \| y_{n} \|_{\mathcal{H}^{1}}^{2} \diff s + \int_{0}^{t} \| f \|_{\mathcal{H}^{0}}^{2} \diff s \right).
\end{equation}
Dropping the second term on the left-hand side and using Gronwall's lemma, we find that
$$
	\sup_{t \leq T} \| y_{n}(t) \|_{\mathcal{H}^{1}}^{2} \leq C_{y_{0},N,T,f}.
$$
Using this information in \eqref{exuniq_eq_H1H2_est}, we find that also
$$
	\int_{0}^{t} \| y_{n} \|_{\mathcal{H}^{2}}^{2} \diff s \leq C_{y_{0},N,T,f}.
$$
Now for a fixed $k \in \mathbb{N}$, set $G_{n}^{(k)}(t) := \langle y_{n}(t), e_{k} \rangle_{\mathcal{H}^{1}}$. Then by the preceding step, the $G_{n}^{(k)}$ are uniformly bounded on $[0,T]$. Furthermore, they are equi-continuous, as can be seen from
\begin{align*}
	| G_{n}^{(k)}(t) - G_{n}^{(k)}(r) | &= |\langle y_{n}(t), e_{k} \rangle_{\mathcal{H}^{1}} - \langle y_{n}(r), e_{k} \rangle_{\mathcal{H}^{1}}| \\
	&= \left| \int_{r}^{t} \langle \mcA(y_{n}(s), e_{k} \rangle_{\mcH^{1}} \diff s + \int_{r}^{t} \langle \rho_{n} * f(s), e_{k} \rangle_{\mathcal{H}^{1}} \diff s \right| \\
	&\leq C \int_{r}^{t} (1 + \| y_{n}(s) \|_{\mathcal{H}^{1}}^{3}) \| e_{k} \|_{\mathcal{H}^{3}} \diff s +  \| e_{k} \|_{\mathcal{H}^{2}} \int_{r}^{t} \| f(s) \|_{\mathcal{H}^{0}} \diff s,
\end{align*}
and equation \eqref{exuniq_eq_H1H2_est}, where we used Lemma \ref{ar_thm_estimates} \ref{ar_itm_H3_bound}. Therefore, the theorem of Arzel\`{a}-Ascoli implies that $\left( G_{n}^{(k)} \right)_{n \in \mathbb{N}}$ is sequentially relatively compact with respect to the uniform topology and hence there is a subsequence such that $\left( G_{n_{l}^{k}}^{(k)} \right)_{l \in \mathbb{N}}$ converges uniformly to a limit $G^{(k)}$. Now a diagonalisation argument implies that there is a subsequence, which we denote again by $(G_{n}^{k})_{n \in \N}$
$$
	\forall k \in \mathbb{N}: \quad \lim_{n \rightarrow \infty} \sup_{t \in [0,T]} | G_{n}^{(k)}(t) - G^{(k)}(t) | = 0.
$$
Again invoking \eqref{exuniq_eq_H1H2_est}, we see that $\sup_{n \in \N} \sup_{t \in [0,T]} \| y_{n}(t) \|_{\mathcal{H}^{1}}^{2} \leq C$. Since closed balls of $\mathcal{H}^{1}$ are weakly compact by the Banach-Alaoglu theorem, we find that for almost all $t \in [0,T]$ we have the $\mathcal{H}^{1}$-weak convergence $y_{n}(t) \rightharpoonup y(t)$ as $n \rightarrow \infty$. To conclude that this holds true for all $t \in [0,T]$ we note that $G^{(k)}$, as the uniform limit of continuous functions, is continuous, and that on the other hand by the weak convergence just mentioned,
$$
	G_{n}^{(k)}(t) = \langle y_{n}(t), e_{k} \rangle_{\mathcal{H}^{1}} \rightarrow \langle y(t), e_{k} \rangle_{\mathcal{H}^{1}}.
$$
Hence, $t \mapsto \langle y(t), e_{k} \rangle_{\mathcal{H}^{1}}$ is continuous for all $k \in \mathbb{N}$ and by the density of the $\{ e_{k} \}_{k \in \N}$, we find that $t \mapsto \langle y(t), \tilde{y} \rangle_{\mathcal{H}^{1}}$ is continuous for all $\tilde{y} \in \mathcal{H}^{1}$. We can thus conclude that $t \mapsto y(t)$ is weakly continuous in $\mathcal{H}^{1}$, and that for all $\tilde{y} \in \mathcal{H}^{1}$:
\begin{align*}
	\lim_{n \rightarrow \infty} \sup_{t \in [0,T]} | \langle y_{n}(t) - y(t), \tilde{y} \rangle_{\mathcal{H}^{1}} | = 0.
\end{align*}
This implies (by considering $\tilde{y} = (I - \Delta)^{-1} \tilde{z} \in \mathcal{H}^{2} \subset \mathcal{H}^{1}$ for $z \in \mathcal{H}^{0}$ and using the formal self-adjointness of $(I - \Delta)$)
$$
		\lim_{n \rightarrow \infty} \sup_{t \in [0,T]} | \langle y_{n}(t) - y(t), \tilde{z} \rangle_{\mathcal{H}^{0}} | = 0.
$$
We next invoke the Helmholtz-Weyl decomposition $\mathcal{L}^{2} = \mathcal{H}^{0} \oplus \left( \mathcal{H}^{0} \right)^{\perp}$. Since $y_{n}(t) - y(t) \in \mathcal{H}^{1} \subset \mathcal{H}^{0}$, this allows us to conclude that
\begin{equation}\label{exuniq_eq_unifL2conv}
		\lim_{n \rightarrow \infty} \sup_{t \in [0,T]} | \langle y_{n}(t) - y(t), \tilde{z} \rangle_{\mathcal{L}^{2}} | = 0 \quad \forall \tilde{z} \in \mathcal{L}^{2}.
\end{equation}
To be more precise, we can use the Helmholtz-Weyl decomposition for $L^{2}(\mathbb{R}^{3};\mathbb{R}^{3})$ for both the velocity component and the magnetic field component of $y_{n} - y$, and putting these two together yields \eqref{exuniq_eq_unifL2conv}. Note that this equation implies the component-wise convergence
\begin{equation}\label{exuniq_eq_unifL2conv_compwise}
		\lim_{n \rightarrow \infty} \sup_{t \in [0,T]} | \langle \phi_{i,n}(t) - \phi_{i}(t), \tilde{z} \rangle_{L^{2}} | = 0 \quad \forall z \in L^{2}(\R^{3};\R), \phi \in \{ \bm{v} , \bm{B} \}.
\end{equation}
This can be seen by taking $\tilde{z}$ of the form $\tilde{z} = (0, \ldots, 0, z, 0, \ldots, 0)$, where $z \in L^{2}(\R^{3};\R)$. 

From this equation as well as \eqref{exuniq_eq_H1H2_est} and Fatou's lemma, we get
\begin{align*}
	\int_{0}^{T} \| y(s) \|_{\mathcal{H}^{2}}^{2} \diff s \leq \liminf_{n \rightarrow \infty} \int_{0}^{T} \| y_{n}(s) \|_{\mathcal{H}^{2}}^{2} \diff s < \infty.
\end{align*}
Next we want to show that $y$ is indeed a solution of the tamed MHD equations \eqref{intro_eq_TNSE}.

To this end, recall first the following Friedrichs' inequality\footnote{The first use of the inequality in the context of the Navier-Stokes equations seems to be in E. Hopf \cite[p. 230]{Hopf51}. Hopf uses the inequality and cites R. Courant's and D. Hilbert's book \cite{CH43}. The inequality and a proof can be found in Chapter VII, Paragraph 3, Section 1, Satz 1, p. 489. Hopf also notes that the statement is not true for arbitrary bounded domains. For a more modern presentation, cf. J.C. Robinson, J.L. Rodrigo and W. Sadowski \cite[Exercises 4.2--4.9, pp. 107 ff.]{RRS16}.} (see e.g. \cite{Lady69}): let $Q \subset \mathbb{R}^{3}$ be a bounded cuboid. Then for all $\varepsilon > 0$ there is a $K_{\varepsilon} \in \mathbb{N}$ and functions $h_{i}^{\varepsilon} \in L^{2}(G)$, $i=1, \ldots, K_{\varepsilon}$ such that for all $w \in W_{0}^{1,2}(G)$ 
\begin{equation}\label{DTMHD_exuniq_eq_Friedrichs}
	\int_{Q} | w(x) |^{2} \diff x \leq \sum_{i=1}^{K_{\varepsilon}} \left( \int_{Q} w(x) h_{i}^{\varepsilon}(x) \diff x \right)^{2} + \varepsilon \int_{Q} | \nabla w(x) |^{2} \diff x.
\end{equation}
Now let $G \subset \bar{G} \subset Q$ and choose a smooth cutoff function $\rho$ such that $1 \geq \rho \geq 0$, $\rho \equiv 1$ on $G$ and $\supp \rho \subset Q$. Then we have for all $j = 1, 2,3$ and $\phi \in \{ \bm{v}, \bm{B} \}$ that $\rho (\phi_{j,n}(t, \cdot) - \phi_{j}(t,\cdot)) \in W_{0}^{1,2}(Q)$ and hence by applying Friedrichs' inequality, we find
\begin{align*}
		\int_{G} | y_{n}(t,x) - y(t,x) |^{2} \diff x 
		&\leq \sum_{j=1}^{3} \int_{Q} \rho^{2}(x) | v_{j,n}(t,x) - v_{j}(t,x) |^{2} \diff x + \int_{Q} \rho^{2}(x) | B_{j,n}(t,x) - B_{j}(t,x) |^{2} \diff x \\
	&\leq \sum_{j=1}^{3} \sum_{i=1}^{K_{\varepsilon}} \left( \int_{Q} ( v_{j,n} - v_{j} ) \rho h_{i}^{\varepsilon} \diff x \right)^{2} + \varepsilon \sum_{j=1}^{3} \int_{Q} | \nabla \left[ \rho (v_{j,n} - v_{j} ) \right] |^{2} \diff x \\
		&\quad+ \sum_{j=1}^{3} \sum_{i=1}^{K_{\varepsilon}} \left( \int_{Q} ( B_{j,n} - B_{j} ) \rho h_{i}^{\varepsilon} \diff x \right)^{2} + \varepsilon \sum_{j=1}^{3} \int_{Q} | \nabla \left[ \rho (B_{j,n} - B_{j} ) \right] |^{2} \diff x.
\end{align*}
The first and third terms in the last two lines vanish in the limit $n \rightarrow \infty$ by \eqref{exuniq_eq_unifL2conv_compwise}, since $\rho h_{i}^{\varepsilon} \in L^{2}(\R^{3})$. To the second and fourth term (those proportional to $\varepsilon$), we apply the product rule for weak derivatives (see e.g. \cite[Theorem 5.2.3.1 $(iv)$, pp. 261 f.]{Evans10}) and \eqref{exuniq_eq_H1H2_est} to see that the integrals are bounded. As $\varepsilon > 0$ is arbitrary, we find
\begin{equation} \label{exuniq_eq_convbddset}
	\lim_{n \rightarrow \infty} \sup_{t \in [0,T]} \int_{G} | y_{n}(t,x) - y(t,x) |^{2} \diff x  = 0.
\end{equation}
Now let, for $k \in \mathbb{N}$, $\supp(e_k) \subset G_{k}$ for bounded sets $G_{k}$. If we fix $s \in [0,t]$, then by \eqref{exuniq_eq_H1H2_est} and \eqref{exuniq_eq_convbddset} we get 
\begin{align*}
\sup_{n} \| y_{n}(s, \cdot) \|_{\mathcal{H}^{1}} < \infty \quad \text{and} \quad \lim_{n} \| ( y_{n}(s,\cdot) - y(s,\cdot) ) 1_{G_{k}} \|_{L^{2}} = 0.
\end{align*}
Thus an application of Lemma \ref{ar_thm_convergence} and Lebesgue's dominated convergence theorem yields
\begin{align*}
	\int_{0}^{t} \langle \mcA(y_{n}(s)), e_{k} \rangle_{\mcH^{1}} \diff s \rightarrow \int_{0}^{t} \langle \mcA(y(s)), e_{k} \rangle_{\mcH^{1}} \diff s.
\end{align*}
Having established this convergence, we can take limits $n \rightarrow \infty$ in \eqref{exuniq_eq_eqnfinitedim} to find
\begin{align*}
	\langle y(t), e_{k} \rangle_{\mathcal{H}^{1}} = \langle y_{0}, e_{k} \rangle_{\mathcal{H}^{1}} +  \int_{0}^{t} \langle \mcA(y(s)), e_{k} \rangle_{\mcH^{1}} \diff s +  \int_{0}^{t} \langle f(s), (I - \Delta) e_{k} \rangle_{\mathcal{H}^{0}} \diff s.
\end{align*}
As this equation is linear in $e_{k}$, it holds for linear combinations and since $\spann \{e_{k}\}$ forms a dense subset in $\mathcal{H}^{3}$, we conclude
\begin{align*}
	\langle y(t), \tilde{y} \rangle_{\mathcal{H}^{1}} = \langle y_{0}, \tilde{y} \rangle_{\mathcal{H}^{1}} +  \int_{0}^{t} \langle \mcA(y(s)), \tilde{y} \rangle_{\mcH^{1}} \diff s +  \int_{0}^{t} \langle f(s), (I - \Delta) \tilde{y} \rangle_{\mathcal{H}^{0}} \diff s.
\end{align*}
Now, letting $\tilde{y} := (I - \Delta)^{-1} \bar{y}$ for $\bar{y} \in \mathcal{H}^{3}$,
\begin{align*}
		\langle y(t), \bar{y} \rangle_{\mathcal{H}^{0}} = \langle y_{0}, \bar{y} \rangle_{\mathcal{H}^{0}} +  \int_{0}^{t} \langle \mcA(y(s), \bar{y} \rangle_{\mcH^{0}} \diff s +  \int_{0}^{t} \langle f(s), \bar{y}\rangle_{\mathcal{H}^{0}} \diff s,
\end{align*}
that is, Equation \eqref{exuniq_eq_defweak}.

We are left to prove \ref{exuniq_itm_H0est} - \ref{exuniq_itm_pointwiseeqn}. We will start with \ref{exuniq_itm_pointwiseeqn}. In equation \eqref{exuniq_eq_defweak} we set $\tilde{y} = (\tilde{\bm{v}}, 0)$ and we find for almost all $t \geq 0$ that
\begin{equation}\label{exuniq_eq_projv}
		\frac{\partial \bm{v} }{\partial t} = \Delta \bm{v} - \mathcal{P} \left[ \left( \bm{v} \cdot \nabla \right) \bm{v}  +  \left( \bm{B} \cdot \nabla \right) \bm{B} \right] - \mathcal{P} \left[ g_{N}(|y|^{2}) \bm{v} \right] + \bm{f}_{1}(t)
\end{equation}
and infer from \cite[Proposition 1.1.]{Temam01} the existence of a function $\bar{p}$ with $ \nabla \bar{p} \in L_{\text{loc}}^{2}(\mathbb{R}_{+};L^{2}(\mathbb{R}^{3}) )$ such that (for almost all $t \geq 0$)
\begin{align*}
		\frac{\partial \bm{v} }{\partial t} = \Delta \bm{v} - \left( \bm{v} \cdot \nabla \right) \bm{v}  +  \left( \bm{B} \cdot \nabla \right) \bm{B} - g_{N}(|y|^{2}) \bm{v} + \nabla \bar{p} + \bm{f}_{1}(t).
\end{align*}
Now we \emph{define} the pressure $p$ by
\begin{align*}
	\bar{p} = p + \frac{|\bm{B}|^{2}}{2},
\end{align*}
and we observe that since $\nabla \frac{|\bm{B}|^{2}}{2} = (\bm{B} \cdot \nabla) \bm{B} \in L^{2}_{\text{loc}}(\mathbb{R}_{+}; L^{2}(\mathbb{R}^{3}; \mathbb{R}^{3}))$ due to \eqref{exuniq_eq_soln_H1est}, and by the regularity of $\bar{p}$, $p$ also satisfies the right regularity condition $\nabla p \in L_{\text{loc}}^{2}(\mathbb{R}_{+};L^{2}(\mathbb{R}^{3}) )$.

In the same way (cf. Section \ref{DTMHD_sssec_MPP}), by testing against $\tilde{y} = (0, \tilde{\bm{B}}) \in \mcV$ to get
\begin{equation}\label{exuniq_eq_projB}
			\frac{\partial \bm{B}}{\partial t}	= \Delta \bm{B} - \mcP \left( \bm{v} \cdot \nabla \right) \bm{B} + \mcP (\bm{B} \cdot \nabla) \bm{v}  - \mcP \left[ g_{N}(| (\bm{v}, \bm{B}) |^{2}) \bm{B} \right] + \bm{f}_{2},
\end{equation}
we can find a $\pi$ such that $\nabla \pi \in L_{\text{loc}}^{2}(\mathbb{R}_{+};L^{2}(\mathbb{R}^{3}))$ (note that this function in general cannot be expected to be equal to zero for our equation, as we argued in Section \ref{DTMHD_sssec_MPP}) such that for almost all $t \geq 0$
\begin{align*}
		\frac{\partial \bm{B}}{\partial t}	= \Delta \bm{B} - \left( \bm{v} \cdot \nabla \right) \bm{B} + (\bm{B} \cdot \nabla) \bm{v} + \nabla \pi - g_{N}(| (\bm{v}, \bm{B}) |^{2}) \bm{B} + \bm{f}_{2}.
\end{align*}
These two statements imply \ref{exuniq_itm_pointwiseeqn}.

Next we want to prove \ref{exuniq_itm_H0est}. We take the scalar product in $\mathcal{H}^{0}$ of the system \eqref{exuniq_eq_projv}, \eqref{exuniq_eq_projB} with $y(t)$ as well as \eqref{ar_eq_0eq} to find
\begin{equation}\label{exuniq_eq_H0_est_time_derivative}
	\begin{split}
	\langle \partial_{t} y(t), y(t) \rangle_{\mathcal{H}^{0}} &= \langle \mcA(y(t)), y(t) \rangle_{\mcH^{0}} + \langle f(t), y(t) \rangle_{\mathcal{H}^{0}} \\
	&= - \| \nabla y(t) \|_{\mathcal{H}^{0}}^{2} - \| \sqrt{g_{N}(|y(t)|^{2})} |y(t)| \|_{L^{2}}^{2} + \langle f(t), y(t) \rangle_{\mathcal{H}^{0}} \\
	&\leq - \| \nabla y(t) \|_{\mathcal{H}^{0}}^{2} - \| \sqrt{g_{N}(|y(t)|^{2})} |y(t)| \|_{L^{2}}^{2} + \| f(t) \|_{\mathcal{H}^{0}} \| y(t) \|_{\mathcal{H}^{0}},
	\end{split}
\end{equation}
which yields (as $\langle \partial_{t} y(t), y(t) \rangle_{\mathcal{H}^{0}} = \frac{1}{2} \frac{\diff }{\diff t} \| y(t) \|_{\mathcal{H}^{0}}^{2} = \| y(t) \|_{\mathcal{H}^{0}} \frac{\diff }{\diff t} \| y(t) \|_{\mathcal{H}^{0}}$)
\begin{align*}
	\frac{\diff }{\diff t} \| y(t) \|_{\mathcal{H}^{0}} \leq \| f(t) \|_{\mathcal{H}^{0}}.
\end{align*}
Integrating this inequality gives 
\begin{equation}\label{exuniq_eq_yH0_est}
	\| y(t) \|_{\mathcal{H}^{0}} \leq \| y_{0} \|_{\mathcal{H}^{0}} + \int_{0}^{t} \| f(s) \|_{\mathcal{H}^{0}} \diff s,
\end{equation}
and integrating \eqref{exuniq_eq_H0_est_time_derivative}, we find
\begin{equation}\label{exuniq_eq_yH1tamed_est}
	\begin{split}
	&\int_{0}^{t} \| \nabla y(s) \|_{\mathcal{H}^{0}}^{2} + \| \sqrt{g_{N}(|y(s)|^{2})} |y(s)| \|_{L^{2}}^{2} \diff s \leq \frac{1}{2} \| y_{0} \|_{\mathcal{H}^{0}}^{2} + \int_{0}^{t} \| f(s) \|_{\mathcal{H}^{0}} \| y(s) \|_{\mathcal{H}^{0}} \diff s \\
	&\leq \frac{1}{2} \| y_{0} \|_{\mathcal{H}^{0}}^{2} + \int_{0}^{t} \| f(s) \|_{\mathcal{H}^{0}} \left( \| y_{0} \|_{\mathcal{H}^{0}} + \int_{0}^{s} \| f(r) \|_{\mathcal{H}^{0}} \diff r \right) \diff s \leq \| y_{0} \|_{\mathcal{H}^{0}}^{2} + \frac{3}{2} \left[\int_{0}^{t} \| f(s) \|_{\mathcal{H}^{0}} \diff s \right]^{2}.
	\end{split}
\end{equation}
Thus we have shown \ref{exuniq_itm_H0est}. 

For \ref{exuniq_itm_H1est}, we note that
\begin{align*}
	\mathcal{H}^{2} \hookrightarrow \mathcal{H}^{1} \hookrightarrow \mathcal{H}^{0}
\end{align*}
forms a Gelfand triple and thus by \cite[Chapter III, Section 1, Lemma 1.2, p. 260 f.]{Temam01} and \eqref{exuniq_eq_projv}, \eqref{exuniq_eq_projB} we get the equality
\begin{align*}
	\| y(t) \|_{\mathcal{H}^{1}}^{2} = \| y_{0} \|_{\mathcal{H}^{1}}^{2} + 2 \int_{0}^{t} \langle \mcA(y),y \rangle_{\mcH^{1}} \diff s + 2 \int_{0}^{t} \langle f, y \rangle_{\mathcal{H}^{1}} \diff s.
\end{align*}
The right-hand side is continuous in $t$ and thus together with the weak continuity of $t \mapsto y(t) \in \mathcal{H}^{1}$ by \cite[Proposition 21.23 (d), p. 258]{Zeidler90IIa} we get that $t \mapsto y(t) \in \mathcal{H}^{1}$ is strongly continuous. We then apply \eqref{ar_eq_main_estimate}, \ref{exuniq_itm_H0est} and Young's inequality to find
\begin{align*}
	&\| y(t) \|_{\mathcal{H}^{1}}^{2} \\
	&\leq \| y_{0} \|_{\mathcal{H}^{1}}^{2} - \int_{0}^{t}  \| y \|_{\mathcal{H}^{2}}^{2} \diff s +  2 \int_{0}^{t} \| y \|_{\mathcal{H}^{0}}^{2} \diff s + 2(N + 1) \int_{0}^{t} \| \nabla y \|_{\mathcal{H}^{0}}^{2} \diff s  \\
		& - 2 \int_{0}^{t} \| | \bm{v} | | \nabla \bm{v} | \|_{\bm{L}^{2}}^{2}  - \| | \bm{B} | | \nabla \bm{B} | \|_{\bm{L}^{2}}^{2} - \| | \bm{v}| | \nabla \bm{B} | \|_{\bm{L}^{2}}^{2} - \| | \bm{B} | | \nabla \bm{v} | \|_{\bm{L}^{2}}^{2} \diff s \\
		&  + 2 \int_{0}^{t} \| f \|_{\mathcal{H}^{0}} \| y \|_{\mathcal{H}^{2}} \diff s \\
		&\leq C \left(  \| y_{0} \|_{\mathcal{H}^{1}}^{2} + \int_{0}^{t} \| f \|_{\mathcal{H}^{0}} \diff s \right) +  C(1+ N + t) \left( \| y_{0} \|_{\mathcal{H}^{0}}^{2} + \left[ \int_{0}^{t} \| f \|_{\mathcal{H}^{0}} \diff s \right]^{2} \right)  \\
		& - 2 \int_{0}^{t} \| | \bm{v} | | \nabla \bm{v} | \|_{\bm{L}^{2}}^{2}  - \| | \bm{B} | | \nabla \bm{B} | \|_{\bm{L}^{2}}^{2} - \| | \bm{v}| | \nabla \bm{B} | \|_{\bm{L}^{2}}^{2} - \| | \bm{B} | | \nabla \bm{v} | \|_{\bm{L}^{2}}^{2} \diff s \\
		& - \frac{1}{2} \int_{0}^{t}  \| y \|_{\mathcal{H}^{2}}^{2} \diff s.
\end{align*}
Hence we can conclude that
\begin{align*}
	&\| y(t) \|_{\mathcal{H}^{1}}^{2} + \frac{1}{2} \int_{0}^{t}  \| y \|_{\mathcal{H}^{2}}^{2} \diff s \\
	&\quad + 2 \int_{0}^{t} \| | \bm{v} | | \nabla \bm{v} | \|_{\bm{L}^{2}}^{2}  + \| | \bm{B} | | \nabla \bm{B} | \|_{\bm{L}^{2}}^{2} + \| | \bm{v}| | \nabla \bm{B} | \|_{\bm{L}^{2}}^{2} + \| | \bm{B} | | \nabla \bm{v} | \|_{\bm{L}^{2}}^{2} \diff s \\
	&\leq C \left(  \| y_{0} \|_{\mathcal{H}^{1}}^{2} + \int_{0}^{t} \| f \|_{\mathcal{H}^{0}} \diff s \right) +  C(1+ N + t) \left( \| y_{0} \|_{\mathcal{H}^{0}}^{2} + \left[ \int_{0}^{t} \| f \|_{\mathcal{H}^{0}} \diff s \right]^{2} \right),
\end{align*}
which implies \eqref{exuniq_eq_soln_H1est}.
\end{proof}

\subsection{Existence, Uniqueness and Regularity of a Strong Solution}\label{sec:DTMHD_regul}
In this section, we will show that for smooth initital data, the TMHD equations admit a smooth solution, i.e., Theorem \ref{DTMHD_thm_reg}. To prove this, we have to prove their regularity, which is done via the regularity result of Appendix \ref{chap:FJR}. 
We use the notation from Appendix \ref{chap:FJR}. We denote the space-time $L^{p}$ norms by $\| y \|_{\mathcal{L}^{p}(S_{T})}$. Let $(\mathcal{F}_{t})_{t \geq 0}$ be the Gaussian heat semigroup on $\mathbb{R}^{3}$. We define its action on a function by the space-convolution
\begin{align*}
	\mathcal{F}_{t} h(x) := \frac{1}{(4 \pi t)^{3/2}} \int_{\mathbb{R}^{3}} e^{- \frac{|x-z|^2}{4t}} h(z) \diff z.
\end{align*}
We can thus rewrite the operator $\bar{\mathcal{B}}$ as
\begin{align*}
	\bar{\mathcal{B}}(u,u)_{i}(t,x) = \sum_{j=1}^{3} \int_{0}^{t} \left(D_{x_{j}} \mathcal{F}_{t-s} \right) \left[ u^{j}(s) u_{i}(s) - \sum_{k} R_{i} R_{k} u^{k}(s) u^{j}(s) \right](x) \diff s.
\end{align*}

Then by Appendix \ref{chap:FJR}, Theorem \ref{AppA_thm_equiv}, the weak solution $y$ constructed in Theorem \ref{DTMHD_thm_intro_GWP} satisfies the integral equation
\begin{align*}
	y(t,x) = f_{N}(t,x) - \mathcal{B}(y,y)(t,x),
\end{align*}
where
\begin{align*}
	f_{N}(t,x) &:= \mathcal{F}_{t} y_{0} - \Gamma \bar{\oast} \left( 1_{\mathbb{R}_{+}} \mathcal{P} (g_{N}(|y|^{2})y - f) \right)(t,x) \\
		&:= \mathcal{F}_{t} y_{0} - \int_{0}^{t} \mathcal{F}_{t-s} \left( \mathcal{P} (g_{N}(|y(s)|^{2})y(s) - f(s)) \right)(x) \diff s.
\end{align*}
The Riesz projection term vanishes here because the Helmholtz-Leray projection $\mathcal{P}$ ensures that the divergence of the taming term is zero, and the forcing term has zero divergence by assumption.

\begin{proof}[Proof of Theorem \ref{DTMHD_thm_reg}]
	We first prove the regularity statement. To this end, we show the following for all $k \in \mathbb{N}$:
	\begin{equation}\label{exuniq_eq_induction_hypo}
		y \in \mathcal{L}^{10 \cdot (\frac{5}{3})^{k-1}}(S_{T}), D_{x}^{\alpha} D_{t}^{j} y \in \mathcal{L}^{2 \cdot (\frac{5}{3})^{k}}(S_{T}), \quad | \alpha | + 2j \leq 2.
	\end{equation}
We use a proof by induction.
\newline
$\underline{k=1}$. First, by the Sobolev embedding theorem, we have
\begin{align*}
	\| y \|_{\mathcal{L}^{10}(S_{T})}^{10}& = \int_{0}^{T} \int_{\mathbb{R}^{3}} | y(s,x) |^{10} \diff x \diff s \leq C_{2,0,2,6,10}^{10} \int_{0}^{T} \left( \| y \|_{\mathcal{H}^{2}}^{1/5} \| y \|_{\mathcal{L}^{6}}^{4/5} \right)^{10} \diff s \\
	&= C_{2,0,2,6,10}^{10} \int_{0}^{T} \| y \|_{\mathcal{H}^{2}}^{2} \| y \|_{\mathcal{L}^{6}}^{8} \diff s \leq C \int_{0}^{T} \| y \|_{\mathcal{H}^{2}}^{2} \| y \|_{\mathcal{H}^{1}}^{8} \diff s \\
	&\leq C \sup_{t \in [0,T]} \| y(t) \|_{\mathcal{H}^{1}}^{8}  \int_{0}^{T} \| y(s) \|_{\mathcal{H}^{2}}^{2}  \diff s < \infty.
\end{align*}
Hence we find that 
\begin{align*}
	g_{N}(|y|^{2})y \in \mathcal{L}^{10/3}(S_{T}).
\end{align*}
Now, as $y_{0}$ and $f$ are smooth, by Lemma \ref{AppA_thm_keylemma}, we find that
\begin{align*}
	D_{x}^{\alpha} D_{t}^{j} f_{N} \in \mcL^{10/3}(S_{T}), \quad | \alpha| + 2j \leq 2.
\end{align*}
An application of Theorem \ref{AppA_thm_reg} then yields
\begin{align*}
	D_{x}^{\alpha} D_{t}^{j} y \in \mcL^{10/3}(S_{T}), \quad | \alpha| + 2j \leq 2.	
\end{align*}
\newline
$\underline{k \rightarrow k+1}$. Assume \eqref{exuniq_eq_induction_hypo}. We want to apply the Sobolev embedding theorem, which is justified as
\begin{align*}
	\frac{0}{3} + \frac{1}{5} \left( \frac{1}{2 \cdot \left( \frac{5}{3}\right)^{k}} - \frac{2}{3} \right)  + \frac{1 - 1/5}{6} = \frac{1}{10 \cdot \left( \frac{5}{3}\right)^{k}}.
\end{align*}
Therefore,
\begin{align*}
\| y \|_{\mathcal{L}^{10 \cdot \left( \frac{5}{3} \right)^{k} }}^{10 \cdot \left( \frac{5}{3} \right)^{k}} &\leq C^{10 \cdot \left( \frac{5}{3} \right)^{k}}_{2,0,2 \cdot \left( \frac{5}{3} \right)^{k},6,10 \cdot \left( \frac{5}{3} \right)^{k}} \int_{0}^{T} \left( \| y \|_{2,2 \cdot \left( \frac{5}{3} \right)^{k}}^{1/5} \| y \|_{\mathcal{L}^{6}}^{4/5} \right)^{10 \cdot \left( \frac{5}{3} \right)^{k}} \diff s \\
	&\leq C \int_{0}^{T} \| y \|_{2,2 \cdot \left( \frac{5}{3} \right)^{k}}^{2 \cdot \left( \frac{5}{3} \right)^{k}} \| y \|_{\mathcal{H}^{1}}^{8 \cdot \left( \frac{5}{3} \right)^{k}} \diff s \\
	&\leq C \sup_{t \in [0,T]} \| y(t) \|_{\mathcal{H}^{1}}^{8 \cdot \left( \frac{5}{3} \right)^{k}} \int_{0}^{T} \| y \|_{2,2 \cdot \left( \frac{5}{3} \right)^{k}}^{2 \cdot \left( \frac{5}{3} \right)^{k}} \diff s < \infty,
\end{align*}
which implies
\begin{align*}
	g_{N}(|y|^{2})y \in \mathcal{L}^{2 \cdot \left( \frac{5}{3} \right)^{k+1}}(S_{T})
\end{align*}
and by another application of Lemma \ref{AppA_thm_keylemma}, this yields
\begin{align*}
	D_{x}^{\alpha} D_{t}^{j} f_{N} \in L^{2 \cdot \left( \frac{5}{3} \right)^{k+1}}(S_{T}), \quad | \alpha| + 2j \leq 2,
\end{align*}
and hence, by Theorem \ref{AppA_thm_reg},
\begin{align*}
	D_{x}^{\alpha} D_{t}^{j} y \in L^{2 \cdot \left( \frac{5}{3} \right)^{k+1}}(S_{T}), \quad | \alpha| + 2j \leq 2.
\end{align*}
We have thus shown that 
\begin{align*}
	D_{x}^{\alpha} D_{t}^{j} f_{N} \in \bigcap_{q > 1} L^{q}(S_{T}), \quad | \alpha| + 2j \leq 2.
\end{align*}
The next step of the proof consists of another induction on the number of derivatives. Namely we want to show that
\begin{align*}
	D_{x}^{\alpha} D_{t}^{j} f_{N} \in \bigcap_{q > 1} L^{q}(S_{T}), \quad | \alpha| + 2j \leq m.
\end{align*}
We have shown the base case $m=2$ already. So we are left to show the induction step $m \rightarrow m+1$.

There are two cases to consider:
\begin{enumerate}[label=(\alph*), ref=(\alph*)]
	\item There is at least one spatial derivative, i.e., we have
	\begin{align*}
		D_{x}^{\alpha} D_{t}^{j} f_{N} = D_{x_{k}} D_{x}^{\beta} D_{t}^{j} f_{N}, \quad |\beta| = |\alpha| - 1 > 0, \quad |\beta| + 1 + 2j = m+1.
	\end{align*}
	In this case we have
	\begin{align*}
		\| D_{x_{k}} D_{x}^{\beta} D_{t}^{j} f_{N} \|_{L^{q}(S_{T})} = \left\|  D_{x_{k}} D_{x}^{\beta} D_{t}^{j} \left( \mathcal{F}_{t} y_{0} - \int_{0}^{t} \mathcal{F}_{t-s} \left( \mathcal{P} (g_{N}(|y(s)|^{2})y(s) - f(s)) \right)(x) \diff s \right) \right\|_{L^{q}}.
	\end{align*}
	Applying linearity and the triangle inequality, we see that the term containing the initial condition $y_{0}$ is bounded. For the other term we get the upper bound
	\begin{align*}
		&\left\|  D_{x_{k}} D_{x}^{\beta} D_{t}^{j} \int_{0}^{t} \mathcal{F}_{t-s} \left( \mathcal{P} (g_{N}(|y(s)|^{2})y(s) - f(s)) \right) \diff s \right\|_{L^{q}} \\
		&\leq\left\|  D_{t}^{j-1} D_{x_{k}} D_{x}^{\beta} \mathcal{P} (g_{N}(|y(t)|^{2})y(t) - f(t)) \right\|_{L^{q}} \\ 
		&+ \left\| \int_{0}^{t} \left( D_{x_{k}} \mathcal{F}_{t-s} \right)  D_{x}^{\beta} D_{t}^{j} \left( \mathcal{P} (g_{N}(|y(s)|^{2})y(s) - f(s)) \right) \diff s \right\|_{L^{q}}.
	\end{align*}
	The first term is bounded by the induction hypothesis, since 
	$$ |\beta| + 1 + 2(j-1) = |\alpha| -1 + 2j = m.$$
	The second term is bounded by Young's convolution inequality and the fact that $D_{x_{k}} \mathcal{F}_{t} \in L^{1}(S_{T})$ just like in the proof of Lemma \ref{AppA_thm_Bwelldef}.
	\item There are only derivatives with respect to time, i.e.,
	\begin{align*}
			D_{x}^{\alpha} D_{t}^{j} f_{N} = D_{t}^{j} f_{N}, \quad 2j = m+1.
	\end{align*}
	The term containing the initial condition is again not a problem. In a similar way as before we find
	\begin{align*}
		&\left\|  D_{t}^{j} \int_{0}^{t} \mathcal{F}_{t-s} \left( \mathcal{P} (g_{N}(|y(s)|^{2})y(s) - f(s)) \right) \diff s \right\|_{L^{q}} \\
		&\leq \left\|  D_{t}^{j-1} \mathcal{P} (g_{N}(|y(t)|^{2})y(t) - f(t)) \right\|_{L^{q}} \\ 
		&+ \left\| \int_{0}^{t} \left( \Delta \mathcal{F}_{t-s} \right) D_{t}^{j-1} \left( \mathcal{P} (g_{N}(|y(s)|^{2})y(s) - f(s)) \right) \diff s \right\|_{L^{q}},
	\end{align*}
	where we have used that $\mathcal{F}_{t}$ solves the heat equation. Now by using integration by parts, in the last term, we transfer one spatial derivative from the Laplacian to the second factor and since $2j - 1 = m$, we conclude the boundedness as before by Young's convolution inequality and the $L^{1}$-boundedness of $D_{x_{k}} \mathcal{F}_{t-s}$.
\end{enumerate}
By the Sobolev embedding theorem we now find that $y$ is smooth. Thus we get for every $t \geq 0$ that
\begin{equation}\label{exuniq_eq_eqn_all_t}
	\begin{split}
	\partial_{t} \bm{v}(t) &= \Delta \bm{v}(t) - \mathcal{P}((\bm{v} \cdot \nabla ) \bm{v}) + \mathcal{P}((\bm{B} \cdot \nabla ) \bm{B}) - \mathcal{P}( g_{N}(| \bm{v}|^{2} ) \bm{v}) + \bm{f}_{v}(t) \\
		\partial_{t} \bm{B}(t) &= \Delta \bm{B}(t) - \mathcal{P}((\bm{v} \cdot \nabla ) \bm{B}) + \mathcal{P}((\bm{B} \cdot \nabla ) \bm{v}) - \mathcal{P}( g_{N}(| \bm{B}|^{2} ) \bm{B}) + \bm{f}_{B}(t)
	\end{split}
\end{equation}
We take this equation and apply 3 different inner products to both sides of the equations:
\begin{enumerate}[label=(\alph*), ref=(\alph*)]
	\item $\langle \cdot, \partial_{t} y(t) \rangle_{\mathcal{H}^{0}}$, which will lead to an estimate for $\int_{0}^{T} \| \partial_{t} y \|_{\mathcal{H}^{0}}^{2} \diff t$
	\item First apply $\partial_{t}$, then apply $\langle \cdot, \partial_{t} y(t) \rangle_{\mathcal{H}^{0}}$. This will lead to an estimate for $\| \partial_{t} y \|_{\mathcal{H}^{0}}^{2}$.
	\item $\langle \cdot, y(t) \rangle_{\mathcal{H}^{1}}$, which gives an estimate for $ \|  y(t) \|_{\mathcal{H}^{2}}^{2}$.
\end{enumerate}
\begin{enumerate}[label=(\alph*), ref=(\alph*)]
	\item We find by using Young's inequality
	\begin{align*}
		\| \partial_{t} y \|_{\mathcal{H}^{0}}^{2} &= - \langle \nabla y, \partial_{t} \nabla y(t) \rangle_{\mathcal{H}^{0}} - \langle (\bm{v} \cdot \nabla ) \bm{v}, \partial_{t} \bm{v} \rangle_{L^{2}} + \langle (\bm{B} \cdot \nabla ) \bm{B}, \partial_{t} \bm{v} \rangle_{L^{2}} \\
		&\quad - \langle (\bm{v} \cdot \nabla ) \bm{B}, \partial_{t} \bm{B} \rangle_{L^{2}} + \langle (\bm{B} \cdot \nabla ) \bm{v}, \partial_{t} \bm{B} \rangle_{L^{2}} - \int g_{N}(|y|^{2}) \frac{1}{2} \partial_{t} |y|^{2} \diff x  + \langle f, \partial_{t} y \rangle_{\mathcal{H}^{0}} \\
		&= - \frac{1}{2} \frac{\diff}{\diff t} \| \nabla y \|_{\mathcal{H}^{0}} - \langle (\bm{v} \cdot \nabla ) \bm{v}, \partial_{t} \bm{v} \rangle_{L^{2}} + \langle (\bm{B} \cdot \nabla ) \bm{B}, \partial_{t} \bm{v} \rangle_{L^{2}} \\
		&\quad - \langle (\bm{v} \cdot \nabla ) \bm{B}, \partial_{t} \bm{B} \rangle_{L^{2}} + \langle (\bm{B} \cdot \nabla ) \bm{v}, \partial_{t} \bm{B} \rangle_{L^{2}} - \frac{1}{2} \frac{\diff }{\diff t} \| G_{N}(|y|^{2}) \|_{L^{1}} + \langle f, \partial_{t} y \rangle_{\mathcal{H}^{0}} \\
		&\leq  - \frac{1}{2} \frac{\diff }{\diff t} \| \nabla y \|_{\mathcal{H}^{0}} + \frac{1}{8} \left( 2 \| \partial_{t} \bm{v}(t) \|_{\mathcal{H}^{0}}^{2} + 2 \| \partial_{t} \bm{B}(t) \|_{\mathcal{H}^{0}}^{2} \right) \\
		&\quad + 2 \left( \| | \bm{v} | | \nabla \bm{v} | \|_{L^{2}}^{2} + \| | \bm{B} | | \nabla \bm{B} | \|_{L^{2}}^{2} + \| | \bm{v} | | \nabla \bm{B} | \|_{L^{2}}^{2} + \| | \bm{B} | | \nabla \bm{v} | \|_{L^{2}}^{2} \right) \\
		&\quad + \frac{1}{4} \| \partial_{t} y(t) \|_{\mathcal{H}^{0}}^{2} + \| f(t) \|_{\mathcal{H}^{0}}^{2} - \frac{1}{2} \frac{\diff }{\diff t} \| G_{N}(|y|^{2}) \|_{L^{1}} \\
		&= - \frac{1}{2} \frac{\diff }{\diff t} \| \nabla y \|_{\mathcal{H}^{0}} + \frac{1}{2} \| \partial_{t} y(t) \|_{\mathcal{H}^{0}}^{2} + 2 \| |y| \cdot | \nabla y| \|_{L^{2}}^{2} + \| f(t) \|_{\mathcal{H}^{0}}^{2} - \frac{1}{2} \frac{\diff }{\diff t} \| G_{N}(|y|^{2}) \|_{L^{1}}.
	\end{align*}
	Here, we denote by $G_{N}$ a primitive function of $g_{N}$. Since $g_{N}(r) \leq 2 r$, we find $0 \leq G_{N}(r) := \int_{0}^{r} g_{N}(s) \diff s \leq 2 \frac{r^{2}}{2} = r^{2}$.
	Integrating from 0 to $T$ yields -- estimating the nonpositive terms by zero -- the following:
	\begin{equation}\label{exuniq_eq_ineq_1}
		\begin{split}
		&\frac{1}{2} \int_{0}^{T} \| \partial_{t} y(t) \|_{\mathcal{H}^{0}}^{2} \diff t \leq \frac{1}{2} \| \nabla y(0) \|_{\mathcal{H}^{0}}^{2} + \int_{0}^{T} \left( 2 \| |y| \cdot | \nabla y| \|_{L^{2}}^{2} + \| f \|_{\mathcal{H}^{0}}^{2}  \right) \diff t  + \frac{1}{2} \| G_{N}(|y(0)|^{2}) \|_{L^{1}} \\
		&\leq \frac{1}{2} \| \nabla y(0) \|_{\mathcal{H}^{0}}^{2} + 2 T C_{T,N,y_{0},f}^{(1)} + \int_{0}^{T} \| f \|_{\mathcal{H}^{0}}^{2} \diff t + \frac{1}{2} \| y_{0} \|_{\mathcal{L}^{4}}^{4} =: C_{T,N,y_{0},f}^{(2)}.
		\end{split}
	\end{equation}
	\item We first differentiate Equation \eqref{exuniq_eq_eqn_all_t} with respect to $t$ and then take the inner product with $\partial_{t} y$ in $\mathcal{H}^{0}$. Note that for $\bm{\theta}, \bm{\phi},\bm{\psi} \in \{ \bm{v}, \bm{B} \}$, we get
	\begin{equation} \label{exuniq_eq_time_deriv_NL}
		\langle \partial_{t} ( (\bm{\theta} \cdot \nabla) \bm{\phi} ), \partial_{t} \bm{\psi} \rangle_{L^{2}} = \langle ( \partial_{t} \bm{\theta} \cdot \nabla) \bm{\phi} , \partial_{t} \bm{\psi} \rangle_{L^{2}} + \langle  (\bm{\theta} \cdot \nabla) \partial_{t} \bm{\phi} , \partial_{t} \bm{\psi} \rangle_{L^{2}}.
	\end{equation}
	By the (anti-)symmetry of the nonlinear terms, if $\bm{\phi} = \bm{\psi}$, the second term vanishes, which accounts for the $(\bm{v} \cdot \nabla) \bm{v}$ and $(\bm{B} \cdot \nabla) \bm{B}$ terms. 	
	The other two nonlinear terms cancel each other, so we are left with four variations of the first term of the right-hand side of Equation \eqref{exuniq_eq_time_deriv_NL}, which can be simplified further using the divergence-freeness to yield
	\begin{align*}
		\langle ( \partial_{t} \bm{\theta} \cdot \nabla) \bm{\phi} , \partial_{t} \bm{\psi} \rangle_{L^{2}} =  \langle \nabla \cdot ( \partial_{t} \bm{\theta} \otimes \bm{\phi} ) , \partial_{t} \bm{\psi} \rangle_{L^{2}} = - \langle  \partial_{t} \bm{\theta} \otimes \bm{\phi} , \partial_{t} \nabla \bm{\psi} \rangle_{L^{2}}.
	\end{align*}
	Taking this into account and applying Young's inequality, we find
	\begin{align*}
		&\frac{1}{2} \frac{\diff }{\diff t} \| \partial_{t} y(t) \|_{\mathcal{H}^{0}}^{2} \\
		&= - \| \partial_{t} \nabla y(t)\|_{\mathcal{H}^{0}}^{2} + \langle \partial_{t} f, \partial_{t} y \rangle_{\mathcal{H}^{0}} \\
		&\quad+ \langle \partial_{t} \bm{v} \otimes \bm{v}, \partial_{t} \nabla \bm{v} \rangle_{L^{2}} - \langle \partial_{t} \bm{B} \otimes \bm{B}, \partial_{t} \nabla \bm{v} \rangle_{L^{2}} \\
		&\quad+ \langle \partial_{t} \bm{v} \otimes \bm{B}, \partial_{t} \nabla \bm{B} \rangle_{L^{2}}  - \langle \partial_{t} \bm{B} \otimes \bm{v}, \partial_{t} \nabla \bm{B} \rangle_{L^{2}}  \\
		&\quad- \| \sqrt{g_{N}(|y|^{2})} ~ |\partial_{t} y|~ \|_{L^{2}}^{2} - \| \sqrt{g_{N}'(|y|^{2}) } ~ | \partial_{t} |y|^{2} | ~ \|_{L^{2}}^{2} \\
		&\leq - \| \partial_{t} \nabla y(t)\|_{\mathcal{H}^{0}}^{2} + \frac{1}{4} \| \partial_{t} f \|_{\mathcal{H}^{0}}^{2} + \|\partial_{t} y \|_{\mathcal{H}^{0}}^{2} \\
		&\quad+ \| | \bm{v} | | \partial_{t} \bm{v} | \|_{L^{2}}^{2} + \| | \bm{B} | | \partial_{t} \bm{B} | \|_{L^{2}}^{2}  + \| | \bm{v} | | \partial_{t} \bm{B} | \|_{L^{2}}^{2} + \| | \bm{B} | | \partial_{t} \bm{v} | \|_{L^{2}}^{2} \\
		&\quad+ \frac{1}{4} \left( 2 \| \partial_{t}\nabla \bm{v} \|_{\mathcal{H}^{0}}^{2} + 2 \| \partial_{t}\nabla \bm{B} \|_{\mathcal{H}^{0}}^{2} \right) - \| \sqrt{g_{N}(|y|^{2})} ~ |\partial_{t} y|~ \|_{L^{2}}^{2} \\
		&\leq - \frac{1}{2} \| \partial_{t} \nabla y(t)\|_{\mathcal{H}^{0}}^{2} - \| |y| |\partial_{t} y| \|_{L^{2}}^{2} + \frac{1}{4} \| \partial_{t} f \|_{\mathcal{H}^{0}}^{2} + 2 (N+1) \|\partial_{t} y \|_{\mathcal{H}^{0}}^{2}.
	\end{align*}
	Integrating from 0 to $t \leq T$ then gives (again estimating non-positive terms by zero)
	\begin{equation}\label{exuniq_eq_ineq_2}
		\begin{split}
		&\| \partial_{t} y(t) \|_{\mathcal{H}^{0}}^{2} \leq \| \partial_{t} y(0) \|_{\mathcal{H}^{0}}^{2} + 4 (N+1) \int_{0}^{T} \| \partial_{s} y(s) \|_{\mathcal{H}^{0}}^{2} \diff s + \frac{1}{2} \int_{0}^{T} \| \partial_{s} f(s) \|_{\mathcal{H}^{0}}^{2} \diff s \\
		&\leq C(1+ \| y_{0} \|_{\mathcal{H}^{2}}^{6} + \| f(0) \|_{\mathcal{H}^{0}}^{2} ) + 8 (N+1)C_{T,N,y_{0},f}^{(2)} + \frac{1}{2} \int_{0}^{T} \| \partial_{s} f \|_{\mathcal{H}^{0}}^{2} \diff s =: C_{T,N,y_{0},f}^{(3)}.
		\end{split}
	\end{equation}
	Here we used \eqref{exuniq_eq_ineq_1} as well as the following estimate for the time derivative of the initial condition: since \eqref{exuniq_eq_eqn_all_t} holds for all $t$, we can set $t=0$ there to and take the $\mathcal{H}^{0}$-norm to find
	\begin{align*}
		\| \partial_{t} y_{0} \|_{\mathcal{H}^{0}}^{2} &\leq C \left( \| y_{0} \|_{\mathcal{H}^{2}}^{2} + \| |y_{0}| \|_{\mathcal{H}^{0}}^{2} \| |\nabla y_{0}| \|_{\mathcal{H}^{0}}^{2} + \| y_{0} \|_{\mathcal{H}^{2}}^{6} + \| f(0) \|_{\mathcal{H}^{0}}^{2} \right) \\
		&\leq C \left( \| y_{0} \|_{\mathcal{H}^{2}}^{2} + \| |y_{0}| \|_{\mathcal{H}^{0}}^{4} + \| |\nabla y_{0}| \|_{\mathcal{H}^{0}}^{4} + \| y_{0} \|_{\mathcal{H}^{2}}^{6} + \| f(0) \|_{\mathcal{H}^{0}}^{2} \right) \\
		&\leq C \left( 1 + \| y_{0} \|_{\mathcal{H}^{2}}^{6} + \| f(0) \|_{\mathcal{H}^{0}}^{2} \right).
	\end{align*}
	\item Finally, we take the $\mathcal{H}^{1}$ inner product with $y(t)$ and use Equation  \eqref{ar_eq_main_estimate}:
	\begin{align*}
		\langle \partial_{t} y(t), y(t) \rangle_{\mathcal{H}^{1}} 
		&\leq - \frac{1}{2} \| y \|_{\mathcal{H}^{2}}^{2} + \| y \|_{\mathcal{H}^{0}}^{2} + 2(N+1)\| \nabla y \|_{\mathcal{H}^{0}}^{2} - \| | \bm{v} | | \nabla \bm{v} | \|_{\bm{L}^{2}}^{2}  - \| | \bm{B} | | \nabla \bm{B} | \|_{\bm{L}^{2}}^{2} \\
		&\quad - \| | \bm{v}| | \nabla \bm{B} | \|_{\bm{L}^{2}}^{2} - \| | \bm{B} | | \nabla \bm{v} | \|_{\bm{L}^{2}}^{2} + \frac{1}{4} \| y \|_{\mathcal{H}^{2}}^{2} + \| f \|_{\mathcal{H}^{0}}^{2},
	\end{align*}
	which implies
	\begin{align*}
		&\| y(t) \|_{\mathcal{H}^{2}}^{2} \leq 4 \| y(t) \|_{\mathcal{H}^{0}}^{2} + 8(N+1)\| \nabla y(t) \|_{\mathcal{H}^{0}}^{2} + 4 \| f(t) \|_{\mathcal{H}^{0}}^{2} + 8 \| \partial_{t} y(t) \|_{\mathcal{H}^{0}}^{2} +  \frac{1}{2} \| y(t) \|_{\mathcal{H}^{2}}^{2},
	\end{align*}
	and hence, using \eqref{exuniq_eq_ineq_2} and \eqref{exuniq_eq_ssoln_H0est},
	\begin{align*}
		\sup_{t \in [0,T]} \| y(t) \|_{\mathcal{H}^{2}}^{2} &\leq 16(N+1) \sup_{t \in [0,T]} \| \nabla y(t) \|_{\mathcal{H}^{0}}^{2} + 8 \sup_{t \in [0,T]} \| y(t) \|_{\mathcal{H}^{0}}^{2} \\
		&+ 8 \sup_{t \in [0,T]} \| f(t) \|_{\mathcal{H}^{0}}^{2} + 16 \sup_{t \in [0,T]} \| \partial_{t} y(t) \|_{\mathcal{H}^{0}}^{2} \\
		&\leq C(N+1) C_{T,N,y_{0},f}^{(1)} + 4 \left[ \| y_{0} \|_{\mathcal{H}^{0}} + \int_{0}^{T} \| f \|_{\mathcal{H}^{0}} \diff s \right]^{2} \\
		&+ 8 \sup_{t \in [0,T]} \| f(t) \|_{\mathcal{H}^{0}}^{2} + 16 C_{T,N,y_{0},f}^{(3)} =: C_{T,N,y_{0},f}' + C_{T,N,y_{0},f}(1 + N^{2}),
	\end{align*}
	i.e., \eqref{exuniq_eq_ssoln_H2est}. Equation \eqref{exuniq_eq_ssoln_H0est} follows from \eqref{exuniq_eq_soln_H1dot_int_est}, and Equation \eqref{exuniq_eq_ssoln_H1est} follows from \eqref{exuniq_eq_soln_H1est}. This concludes the proof.
\end{enumerate}
\end{proof}

\subsection{Convergence to the Untamed MHD Equations}\label{sec:DTMHD_conv}
In this section we stress the dependence of the solution to the tamed equation on $N$ by writing $y_{N}$. We will prove that as $N \rightarrow \infty$, the solutions to the tamed equations converge to weak solutions of the untamed equations, i.e., Theorem \ref{DTMHD_thm_conv}. 
\begin{proof}
	The proof follows along the same lines as that of Theorem 1.2 in \cite{RZ09a}. 
	
	Let $y_{0}^{N} \in \mathcal{H}^{1}$ with $y_{0}^{N} \rightarrow y_{0}$ in $\mathcal{H}^{0}$ and $(y_{N},p_{N})$ be the associated unique strong solution given by Theorem \ref{DTMHD_thm_intro_GWP}. Combining \eqref{exuniq_eq_yH0_est} with \eqref{exuniq_eq_yH1tamed_est} yields
	\begin{equation} \label{conv_eq_H0H1est}
		\sup_{t \in [0,T]} \| y_{N}(t) \|_{\mathcal{H}^{0}}^{2} + \int_{0}^{T} \| y_{N} \|_{\mathcal{H}^{1}}^{2} + \| \sqrt{g_{N}(|y|^{2})}|y| \|_{L^{2}}^{2} \diff s \leq C_{y_{0},f,T}.
	\end{equation}
	For $q \in [2, \infty)$, $r \in (2,6]$ such that
	$$
		\frac{3}{r} + \frac{2}{q} = \frac{3}{2},
	$$
	by an application of the Sobolev embedding \eqref{ar_eq_interpolation_6D} and \eqref{conv_eq_H0H1est} we find
	\begin{equation} \label{conv_eq_Sobolev}
		\int_{0}^{T} \| y_{N} \|_{\mathcal{L}^{r}}^{q} \diff t \leq C_{1,0,2,2,r}^{q} \int_{0}^{T}  \| y_{N} \|_{\mathcal{H}^{1}}^{2}  \| y_{N} \|_{\mathcal{H}^{0}}^{q-2} \diff t \leq C_{y_{0},f,T,r,q}.
	\end{equation}
	Employing the Arzel\`{a}-Ascoli theorem and the Helmholtz-Weyl decomposition in the same way as in the proof of Theorem \ref{DTMHD_thm_intro_GWP}, we find a subsequence $y_{N_{k}}$ (again denoted by $y_{N}$) and a $y = \begin{pmatrix} \bm{v} , \bm{B} \end{pmatrix}	 \in L^{\infty}([0,T];\mathcal{H}^{0}) \cap L^{2}([0,T];\mathcal{H}^{1})$ such that for all $\tilde{y} = \begin{pmatrix} \tilde{\bm{v}} , \tilde{\bm{B}} \end{pmatrix} \in \mathcal{L}^{2}$
	\begin{equation}\label{conv_eq_L2conv}
		\lim_{N \rightarrow \infty} \sup_{t \in [0,T]} | \langle y_{N}(t) - y(t), \tilde{y} \rangle_{\mathcal{L}^{2}} | = 0.
	\end{equation}
	In fact, we can even prove that for every bounded open set $G \subset \mathbb{R}^{3}$, we have
	\begin{equation}\label{conv_eq_Friedrichs}
		\lim_{N \rightarrow \infty} \int_{0}^{T} \int_{G} | y_{N}(t,x) - y(t,x) |^{2} \diff x \diff t = 0.
	\end{equation}
	To this end, let $G \subset \bar{G} \subset Q$ for a cuboid $Q$, and $\rho$ be a smooth, non-negative cutoff function with $\rho \equiv 1$ on $G$, $\rho \equiv 0$ on $\mathbb{R}^{3} \backslash Q$, and by Friedrichs' inequality \eqref{DTMHD_exuniq_eq_Friedrichs}
	\begin{align*}
		&\int_{0}^{T} \int_{G} | y_{N}(t,x) - y(t,x) |^{2} \diff x \diff t \\
		&\leq \int_{0}^{T} \int_{Q} | \bm{v}_{N}(t,x) - \bm{v}(t,x) |^{2} \rho^{2}(x) \diff x \diff t + \int_{0}^{T} \int_{Q} | \bm{B}_{N}(t,x) - \bm{B}(t,x) |^{2} \rho^{2}(x) \diff x \diff t \\ 
		&\leq \sum_{i=1}^{K_{\varepsilon}} \int_{0}^{T}  \left( \int_{Q} (\bm{v}_{N}(t,x) - \bm{v}(t,x)) \rho(x) h_{i}^{\varepsilon}(x) \diff x \right)^{2} \diff t \\
		&+ \sum_{i=1}^{K_{\varepsilon}} \int_{0}^{T}  \left( \int_{Q} (\bm{B}_{N}(t,x) - \bm{B}(t,x)) \rho(x) h_{i}^{\varepsilon}(x) \diff x \right)^{2} \diff t \\
		&+ \varepsilon \int_{0}^{T} \int_{Q} | \nabla \left( (\bm{v}_{N} - \bm{v} ) \rho \right)(x) |^{2} \diff x \diff t + \varepsilon \int_{0}^{T} \int_{Q} | \nabla \left( (\bm{B}_{N} - \bm{B} ) \rho \right)(x) |^{2} \diff x \diff t \\
		&=: I_{1}(N,\varepsilon) + I_{2}(N,\varepsilon) + I_{3}(N,\varepsilon) + I_{4}(N,\varepsilon).
	\end{align*}
	The terms $I_{1}(N,\varepsilon)$, $I_{2}(N,\varepsilon)$ vanish for $N \rightarrow \infty$ as using \eqref{conv_eq_L2conv} we get
	\begin{align*}
		\lim_{N \rightarrow \infty} I_{1}(N,\varepsilon) \leq T \sum_{i=1}^{K_{\varepsilon}} \lim_{N \rightarrow \infty} \sup_{t \in [0,T]} \left| \int_{\mathbb{R}^{3}} (\bm{v}_{N}(t,x) - \bm{v}(t,x)) \rho(x) h_{i}^{\varepsilon}(x) 1_{Q}(x) \diff x  \right|^{2} = 0,
	\end{align*}
	and an analogous computation yields $\lim_{N \rightarrow \infty} I_{2}(N,\varepsilon) = 0$.
	
	The other two terms can be bounded by
	\begin{align*}
		I_{3}(N,\varepsilon) + I_{4}(N,\varepsilon) \leq \varepsilon \cdot C_{\rho} \int_{0}^{T} \left( \| y_{N}(t) \|_{\mathcal{H}^{1}}^{2} + \| y(t) \|_{\mathcal{H}^{1}}^{2} \right) \diff t \leq C_{\rho, y_{0},T,f} \cdot \varepsilon,
	\end{align*}
	and the arbitrariness of $\varepsilon > 0$ implies the claim.
	
	Next we prove that for any $\tilde{y} \in \mathcal{V}$
	\begin{align*}
		\lim_{N \rightarrow \infty} \int_{0}^{t} \langle g_{N}(|y_{N}(s)|^{2}) y_{N}(s), \tilde{y} \rangle_{\mathcal{H}^{0}} \diff s = 0.
	\end{align*}
	This can be seen as follows:
	\begin{align*}
		&\lim_{N \rightarrow \infty} \int_{0}^{t} \langle g_{N}(|y_{N}(s)|^{2}) y_{N}(s), \tilde{y} \rangle_{\mathcal{H}^{0}} \diff s \\
		&\leq \| \tilde{y} \|_{L^{\infty}} \cdot \limsup_{N \rightarrow \infty} \int_{0}^{t} \int_{\mathbb{R}^{3}} | y_{N}(s)|^{3} \cdot 1_{ \{|y_{N}(s,x)|^{2} \geq N \}} \diff x \diff s \\
		&\leq \| \tilde{y} \|_{L^{\infty}} \cdot \limsup_{N \rightarrow \infty} \left( \int_{0}^{t} \| y_{N}(s)\|_{\mcL^{10/3}}^{10/3} \diff s \right)^{9/10} \cdot \left( \int_{0}^{t} \int_{\mathbb{R}^{3}} 1_{ \{|y_{N}(s,x)|^{2} \geq N \}} \diff x \diff s \right)^{1/10} \\
		&\leq C_{\tilde{y}, y_{0},T,f} \cdot \limsup_{N \rightarrow \infty} \left( \frac{1}{N} \int_{0}^{t} \|y_{N}(s) \|_{\mathcal{H}^{0}}^{2} \diff s  \right)^{1/10} = 0,
	\end{align*}
	where we have used \eqref{conv_eq_Sobolev} and Chebychev's inequality. 
	
	As in the proof of Theorem \ref{DTMHD_thm_intro_GWP}, there exist pressure functions $p_{N}, \pi_{N} \in L^{2}([0,T];L_{\text{loc}}^{2}(\mathbb{R}^{3}))$ such that $\nabla p_{N}, \nabla \pi_{N} \in L^{2}([0,T];L^{2}(\mathbb{R}^{3};\mathbb{R}^{3}))$, and we have for almost all $t \geq 0$ that
	\begin{align}\label{conv_eq_ptw_v}
		\frac{\partial \bm{v}_{N} }{\partial t} &=  \Delta \bm{v}_{N} - \left( \bm{v}_{N} \cdot \nabla \right) \bm{v}_{N}  +  \left( \bm{B}_{N} \cdot \nabla \right) \bm{B}_{N} + \nabla \left( p_{N} + \frac{| \bm{B}_{N} |^2}{2} \right) \\
		\nonumber &\quad - g_{N}(| (\bm{v}_{N}, \bm{B}_{N}) |^{2}) \bm{v}_{N} + \bm{f}_{1}\\
		\label{conv_eq_ptw_B}\frac{\partial \bm{B}_{N}}{\partial t}
		&= \Delta \bm{B}_{N} - \left( \bm{v}_{N} \cdot \nabla \right) \bm{B}_{N} + (\bm{B}_{N} \cdot \nabla) \bm{v}_{N} + \nabla \pi_{N} - g_{N}(| (\bm{v}_{N}, \bm{B}_{N}) |^{2}) \bm{B}_{N} + \bm{f}_{2}.
	\end{align}
	To derive the generalised energy inequality, we take a non-negative $\phi \in C_{0}^{\infty}((0,T)\times \mathbb{R}^{3})$ and then take the inner products with $2 y_{N} \phi$ in $\mathcal{H}^{0}$ of this equation. Let us use the abbreviation $\iint = \int_{0}^{T} \int_{\mathbb{R}^{3}} \diff x \diff t$. Then we get
	\begin{align}
		\label{conv_eq_EI_timeterms}&\iint \partial_{t} \bm{v}_{N} \cdot 2 \bm{v}_{N} \phi + \iint \partial_{t} \bm{B}_{N} \cdot 2 \bm{B}_{N} \phi \\
		\label{conv_eq_EI_Deltas}&= \iint \Delta \bm{v}_{N} \cdot 2 \bm{v}_{N} \phi + \iint \Delta \bm{B}_{N} \cdot 2 \bm{B}_{N} \phi \\
		\label{conv_eq_EI_NL1}&\quad - \iint \left( \bm{v}_{N} \cdot \nabla \right) \bm{v}_{N} \cdot 2 \bm{v}_{N} \phi +  \iint \left( \bm{B}_{N} \cdot \nabla \right) \bm{B}_{N} \cdot 2 \bm{v}_{N} \phi  \\
		\label{conv_eq_EI_NL2}&\quad - \iint \left( \bm{v}_{N} \cdot \nabla \right) \bm{B}_{N} \cdot 2 \bm{B}_{N} \phi + \iint \left( \bm{B}_{N} \cdot \nabla \right) \bm{v}_{N} \cdot 2 \bm{B}_{N} \phi \\
		\label{conv_eq_EI_taming}&\quad - 2 \iint g_{N}(|y_{N}|^{2})|\bm{v}_{N}|^{2} \phi - 2 \iint g_{N}(|y_{N}|^{2})|\bm{B}_{N}|^{2} \phi \\
		\label{conv_eq_EI_pf}&\quad + 2 \iint \nabla p_{N} \cdot \bm{v}_{N} \phi + 2 \iint \nabla \pi_{N} \cdot \bm{B}_{N} \phi + 2 \iint \langle f,y_{N}\rangle \phi.
	\end{align}
	Let us discuss this equation line by line. The first line \eqref{conv_eq_EI_timeterms} is a simple application of integration by parts (with respect to the time variable):
	\begin{align*}
		\iint \partial_{t} y_{N} \cdot y_{N} \phi &= - \iint  y_{N} \cdot \partial_{t} \left( y_{N} \phi \right) = - \iint  y_{N} \cdot \left(\partial_{t} y_{N} \right) \phi - \iint  |y_{N}|^{2} \partial_{t} \phi
	\end{align*}
	which in turn yields
	\begin{align*}
		2 \iint \partial_{t} y_{N} \cdot y_{N} \phi = - \iint  |y_{N}|^{2} \partial_{t} \phi.
	\end{align*}
	For the second line \eqref{conv_eq_EI_Deltas}, we proceed along similar lines, this time with respect to the space variable. To avoid confusion, we will write the equation in components:
	\begin{align*}
		&2 \iint \Delta \bm{v}_{N} \cdot  \bm{v}_{N} \phi + 2 \iint \Delta \bm{B}_{N} \cdot  \bm{B}_{N} \phi \\
		&= - 2 \iint \sum_{i,k} \left( \partial_{i} \bm{v}_{N}^{k} \right) \partial_{i} \left(  \bm{v}_{N}^{k} \phi \right) - 2 \iint \sum_{i,k} \left( \partial_{i} \bm{B}_{N}^{k} \right) \partial_{i} \left(  \bm{B}_{N}^{k} \phi \right).
	\end{align*}
	We will focus on the velocity terms, the magnetic field works in exactly the same way.
	\begin{align*}
		- 2 \iint \sum_{i,k} \left( \partial_{i} \bm{v}_{N}^{k} \right) \partial_{i} \left(  \bm{v}_{N}^{k} \phi \right) = - 2 \iint \sum_{i,k} \left| \partial_{i} \bm{v}_{N}^{k} \right|^{2} \phi - 2 \iint (\partial_{i} \bm{v}_{N}^{k}) \bm{v}_{N}^{k} \partial_{i} \phi.
	\end{align*}
	The last term on the right-hand side equals, after another application of integration by parts,
	\begin{align*}
		- 2 \iint (\partial_{i} \bm{v}_{N}^{k}) \bm{v}_{N}^{k} \partial_{i} \phi = 2 \iint \bm{v}_{N}^{k} \partial_{i} \left( \bm{v}_{N}^{k} \partial_{i} \phi \right) = 2 \iint \left( \partial_{i} \bm{v}_{N}^{k} \right) \bm{v}_{N}^{k}   \partial_{i} \phi  + 2 \iint \bm{v}_{N}^{k} \bm{v}_{N}^{k} \left(\partial_{i}^{2} \phi \right),
	\end{align*}
	and thus
	\begin{align*}
		- 2 \iint (\partial_{i} \bm{v}_{N}^{k}) \bm{v}_{N}^{k} \partial_{i} \phi = \iint |\bm{v}_{N}|^{2} \Delta \phi.
	\end{align*}
	Hence, \eqref{conv_eq_EI_Deltas} can be rewritten as
	\begin{align*}
		2 \iint \Delta \bm{y}_{N} \cdot \bm{y}_{N} \phi = -\iint | \nabla y_{N}|^{2} \phi + \iint |y_{N}|^{2} \Delta \phi.
	\end{align*}
	The third line, \eqref{conv_eq_EI_NL1} will be dealt with term by term. By the incompressibility condition
	\begin{align*}
		- 2 \iint \left( \bm{v}_{N} \cdot \nabla \right) \bm{v}_{N} \cdot \bm{v}_{N} \phi = 2 \iint  |\bm{v}_{N}|^{2} \nabla \cdot \left( \bm{v}_{N} \phi \right) 
		= 2 \iint  |\bm{v}_{N}|^{2}  \bm{v}_{N} \cdot \nabla \phi.
	\end{align*}
	The second term, in a similar fashion, becomes (again using the divergence-freeness)
	\begin{align*}
		&2 \iint \left( \bm{B}_{N} \cdot \nabla \right) \bm{B}_{N} \cdot \bm{v}_{N} \phi = - 2 \iint (\bm{B}_{N} \cdot \nabla)\bm{v}_{N} \cdot \bm{B}_{N} \phi - 2 \iint (\bm{B}_{N} \cdot \bm{v}_{N})(\bm{B}_{N} \cdot \nabla \phi).
	\end{align*}
	The first term of the last line here cancels with the second term of \eqref{conv_eq_EI_NL2}. Thus we only have to deal with the first term of \eqref{conv_eq_EI_NL2}:
	\begin{align*}
		&- 2 \iint \left( \bm{v}_{N} \cdot \nabla \right) \bm{B}_{N} \cdot  \bm{B}_{N} \phi = -2 \sum_{i,k} \iint  \bm{v}_{N}^{i} \left( \partial_{i} \bm{B}_{N}^{k} \right) \bm{B}_{N}^{k} \phi \\
		&= 2 \sum_{i,k} \iint \bm{B}_{N}^{k} \partial_{i}\left( \bm{v}_{N}^{i} \bm{B}_{N}^{k} \phi \right) = 2 \sum_{i,k} \iint  \bm{v}_{N}^{i}   \bm{B}_{N}^{k} \partial_{i}\left( \bm{B}_{N}^{k} \phi \right) \\
		&= 2 \sum_{i,k} \iint  \bm{v}_{N}^{i}   \bm{B}_{N}^{k} \left( \partial_{i} \bm{B}_{N}^{k} \right)\phi + 2 \sum_{i,k} \iint  \bm{v}_{N}^{i}   \bm{B}_{N}^{k} \bm{B}_{N}^{k} \partial_{i}\phi \\
		&= 2 \iint \left( \bm{v}_{N} \cdot \nabla \right)\bm{B}_{N} \cdot \bm{B}_{N} \phi + 2 \iint |\bm{B}_{N}|^{2} \bm{v}_{N} \cdot \nabla \phi,
	\end{align*}
	and therefore
	\begin{align*}
		- 2 \iint \left( \bm{v}_{N} \cdot \nabla \right) \bm{B}_{N} \cdot  \bm{B}_{N} \phi = \iint |\bm{B}_{N}|^{2} \bm{v}_{N} \cdot \nabla \phi.
	\end{align*}
	The last terms that we have to treat are the pressure terms of \eqref{conv_eq_EI_pf}. For the first term, we find, again by integration by parts and the incompressibility
	\begin{align*}
		2 \iint \nabla p_{N} \cdot \bm{v}_{N} \phi = - 2 \iint p_{N} \bm{v}_{N} \cdot \nabla \phi,
	\end{align*}
	and similarly for the second term
	\begin{align*}
		2 \iint \nabla \pi_{N} \cdot \bm{B}_{N} \phi = - 2 \iint \pi_{N} \bm{B}_{N} \cdot \nabla \phi.
	\end{align*}
	Thus, altgether we find that
	\begin{equation}\label{conv_eq_EE}
		\begin{split}
		&2 \int_{0}^{T} \int_{\mathbb{R}^{3}} | \nabla y_{N}|^{2} \phi ~\diff x \diff s + 2 \int_{0}^{T} \int_{\mathbb{R}^{3}} g_{N}(|y_{N}|^{2})|y_{N}|^{2} \phi ~\diff x \diff s\\
		&= \int_{0}^{T} \int_{\mathbb{R}^{3}} \Big[ |y_{N}|^{2} \left( \partial_{t} \phi + \Delta \phi \right)  + 2 \langle y_{N}, f \rangle \phi - 2 \pi_{N} \langle \bm{B}_{N}, \nabla \phi \rangle_{\R^{3}} \\
		&\quad + (|y_{N}|^{2} - 2 p_{N}) \langle \bm{v}_{N} , \nabla \phi \rangle_{\mathbb{R}^{3}} - 2 \langle \bm{B}_{N} , \bm{v}_{N} \rangle_{\mathbb{R}^{3}} \langle \bm{B}_{N}, \nabla \phi \rangle_{\mathbb{R}^{3}} \Big] \diff x \diff s.
		\end{split}
	\end{equation}
	Since $0 \leq \phi \in C_{c}^{\infty}$, it acts as a density, and thus, by \cite[Theorem 1.2.1]{Davies90}, the map $y \mapsto \int_{0}^{T} \int_{\mathbb{R}^{3}} |\nabla y |^{2} \phi~ \diff x \diff s$ is lower semi-continuous in $L^{2}([0,T];\mathcal{H}^{0})$. Thus, the limit of the left-hand side of \eqref{conv_eq_EE} as $N \rightarrow \infty$ is greater than or equal to
	\begin{align*}
		&\liminf_{N \rightarrow \infty} \int_{0}^{T} \int_{\mathbb{R}^{3}} | \nabla y_{N}|^{2} \phi ~\diff x \diff s  + 2 \int_{0}^{T} \int_{\mathbb{R}^{3}} g_{N}(|y_{N}|^{2})|y_{N}|^{2} \phi ~\diff x \diff s \\
		&\geq \liminf_{N \rightarrow \infty} \int_{0}^{T} \int_{\mathbb{R}^{3}} | \nabla y_{N}|^{2} \phi ~\diff x \diff s \geq 2 \int_{0}^{T} \int_{\mathbb{R}^{3}} | \nabla y|^{2} \phi ~\diff x \diff s.
	\end{align*}
	On the other hand, the limit of the right-hand side as $N \rightarrow \infty$ consists of four terms, which we treat individually. We denote $G := \supp \phi$. 
	
	For the first term, by Cauchy-Schwarz-Buniakowski, \eqref{conv_eq_H0H1est} and \eqref{conv_eq_Friedrichs}, we find
	\begin{align*}
		&\int_{0}^{T} \int_{\mathbb{R}^{3}} \left( |y_{N}|^{2} - |y|^{2} \right) \left( \partial_{t} \phi + \Delta \phi \right)\diff x \diff s \\
		&\leq \int_{0}^{T} \int_{G}  |y_{N} - y| (|y_{N}| + |y|) \left( \partial_{t} \phi + \Delta \phi \right)\diff x \diff s \\
		&\leq C_{\phi} \left( \int_{0}^{T} \int_{G}  |y_{N} - y|^{2} \diff x \diff s \right)^{1/2}  \left( \int_{0}^{T} \int_{G}  2( |y_{N}|^{2} + |y|^{2}) \diff x \diff s \right)^{1/2} \\
		&\leq C_{\phi, y_{0},T,f}\left( \int_{0}^{T} \int_{G}  |y_{N} - y|^{2} \diff x \diff s \right)^{1/2} \overset{N \rightarrow \infty}{\longrightarrow} 0.
	\end{align*}
	The second term can be treated in a similar fashion:
	\begin{align*}
		\int_{0}^{T} \int_{\mathbb{R}^{3}} 2 \langle y_{N} - y , f \rangle \phi ~\diff x \diff s &\leq 
		\left( \int_{0}^{T} \int_{G} |y_{N} - y|^{2} \diff x \diff s \right)^{1/2} \left( \int_{0}^{T} \int_{G} |f|^{2} \phi^{2} \diff x \diff s \right)^{1/2} \\
		&\leq C_{\phi} \| f \|_{L^{2}([0,T];\mathcal{H}^{0})} \left( \int_{0}^{T} \int_{G} |y_{N} - y|^{2} \diff x \diff s \right)^{1/2} \overset{N \rightarrow \infty}{\longrightarrow} 0.
	\end{align*}
	For the term 
	\begin{align*}
	\int_{0}^{T} \int_{\mathbb{R}^{3}} |y_{N}|^{2} \langle \bm{v}_{N} , \nabla \phi \rangle_{\mathbb{R}^{3}} \diff x \diff s = \int_{0}^{T} \int_{G} |y_{N}|^{2} \left\langle \bm{v}_{N} , 1_{G}\frac{\nabla \phi}{\phi} \right\rangle_{\mathbb{R}^{3}} \phi \diff x \diff s,
	\end{align*}  
	we note that since $\phi \in C_{c}^{\infty}((0,T)\times \mathbb{R}^{3})$, \eqref{conv_eq_Friedrichs} implies convergence in measure for the finite measure $\mu := \phi \diff x \otimes \diff s$. One can also prove uniform integrability with respect to this measure.
	Then by the generalised Lebesgue dominated convergence theorem we get
	\begin{align*}
		\int_{0}^{T} \int_{\mathbb{R}^{3}} |y_{N}|^{2} \langle \bm{v}_{N} , \nabla \phi \rangle_{\mathbb{R}^{3}} \diff x \diff s \overset{N \rightarrow \infty}{\longrightarrow} \int_{0}^{T} \int_{\mathbb{R}^{3}} |y|^{2} \langle \bm{v} , \nabla \phi \rangle_{\mathbb{R}^{3}} \diff x \diff s.
	\end{align*}
	Moving on with the energy inequality, the last term of \eqref{conv_eq_EE} can be treated in the same way as just discussed. We are left with the pressure term. As in \cite{RZ09a}, pp. 547 f., we take the divergence of \eqref{conv_eq_ptw_v} to find
	\begin{equation}\label{conv_eq_Poisson_p}
		\begin{split}
		\Delta p_{N} &= \divv \left( (\bm{v}_{N} \cdot \nabla)\bm{v}_{N} - (\bm{B}_{N} \cdot \nabla)\bm{B}_{N} - \nabla \frac{|\bm{B}_{N}|^{2}}{2} + g_{N}(|y_{N}|^{2})\bm{v}_{N}  \right) \\
		&= \divv \left( (\bm{v}_{N} \cdot \nabla)\bm{v}_{N} - (\bm{B}_{N} \cdot \nabla)\bm{B}_{N} - \bm{B}_{N} \cdot (\nabla \bm{B}_{N}) + g_{N}(|y_{N}|^{2})\bm{v}_{N}  \right).
		\end{split}
	\end{equation}
	Similarly, we take the divergence of \eqref{conv_eq_ptw_B} and obtain\footnote{Noting that $(\bm{v}_{N} \cdot \nabla)\bm{B}_{N} - (\bm{B}_{N} \cdot \nabla)\bm{v}_{N} = \nabla \times (\bm{v} \times \bm{B})$, which is divergence-free, cf. Section \ref{DTMHD_sssec_MPP}.}
	\begin{equation}\label{conv_eq_Poisson_pi}
		\Delta \pi_{N} = \divv \left( g_{N}(|y_{N}|^{2})\bm{B}_{N} \right).
	\end{equation}
	We note that for $N$ sufficiently large
	$$
		(g_{N}(r))^{9/8} \cdot r^{9/16} \leq C g_{N}(r) \cdot r.
	$$
	This is obviously true on the set $\{ r ~|~ g_{N}(r) = 0 \}$. If $r > 0$ is such that $g_{N}(r) > 0$ (which implies $r \geq 1$), we have
	\begin{align*}
		&(g_{N}(r))^{9/8} \cdot r^{9/16} \leq g_{N}(r) (g_{N}(r))^{1/8} \cdot r^{9/16} \\
		&\leq g_{N}(r) 2^{1/8} r^{1/8} \cdot r^{9/16} \leq g_{N}(r) 2^{1/8} \cdot r^{11/16} \leq 2^{1/8} g_{N}(r) \cdot r,
	\end{align*}
where the factor of $2$ appears due to the definition of the taming function. Using this inequality and \eqref{conv_eq_H0H1est}, we find
\begin{equation} \label{conv_eq_aux_est_taming}
	\int_{0}^{T} \int_{\mathbb{R}^{3}} |g_{N}(|y|^{2}) y |^{9/8} \diff x \diff t \leq C \int_{0}^{T} \int_{\mathbb{R}^{3}} g_{N}(|y|^{2}) \cdot |y|^{2} \diff x \diff t \leq C_{T,y_{0},f}.
\end{equation}
For the first three nonlinear terms on the right-hand side of \eqref{conv_eq_Poisson_p}, we note that by H\"older's inequality (first for the product measure $\diff x \otimes \diff t$ and with $p=16/7$, $q=16/9$, then for $\diff t$ with $p=14/6$, $q = 14/8$) and the Sobolev embedding \eqref{ar_eq_interpolation_6D} we have
\begin{equation}\label{conv_eq_aux_est_NL}
	\begin{split}
		&\int_{0}^{T} \int_{\mathbb{R}^{3}} |(\bm{v}_{N} \cdot \nabla ) \bm{v}_{N}|^{9/8} \diff x \diff t \leq \int_{0}^{T} \int_{\mathbb{R}^{3}} |\bm{v}_{N}|^{9/8} |\nabla  \bm{v}_{N}|^{9/8} \diff x \diff t \\
		&\leq \left( \int_{0}^{T} \| \bm{v}_{N} \|_{L^{18/7}}^{18/7} \right)^{\frac{7}{16}} \cdot \left( \int_{0}^{T} \| \bm{v}_{N} \|_{\mathbb{H}^{1}}^{2} \diff t \right)^{\frac{9}{16}} \\
		&\leq \left( \int_{0}^{T} \| y_{N} \|_{\mathcal{L}^{18/7}}^{18/7} \diff t \right)^{\frac{7}{16}} \cdot \left( \int_{0}^{T} \| y_{N} \|_{\mathcal{H}^{1}}^{2} \diff t \right)^{\frac{9}{16}} \\
		&\leq C_{y_{0},T,f} \left( C_{2,0,2,2,18/7}^{18/7} \int_{0}^{T} \left ( \| y_{N} \|_{\mathcal{H}^{1}}^{1/3} \| y_{N} \|_{\mathcal{L}^{2}}^{2/3} \right)^{18/7} \diff t \right)^{7/16} \\
		&= C_{y_{0},T,f}  C_{2,0,2,2,18/7}^{9/8} \left( \int_{0}^{T}  \| y_{N} \|_{\mathcal{H}^{1}}^{6/7} \| y_{N} \|_{\mathcal{L}^{2}}^{12/7} \diff t \right)^{7/16} \\
		&\leq C_{y_{0},T,f}  C_{2,0,2,2,18/7}^{9/8} \left( \left( \int_{0}^{T}  \| y_{N} \|_{\mathcal{H}^{1}}^{2} \diff t \right)^{6/14} \left( \int_{0}^{T}  \| y_{N} \|_{\mathcal{L}^{2}}^{3} \diff t \right)^{8/14} \right)^{7/16} \\
		&\leq C_{T,y_{0},f}.
	\end{split}
\end{equation}
The other terms can be bounded in the same way.

Again using the interpolation inequality \eqref{ar_eq_interpolation_6D}, we find
\begin{equation}\label{conv_eq_aux_est_pressure}
	\int_{0}^{T} \| p_{N} \|_{L^{9/5}}^{9/8} \diff t \leq C_{1,0,9/8,9/8,9/5}^{9/8} \int_{0}^{T} \| p_{N} \|_{1,9/8}^{9/8} \diff t.
\end{equation}
Recall that $ \Delta p_{N} = \divv \mathcal{R}_{N}$, where $\mathcal{R}_{N}$ is defined by \eqref{conv_eq_Poisson_p}. Then we have
\begin{align*}
	\| p_{N} \|_{1,9/8} = \| (I - \Delta)^{1/2}  \Delta^{-1} \Delta p_{N} \|_{L^{9/8}} = \| (I - \Delta)^{1/2} \nabla^{-1} \mathcal{R}_{N} \|_{L^{9/8}} \leq \| \mathcal{R}_{N} \|_{L^{9/8}}.
\end{align*}
Here we used the $L^{p}$ theory for singular integrals, e.g. \cite[Chapter V.3.2, Lemma 2, p. 133 f.]{Stein70}.
By \eqref{conv_eq_aux_est_taming} and \eqref{conv_eq_aux_est_NL} it follows that the right-hand side of \eqref{conv_eq_aux_est_pressure} is uniformly bounded in $N$.

 Therefore, by the Eberlein-\v{S}muljan theorem (cf. \cite[Theorem 21.D, p. 255]{Zeidler90IIa}), there is a subsequence $(p_{N_{k}})_{k}$ and a function
$$
	p \in L^{9/8}([0,T];L^{9/5}(\mathbb{R}^{3};\mathbb{R}^{3}))
$$
such that for $k \rightarrow \infty$
\begin{equation}
	p_{N_{k}} \rightarrow p \quad \mathrm{weakly~in~} L^{9/8}([0,T];L^{9/5}(\mathbb{R}^{3};\mathbb{R}^{3})).
\end{equation}	
	Finally, by another application of \eqref{conv_eq_Sobolev}, with $q = 12, r = 9/4$, we find
	$$
		\int_{0}^{T} \| y_{N} \|_{\mathcal{L}^{9/4}}^{12} \diff t \leq C_{T,y_{0},f}.
	$$
	Thus, in the same way as above, we can employ the generalised Lebesgue dominated convergence theorem to conclude that for $\phi \in C_{0}^{\infty}((0,T)\times \mathbb{R}^{3})$
	\begin{align*}
		\lim_{N \rightarrow \infty} &\left| \int_{0}^{T} \int_{\mathbb{R}^{3}} \langle p_{N} \bm{v}_{N} - p \bm{v}, \nabla \phi \rangle_{\mathbb{R}^{3}} \diff x \diff t \right| \\
		&\leq \lim_{N \rightarrow \infty} \left| \int_{0}^{T} \int_{\mathbb{R}^{3}} (p_{N} - p) \langle \bm{v}, \nabla \phi \rangle_{\mathbb{R}^{3}} \diff x \diff t \right| \\
		&\quad + \left| \int_{0}^{T} \int_{\mathbb{R}^{3}}  p_{N} \langle \bm{v}_{N} - \bm{v}, \nabla \phi \rangle_{\mathbb{R}^{3}} \diff x \diff t \right| \\
		&\leq \lim_{N \rightarrow \infty} \left( \| p_{N} \|_{L^{9/5}}^{9/8} \diff t \right)^{8/9} \left( \int_{0}^{T} \| | \bm{v}_{N} - \bm{v} | \cdot | \nabla \phi | \|_{L^{9/4}}^{9} \right)^{1/9} = 0.
	\end{align*}
	In exactly the same way we find a subsequence $(N_{k})_{k \in \N}$ such that
	\begin{align*}
		\pi_{N_{k}} \rightarrow \pi \quad \mathrm{weakly~in~} L^{9/8}([0,T];L^{9/5}(\mathbb{R}^{3};\mathbb{R}^{3})).
	\end{align*}
	The limit $\pi$ satisfies the equation 
	\begin{align*}
		\Delta \pi = 0,
	\end{align*}
	which, combined with the integrability property of $\pi$ yields $\pi \equiv 0$, thus eliminating the "magnetic pressure" from the resulting weak equation as well as the generalised energy inequality. Hence we have shown that the solutions to the tamed MHD equations converge to suitable weak solutions to the MHD equations.
\end{proof}

\begin{remark}
	It is to be expected that existence and uniqueness in the case of a bounded domain $\mbD \subset \R^{3}$ can be shown in a similar way as for the tamed Navier-Stokes equations, as in W. Liu and M. R\"ockner \cite[p. 170 ff.]{LR15}. However, their method uses an inequality \cite[Equation (5.61), p. 166]{LR15}, sometimes called Xie's inequality, for the $L^{\infty}$-norm of a function in terms of the $L^{2}$-norms of the gradient and the Laplacian (more precisely, the Stokes operator). This inequality holds for Dirichlet boundary conditions on quite general domains (cf. R.M. Brown, Z.W. Shen \cite[Equation (0.2), p. 1184]{BS95} for Lipschitz boundaries). If we were to use the method of \cite{LR15}, we would need to have a similar inequality for the magnetic field as well. Unfortunately, to the best of the author's knowledge, such an inequality has not yet been established for the boundary conditions 
	$$
	\bm{B} \cdot \bm{\nu} = 0, \quad (\nabla \times \bm{B}) \times \bm{\nu} = 0 \quad \mathrm{on} \ \partial \mbD
	$$ 
	of the magnetic field (which mean that the boundary is perfectly conducting, cf. \cite[Equation (1.3), p. 637]{ST83}). Here, $\nu$ is the outward unit normal vector of the boundary of the domain. If such an inequality could be shown, the rest of the proof of Liu and R\"ockner should follow in exactly the same way.
\end{remark}

\appendix 

\section{\texorpdfstring{$L^{p}$}{Lp} Solutions and Integral Equations}\label{chap:FJR}
In \cite{FJR72} E.B. Fabes, B.F. Jones and N.M. Riviere proved that the weak formulation for the Navier-Stokes equations on the whole space is equivalent to solving a nonlinear integral equation of the form
\begin{equation}
 u(x,t) + \mathcal{B}(u,u)(x,t) = \int_{\mathbb{R}^{d}} \Gamma(x-y,t) u_{0}(y) \diff y.
\end{equation}
They used this formulation to prove regularity estimates in mixed space-time $L^p$ spaces. Their results play an important role in showing smoothness for smooth initial data for the weak solution of the tamed Navier-Stokes equations in \cite{RZ09a} and in this section we attempt to derive analogous results for the MHD equations.

The main idea of \cite{FJR72} is threefold:
\begin{enumerate}[label=\arabic*., ref = \arabic*.]
    \item Find a divergence-free solution to the heat equation with the initial data of the Navier-                     Stokes problem via Fourier analysis.
    \item Use this solution as a test function in the weak formulation to derive the integral equation.
    \item Prove regularity of the solution to the integral equation (which amounts to estimating the nonlinear term $\mathcal{B}(u,u)$.
\end{enumerate}
We follow their steps with the necessary modifications of the MHD case. Fabes, Jones and Riviere consider mixed space-time $L^{p,q}$-norms on $S_{T} := \R^{d} \times [0,T]$ for $p,q \geq 2$, defined by 
\begin{align*}
	\| u \|_{L^{p,q}(S_{T})} := \sum_{j=1}^{d} \left[ \int\limits_{0}^{T} \left( \int_{\R^{d}} |u_{j}(x,t)|^{p} \diff x \right)^{q/p} \diff t \right]^{1/q},
\end{align*} 
where $\frac{d}{p} + \frac{2}{q} \leq 1$, $d < p < \infty$. The space of functions that have finite $L^{p,q}$-norm is denoted by $L^{p,q}(S_{T})$.
As we are only interested in the case $p=q$, we will occasionally assume this for simplicity in the following. All the results that follow, however, are true also in the more general case.

\subsection{A Divergence-Free Solution to the Heat Equation on the Whole Space}
The first step consists in constructing a symmetric $d \times d$ matrix-valued function $(t,x) \mapsto (E_{ij}(x,t))_{i,j=1}^{d}$ with the following properties:
\begin{enumerate}[label = (\roman*), ref = (\roman*)]
 \item $\Delta E_{ij}(x,t) - \partial_{t} E_{ij}(x,t) = 0$ for all $t > 0$, $x \in \mathbb{R}^{d}$.
 \item $\nabla \cdot E_{i}(x,t) = 0$ for all $x \in \mathbb{R}^{d}$, $t > 0$ where $E_{i}$ is the i-th row of $E_{ij}$, i.e., $E_{i} = (E_{i1}, \ldots, E_{id})$
 \item For $g \in \mathbb{L}^{p}(\mathbb{R}^{d})$, $1 \leq p < \infty$ (i.e., $g \in L^{p}(\mathbb{R}^{d})$ and $\nabla \cdot g = 0$ in the sense of distributions), the following convergence holds:
 $$
    \int_{\mathbb{R}^{d}} E(x-y,t)(g(y))dy \rightarrow g(x) \quad \mathrm{in} \quad L^{p}(\mathbb{R}^{d}) \quad \mathrm{as} \quad t \downarrow 0.
 $$
\end{enumerate}
The function is given by
\begin{equation}
    E_{ij}(x,t) = \delta_{ij} \Gamma(x,t) - R_{i} R_{j} \Gamma(x,t),
\end{equation}
where
$$
    \Gamma(x,t) := \frac{1}{(4 \pi t)^{n/2}}e^{- |x|^{2} / 4t}
$$
denotes the Weierstra{\ss} kernel and $R_{j}$ denotes the Riesz transformation,
$$
    R_{j}(f)(x) := L^{p}-\lim_{\varepsilon \rightarrow 0} c_{j} \int_{|x-y| > \varepsilon} \frac{x_{j} - y_{j}}{|x-y|^{n+1}} f(y) \diff y.
$$
Then one can show that $E_{ij}(x,t) \in C^{\infty}(\mathbb{R}^{d} \times (0,\infty))$, and for $1 \leq p < \infty$ and $g \in \mathbb{L}^{p}(\mathbb{R}^{d})$,
\begin{equation}
    E_{ij}(x,t)(g_{i})(x) = \int_{\mathbb{R}^{d}} \Gamma(x-y,t) g_{i}(y) \diff y, \quad a.e.
\end{equation}
We now want to define the nonlinear operator $\mathcal{B}$. Recall that $E_{i}$ denotes the i-th row of $E_{ij}$. We denote by $\langle u(y,s), \nabla E(x-y,t-s) \rangle$ the $d \times d$-matrix\footnote{Note that $E$ is a $d \times d$-matrix, so taking its gradient we get a tensor of rank 3. By multiplying with the vector $u$, we obtain a tensor of rank 2, i.e., a matrix.}
$$
    \left( \left\langle u(y,s), D_{x_{k}} E_{i} (x-y, t-s) \right\rangle \right)_{i,k = 1}^{d} 
    =  \left( \sum_{j=1}^{d} u_{j}(y,s) D_{x_{k}} E_{ij} (x-y, t-s) \right)_{i,k = 1}^{d},
$$
and define the operator $\bar{\mathcal{B}}(u,w)$ by
\begin{equation}\label{AppA_eq_def_Bbar}
    \bar{\mathcal{B}}(u,w)(x,t) := \int_{0}^{t} \int_{\mathbb{R}^{d}} \langle u(y,s), \nabla E(x-y,t-s) \rangle \cdot w(y,s) \diff y \diff s.
\end{equation}
Note that even though $E_{ij}(\cdot, 1) \notin L^{1}(\mathbb{R}^{d})$, since its Fourier transform is not continuous at the origin, and $L^{1}$ functions have uniformly continuous Fourier transform, cf. \cite[Satz V.2.2, p. 212]{Werner11}, we still have $D_{x_{k}} E_{ij} \in L^{1}(S_{T})$. This implies the following:

\begin{lemma}\label{AppA_thm_Bwelldef} Let $u,w \in L^{p}(S_{T})$, $p \geq 2$. Then $\bar{\mathcal{B}}(u,w) \in L^{p/2}(S_{T})$.
\end{lemma}
\begin{proof} 
As this statement is not entirely obvious and we will need it below, we prove it here for the reader's convenience. We want to estimate
\begin{align*}
    &\sum_{i} \left[ \int_{0}^{T} \int_{\mathbb{R}^{d}} |\bar{\mathcal{B}}_{i}(u,w)|^{p/2} \diff x \diff t  \right]^{2/p} \\
    &= \sum_{i} \left[ \int_{0}^{T} \int_{\mathbb{R}^{d}} \left| \sum_{j,k} \int_{0}^{t} \int_{\mathbb{R}^{d}} D_{x_{k}} E_{ij}(x-y,t-s) u_{j}(y,s) v_{k}(y,s) \diff y \diff s \right|^{p/2} \diff x \diff t \right]^{2/p}.
\end{align*}
To simplify notations, we let 
$$
F_{ijk}(y,s) := |D_{x_{k}} E_{ij}(y,s) |1_{[0,T]}(s) ~ \mathrm{and} ~ G_{jk}(y,s) := 1_{[0,T]}(s)|u_{j}(y,s) v_{k}(y,s)|.
$$
We denote by $f \oast g$ the convolution in space and time, i.e.,
\begin{align*}
    f \oast g (x,t) := \int_{-\infty}^{\infty} \int_{\mathbb{R}^{d}} f(x-y,t-s) g(y,s) \diff y \diff s.
\end{align*}
Then, using that $s \in [0,t]$ if and only if $t-s \in [0,t]$, the inner integral of the expression we want to estimate can be written and estimated as
\begin{align*}
    &\int\limits_{0}^{T} \int\limits_{\mathbb{R}^{d}} \left| \sum_{j,k} \int\limits_{0}^{t} \int\limits_{\mathbb{R}^{d}} D_{x_{k}} E_{ij}(x-y,t-s) u_{j}(y,s) v_{k}(y,s) \diff y \diff s \right|^{p/2} \diff x \diff t  \\
    &= \int\limits_{0}^{T} \int\limits_{\mathbb{R}^{d}} \left| \sum_{j,k} \int\limits_{-\infty}^{\infty} \int\limits_{\mathbb{R}^{d}} 1_{[0,t]}(t-s)D_{x_{k}} E_{ij}(x-y,t-s) 1_{[0,t]}(s) u_{j}(y,s) v_{k}(y,s) \diff y \diff s \right|^{p/2} \diff x \diff t \\
    &\leq \int\limits_{0}^{T} \int\limits_{\mathbb{R}^{d}} \left( \sum_{j,k} \int\limits_{-\infty}^{\infty} \int\limits_{\mathbb{R}^{d}} 1_{[0,T]}(t-s) \left|D_{x_{k}} E_{ij}(x-y,t-s) 1_{[0,T]}(s) \right| \left| u_{j}(y,s) v_{k}(y,s) \right| \diff y \diff s \right)^{p/2} \diff x \diff t \\
    &= \int\limits_{0}^{T} \int\limits_{\mathbb{R}^{d}} \left( \sum_{j,k} (F_{ijk} \oast G_{jk})(x,t) \right)^{p/2} \diff x \diff t.
\end{align*}
Thus by Young's convolution inequality and the Cauchy-Schwarz-Buniakowski inequality
\begin{align*}
    &\sum_{i} \left[ \int_{0}^{T} \int_{\mathbb{R}^{d}} |\bar{\mathcal{B}}_{i}(u,w)|^{p/2} \diff x \diff t  \right]^{2/p} \leq \sum_{i} \left[ \int_{0}^{T} \int_{\mathbb{R}^{d}} \left( \sum_{j,k} (F_{ijk} \oast G_{jk})(x,t) \right)^{p/2} \diff x \diff t \right]^{2/p} \\
    &= \sum_{i} \| \sum_{j,k} 1_{[0,T]} (F_{ijk} \oast G_{jk}) \|_{L^{p/2}(\mathbb{R}^{d+1})} \leq \sum_{i,j,k} \| (F_{ijk} \oast G_{jk}) \|_{L^{p/2}(\mathbb{R}^{d+1})} \\
    &\leq \sum_{i,j,k} \| F_{ijk} \|_{L^{1}(\mathbb{R}^{d+1})} \| G_{jk} \|_{L^{p/2}(\mathbb{R}^{d+1})} = \sum_{i,j,k} \| D_{x_{k}} E_{ij} \|_{L^{1}(S_{T})} \| u_{j} v_{k} \|_{L^{p/2}(S_{T})} \\
    &= \sum_{i,j,k} \| D_{x_{k}} E_{ij} \|_{L^{1}(S_{T})} \| u_{j} \|_{L^{p}(S_{T})} \| v_{k} \|_{L^{p}(S_{T})} \leq \left( \sum_{i,j,k} \| D_{x_{k}} E_{ij} \|_{L^{1}(S_{T})} \right) \left( \sum_{j,k} \| u_{j} \|_{L^{p}(S_{T})} \| v_{k} \|_{L^{p}(S_{T})} \right) \\
    &= \left( \sum_{i,j,k} \| D_{x_{k}} E_{ij} \|_{L^{1}(S_{T})} \right)  \| u \|_{L^{p}(S_{T})} \| v \|_{L^{p}(S_{T})},
\end{align*}
where we used H\"older's inequality for $p=1, p'=\infty$ for the sum $\sum_{j,k}$ and the embedding $\ell^{\infty}(d) \subset \ell^{1}(d)$. The proof is complete.
\end{proof}
Given this operator, we define the operator $\mathcal{B}(y_{1},y_{2})$ by
\begin{equation}
    \mathcal{B}(y_{1},y_{2}) := \begin{pmatrix}
                                    \bar{\mathcal{B}}(v_{1},v_{2}) - \bar{\mathcal{B}}(B_{1},B_{2}) \\
                                    \bar{\mathcal{B}}(v_{1},B_{2}) - \bar{\mathcal{B}}(B_{1},v_{2})
                                \end{pmatrix}.
\end{equation}
The operator $\mathcal{B}$ occurs naturally when we (formally) use the function $E_{i}$ as a test function in the weak formulation. This will be the subject of the next section.

Before moving on with the theory, let us give a useful Sobolev version of the classical Schauder estimates (for lack of a better name) for the heat equation. To this end, we introduce another short-hand notation for the convolution appearing in the definition of $\bar{\mcB}$, Equation \eqref{AppA_eq_def_Bbar}:
\begin{align*}
	(f \bar{\oast} g) (x,t) := \int_{-\infty}^{t} \int_{\mathbb{R}^{d}} f(x-y,t-s) g(y,s) \diff y \diff s.
\end{align*}
\begin{lemma}\label{AppA_thm_keylemma}
	Let $\gamma \in \R$, $p,q \in (1, \infty)$, and $f \in L^{q}((0,T); W^{\gamma,p})$ be a function. Then the convolution of $f$ with the heat kernel $\Gamma$ lies in the space $ L^{q}((0,T); W^{\gamma + 2,p})$. More precisely,
	\begin{align*}
		\| \Gamma \bar{\oast} f \|_{L^{q}((0,T); W^{\gamma + 2,p})} &\leq c \| f \|_{L^{q}((0,T); W^{\gamma,p})}.
	\end{align*}
	These can also be rewritten into estimates with respect to temporal derivatives.
\end{lemma}
\begin{proof}
	In the case $p=q$, the result is classical, cf. e.g. the book of O.A. Ladyzhenskaya, V.A. Solonnikov, N.N. Ural'ceva \cite[Chapter IV, Equation (3.1), p. 288]{LSU67}. For $p \neq q$, it was proved by N.V. Krylov in \cite[Theorem 1.1]{Krylov01} using a Banach space version of the Calder\'{o}n-Zygmund theorem.
\end{proof}

\subsection{Equivalence of Weak Solutions to the MHD Equations and Solutions to the Integral Equation}

By $\mathcal{S}(\mathbb{R}^{d+1})$ we denote the Schwartz space of rapidly decreasing functions, and by $\mathcal{S}'(\mathbb{R}^{d+1})$ we denote its dual space, the space of tempered distributions. Our space of test functions for the weak formulation of the MHD equations is
$$
    \mathcal{D}_{T} := \{ \phi(x,t) \in \mathbb{R}^{d} ~|~ \phi_{i} \in \mathcal{S}(\mathbb{R}^{d+1}), \phi_{i} \equiv 0 ~\mathrm{for~} t \geq T, \nabla \cdot \phi (x,t) = 0 ~\forall x,t \}.
$$
\begin{definition}[Weak solution]\label{AppA_def_weak_soln}
    A function $y(x,t) = (v(x,t), B(x,t)) \in \mathbb{R}^{2d}$ is a \emph{weak solution} of the MHD equations with initial value $y_{0} = (v_{0}, B_{0})$ if the following conditions hold:
    \begin{enumerate}[label=(\roman*), ref=(\roman*)]
     \item $v, B \in L^{p,q}(S_{T})$, $p,q \geq 2$.
     \item For all test functions $\tilde{y} = (\tilde{v}, \tilde{B}) \in \mathcal{D}_{T}$ we have the following equality:
     \begin{equation}\label{AppA_eq_def_weak_soln}
        \begin{split}
            \int_{0}^{T} \int_{\mathbb{R}^{d}} \langle y, \partial_{t} \tilde{y} + \Delta \tilde{y} \rangle \diff x &+ \int_{0}^{T} \int_{\mathbb{R}^{d}} \langle v, (\nabla \tilde{v})(v) - (\nabla \tilde{B})(B) \rangle \diff x \diff t \\
            &+ \int_{0}^{T} \int_{\mathbb{R}^{d}} \langle B, (\nabla \tilde{B})(v) - (\nabla \tilde{v})(B) \rangle \diff x \diff t \\
            &= - \int_{\mathbb{R}^{d}} \langle y_{0}(x), \tilde{y}(x,0) \rangle \diff x - \int_{0}^{T} \int_{\mathbb{R}^{d}} \langle f, \tilde{y} \rangle \diff x \diff t.
        \end{split}
     \end{equation}
    \item For $dt$-a.e. $t \in [0,T]$, $\nabla \cdot u(\cdot, t) = 0$ in the sense of distributions.
    \end{enumerate}
\end{definition}
The weak solution satisfies a scalar equation. We cast this scalar equation, by choosing suitable test functions, into an equivalent vector-valued equation. More precisely, we have the following result, corresponding to Theorem (2.1) in \cite{FJR72} for $f=0$ and Theorem (4.4) for $f \neq 0$.
\begin{theorem}[Integral Equation]\label{AppA_thm_equiv}
    Let $y_{0} \in \mathbb{L}^{r}(\mathbb{R}^{d})$, $1 \leq r < \infty$, $p,q \geq 2$, $p < \infty$. If $y \in L^{p,q}(S_{T})$ is a weak solution of the MHD equations with initial value $y_{0}$, then $y$ solves the integral equation
    \begin{equation}
        y + \mathcal{B}(y,y) = \int_{\mathbb{R}^{d}} \Gamma(x-z,t) y_{0}(z)\diff z + \int_{0}^{t} \int_{\mathbb{R}^{d}} E(x-z,t-s)(f(z,s))\diff z \diff s.
    \end{equation}
\end{theorem}
\begin{proof}
This can be proven in exactly the same way as in \cite{FJR72}.
\end{proof}
\subsection{Regularity of Solutions to the Integral Equation}

\begin{theorem}[Regularity]\label{AppA_thm_reg} Let $y$ be a solution to the equation $y + \mathcal{B}(y,y) = f$ and $y \in L^{p,q}(S_{T})$, $\frac{2}{q} + \frac{d}{p} \leq 1$. Let $k$ be a positive integer such that $k+1 < p,q < \infty$. If 
\begin{align*}
    D_{x}^{\alpha} D_{t}^{j} f \in L^{p / (|\alpha| + 2j + 1), q/(|\alpha| + 2j + 1)}(S_{T}) \quad whenever~ |\alpha| + 2j \leq k, 
\end{align*}
then also 
\begin{align*}
    D_{x}^{\alpha} D_{t}^{j} y \in L^{p / (|\alpha| + 2j + 1), q/(|\alpha| + 2j + 1)}(S_{T}) \quad for~  |\alpha| + 2j \leq k.
\end{align*}

\end{theorem}

\begin{proof}
The proof proceeds along the same lines as that of Theorem 3.4 in \cite{FJR72}, with the necessary modifications to the MHD case. Since  $y = f - \mathcal{B}(y,y)$, we only have to show that $D_{x}^{\alpha} D_{t}^{j} \mathcal{B}(y,y) \in L^{p / (|\alpha| + 2j + 1), q/(|\alpha| + 2j + 1)}(S_{T})$.

For $k=1$, our assumptions read $f \in L^{p,q}(S_{T})$ and $D_{x_{i}} f \in L^{p/2,q/2}(S_{T})$ for all $i$. We split the argument into two parts: that for the $v$-part of the equation, i.e., for the equation $v + \mathcal{B}_{1}(y,y) = f_{1}$ and that for the $B$-part of the equation, i.e., $B + \mathcal{B}_{2}(y,y) = f_{2}$. 
Since the terms $\bar{\mcB}_{i}(u,v)(x,t)$ are of the form (with summation convention)
\begin{align*}
	\bar{\mcB}_{i}(u,v)(x,t) = \left( D_{x_{k}} \Gamma \right) \bar{\oast} \left( u_{l} \left[ \delta_{il} v_{k} - R_{i}R_{l} v_{k} \right] \right)(x,t),
\end{align*}
and as our assumption on the coefficients $p,q$ implies $1 < \frac{p}{2}, \frac{q}{2} < \infty$, we can apply Lemma \ref{AppA_thm_keylemma} as well as the $L^{p,q}$-boundedness of the Riesz transform (cf. J.E. Lewis \cite[Theorem 4, p. 226]{Lewis67}) to find that
\begin{align*}
		\| D_{x_{j}} \bar{\mcB}_{i}(u,v) \|_{L^{p/2,q/2}(S_{T})} &\leq C \| u_{l} \left[ \delta_{il} v_{k} - R_{i}R_{l} v_{k} \right] \|_{L^{p/2,q/2}(S_{T})} \\
		&\leq C\| u_{l} \|_{L^{p,q}(S_{T})}\left(  \| \delta_{il} v_{k}\|_{L^{p,q}(S_{T})} + \| R_{i}R_{l} v_{k} \|_{L^{p,q}(S_{T})} \right) \\
		&\leq C \| u \|_{L^{p,q}(S_{T})}  \|  v\|_{L^{p,q}(S_{T})},
\end{align*}
and therefore $D_{x_{i}}\mathcal{B}_{1}(y,y) \in L^{p/2,q/2}(S_{T})$. This in turn implies $D_{x_{i}} v \in L^{p/2,q/2}(S_{T})$. The same argument yields $D_{x_{i}} B \in L^{p/2,q/2}(S_{T})$.

For $k>1$, we use induction over $k$, assuming that the theorem is true for $k$. Now assume 
\begin{align*}
    D_{x}^{\alpha} D_{t}^{j} f \in L^{p / (|\alpha| + 2j + 1), q/(|\alpha| + 2j + 1)}(S_{T}) \quad \mathrm{for}~ |\alpha| + 2j \leq k+1, ~ p,q > k+2.
\end{align*}
Derivatives with respect to multi-indices $(j,\alpha)$ with $2j + |\alpha| \leq k$ are covered by the induction hypothesis. We thus only need to consider the case $2j + |\alpha| = k + 1$.

\emph{Case 1:} $j=0$. In this case, we apply all but one derivative and see that $D_{x}^{\alpha} \mathcal{B}_{1}(y,y)$ and $D_{x}^{\alpha} \mathcal{B}_{2}(y,y)$ each can be written as a sum of terms of the form $D_{x_{m}} \bar{\mathcal{B}}(D_{x}^{\beta} u_{1} , D_{x}^{\gamma} u_{2} )$, $u_{1}, u_{2} \in \{ v,B \}$ with $|\beta| + |\gamma| = k$. The same reasoning as above, since the (induction) hypothesis implies $1 < \frac{p}{k+2} , \frac{q}{k+2} < \infty$, yields
\begin{align*}
    \| D_{x_{m}} \bar{\mathcal{B}}(D_{x}^{\beta} u_{1} , D_{x}^{\gamma} u_{2} ) \|_{L^{p/(k+2), q/(k+2)}(S_{T})} \leq C \| \langle D_{x}^{\beta} u_{1}, D_{x}^{\gamma} u_{2} \rangle_{\R^{d}} \|_{L^{p/(k+2), q/(k+2)}(S_{T})}.
\end{align*}
If we set $\bar{p} := \frac{p}{|\beta|+1}$, $\bar{q} := \frac{p}{|\gamma|+1}$, we can apply the generalised H\"older inequality with $\bar{r} := \frac{p}{k+2}$ because 
$$
	\frac{1}{\bar{r}} = \frac{k+2}{p} = \frac{|\gamma| + 1}{p} + \frac{|\beta|+1}{p} = \frac{1}{\bar{p}} + \frac{1}{\bar{q}}.
$$
This implies that
\begin{align*}
	\| D_{x_{m}} \bar{\mathcal{B}}(D_{x}^{\beta} u_{1} , D_{x}^{\gamma} u_{2} ) \|_{L^{p/(k+2), q/(k+2)}(S_{T})} \leq C \| D_{x}^{\beta} u_{1} \|_{L^{\frac{p}{|\beta|+1}}} \| D_{x}^{\gamma} u_{2} \|_{L^{\frac{p}{|\gamma|+1}}},
\end{align*}
which is finite by the induction hypothesis.

\emph{Case 2:} $j>0$. In this case, we place all the spatial derivatives on the functions $u_{1}, u_{2}$, so we can write
\begin{equation}\label{AppA_eq_regul_pf_derivative_of_Bbar}
	D_{t}^{j} D_{x}^{\alpha} \bar{\mathcal{B}}(u_{1},u_{2}) = \sum_{|\beta| + |\gamma| = |\alpha|} C_{\beta,\gamma} D_{t}^{j} \bar{\mathcal{B}}(D_{x}^{\beta} u_{1}, D_{x}^{\gamma} u_{2}).
\end{equation}
Since by definition and integration by parts we have
\begin{align*}
	\bar{\mathcal{B}}(u,w) &= \int_{0}^{t} \int_{\mathbb{R}^{d}} \langle u(y,s), \nabla E(x-y,t-s) \rangle \cdot w(y,s) \diff y \diff s \\
	&= \int_{\mathbb{R}} \int_{\mathbb{R}^{d}}   1_{[0,t]}(t-s) D_{x_{k}} E_{ij}(x-y,t-s) u_{j}(y,s) w_{k}(y,s) \diff y \diff s \\
	&= - \int_{\mathbb{R}} \int_{\mathbb{R}^{d}}     E_{ij}(x-y,s) 1_{[0,t]}(t-s) D_{x_{k}} 
	\left[ u_{j}(y,t-s) w_{k}(y,t-s) \right] \diff y \diff s,
\end{align*}
by applying $D_{t}$ we get two kinds of terms from the product rule:
\begin{enumerate}[label=(\roman*), ref=(\roman*)]
	\item If the derivative hits the indicator function, we (formally) get terms of the form 
	\begin{align*}
		&\int_{\mathbb{R}} \int_{\mathbb{R}^{d}} \delta_{\{ 0 \}}(t-s)E_{ij}(x-y,t-s)  D_{x_{k}} \left[u_{j}(y,s) w_{k}(y,s) \right] \diff y \diff s \\
		&- \int_{\mathbb{R}} \int_{\mathbb{R}^{d}} \delta_{\{t\}}(t-s)E_{ij}(x-y,t-s)  D_{x_{k}} \left[u_{j}(y,s) w_{k}(y,s) \right] \diff y \diff s \\
		&= \int_{\mathbb{R}^{d}} E_{ij}(x-y,0) D_{x_{k}} \left[ u_{j}(y,t) w_{k}(y,t) \right] \diff y  \\
		&\quad - \int_{\mathbb{R}^{d}} E_{ij}(x-y,t) D_{x_{k}} \left[ u_{j}(y,t) w_{k}(y,t) \right] \diff y \\
		&=  D_{x_{k}} \left[ u_{j}(x,t) w_{k}(x,t)\right] - \int_{\mathbb{R}^{d}} E_{ij}(x-y,t) D_{x_{k}} \left[ u_{j}(y,t) w_{k}(y,t) \right] \diff y,
	\end{align*}
	i.e., in the first term both integrals shrink to a point due to the delta functions and we are left with one spatial derivative of $u$ and $w$. In the second term, we are left with an unproblematic spatial integral.
	\item If the derivative operator hits the function $E_{ij}$, we use the definition of $E_{ij}$ to find $D_{t} E_{ij} = \Delta E_{ij}$. So each temporal derivative is transformed into two spatial derivatives. We can then use integration by parts again to transfer all but one of these (spatial) derivatives to $u$ and $w$, so this term becomes proportional to
	\begin{align*}
		 D_{x_{m}} \bar{\mathcal{B}}(D_{x}^{\beta'}u, D_{x}^{\gamma'}w),
	\end{align*}
	where $|\beta'| + |\gamma'| = 1$.
\end{enumerate}
Proceeding inductively, we see that if we apply $D_{t}$ for $j>1$ times, we have two cases:
\begin{enumerate}[label=(\roman*), ref=(\roman*)]
	\item The time derivative hits the indicator function at least once. In this case we get a term
	\begin{align*}
		D_{t}^{r-s}D_{x}^{\nu}u(x,t) (D_{t}^{s}D_{x}^{\eta}w)(x,t),
	\end{align*}
	where $s \leq r$ and $|\nu| + |\eta| + 2r = 2(j-1)+1 = 2j - 1$. Here we get the factor $j-1$ because we ``lose" one time derivative to the $\delta$-distribution, but we get one more spatial derivative (with the scaling factor 1) due to the derivative from $\nabla E$.
	\item All the derivatives hit $\nabla E$. In this case, by continuing as in case 2 above, transferring all but one derivative onto $u$ and $w$, we get a term
	\begin{align*}
		D_{x_{m}} \bar{\mathcal{B}}(D_{x}^{\beta'}u, D_{x}^{\gamma'} w),
	\end{align*}
	where $|\beta'| + |\gamma'| = 2j - 1$.
\end{enumerate}
With regard to \eqref{AppA_eq_regul_pf_derivative_of_Bbar}, we replace $u$ by $D_{x}^{\beta}u$ and $w$ by $D_{x}^{\gamma} w$, $|\beta|+|\gamma| = |\alpha|$ and apply Lemma \ref{AppA_thm_keylemma} to find
\begin{align*}
	\| D_{t}^{j} D_{x}^{\alpha} \bar{B}(u,w) \|_{L^{\frac{p}{k+2}}(S_{T})} \leq C \sum_{|\beta| + |\gamma| + 2r = k} \sum_{s \leq r} \| (D_{t}^{r-s} D_{x}^{\beta} u) (D_{t}^{s} D_{x}^{\gamma} w) \|_{L^{\frac{p}{k+2}}(S_{T})}.
\end{align*}
The summation runs over $|\beta| + |\gamma| + 2r = k$ since in either case we ``lose" one (spatial) derivative. By the inductive hypothesis for the induction over $k$, we have
\begin{align*}
	D_{t}^{r-s}D_{x}^{\beta} u \in L^{\frac{p}{|\beta| + 2r - 2s + 1}}(S_{T}), \quad  D_{t}^{s}D_{x}^{\gamma} w \in L^{\frac{p}{|\gamma| + 2s + 1}}(S_{T}).
\end{align*}
Noting that 
$$
	\frac{|\beta| + 2r - 2s +1}{p} + \frac{|\gamma| + 2s +1}{p} = \frac{k+2}{p},
$$
we apply the generalised H\"older inequality to find
\begin{align*}
	\| D_{t}^{j} D_{x}^{\alpha} \bar{B}(u,w) \|_{L^{\frac{p}{k+2}}(S_{T})} \leq C \sum_{|\beta| + |\gamma| + 2r = k} \sum_{s \leq r} \| D_{t}^{r-s} D_{x}^{\beta} u \|_{L^{\frac{p}{|\beta| + 2r - 2s + 1}}(S_{T})} \|D_{t}^{s} D_{x}^{\gamma} w \|_{L^{\frac{p}{|\gamma| + 2s + 1}}(S_{T})},
\end{align*}
which is finite.
\end{proof}

\section*{Acknowledgements}
Financial support by the German Research Foundation (DFG) through the IRTG 2235 is gratefully acknowledged. The author would further like to thank Michael R\"ockner for helpful discussions, as well as Robert Schippa and Guy Fabrice Foghem Gounoue.
 \bibliographystyle{plain}
\bibliography{refs}

\end{document}